\documentclass{amsart}
\usepackage[foot]{amsaddr}
\makeatletter
\def\BState{\State\hskip-\ALG@thistlm}
\renewcommand{\email}[2][]{%
  \ifx\emails\@empty\relax\else{\g@addto@macro\emails{,\space}}\fi%
  \@ifnotempty{#1}{\g@addto@macro\emails{\textrm{(#1)}\space}}%
  \g@addto@macro\emails{#2}%
}
\makeatother

\usepackage{color}

\usepackage{enumitem}

\usepackage{stmaryrd}
\usepackage{amsthm}
\usepackage{amsmath}
\usepackage{mathtools}
\usepackage{amssymb}
\usepackage{commath}
\usepackage{bigints}
\usepackage{romanbar}
\usepackage{algorithm2e}
\usepackage{enumerate}

\newtheorem{prop}{Proposition}

\newtheorem{remark}{Remark}

\newcommand{\vb}{\mathbf{b}}
\newcommand{\vd}{\mathbf{d}}
\newcommand{\vx}{\mathbf{x}}
\newcommand{\vy}{\mathbf{y}}
\newcommand{\vz}{\mathbf{z}}
\newcommand{\vu}{\mathbf{u}}
\newcommand{\vU}{\mathbf{U}}
\newcommand{\vv}{\mathbf{v}}
\newcommand{\vw}{\mathbf{w}}
\newcommand{\vf}{\mathbf{f}}
\newcommand{\vn}{\mathbf{n}}

\newcommand{\vr}{\mathbf{r}}
\newcommand{\vvarphi}{\boldsymbol{\varphi}}
\newcommand{\vpsi}{\boldsymbol{\psi}}
\newcommand{\vphi}{\boldsymbol{\phi}}
\newcommand{\vnu}{\boldsymbol{\nu}}
\newcommand{\vLambda}{\boldsymbol{\Lambda}}
\newcommand{\veps}{\boldsymbol{\epsilon}}
\newcommand{\vxi}{\boldsymbol{\xi}}
\newcommand{\vzeta}{\boldsymbol{\zeta}}
\newcommand{\bz}{\boldsymbol{0}}
\newcommand{\bx}{\mathbf{x}}
\newcommand{\vF}{\mathbf{F}}
\newcommand{\vL}{\mathbf{L}}
\newcommand{\vY}{\mathbf{Y}}
\newcommand{\vJ}{\mathbf{J}}
\newcommand{\vdY}{\boldsymbol{\delta Y}}
\newcommand{\vo}{\mathbf{0}}
\newcommand{\vmu}{\boldsymbol{\mu}}
\newcommand{\vla}{\boldsymbol{\lambda}}

\newcommand{\eps}{\varepsilon}

\newcommand{\jump}[1]{[ #1 ]} 
\newcommand{\avrg}[1]{\left\{ #1 \right\}} 
\newcommand{\Th}{\mathcal{T}_h} 
\newcommand{\Eh}{\mathcal{E}_h} 
\newcommand{\E}{\mathcal{E}} 
\newcommand{\K}{T} 
\newcommand{\V}{\mathbb{V}} 
\newcommand{\h}{ {\rm h}} 
\newcommand{\Gh}{\Gamma_h}
\newcommand{\Id}{I} 
\newcommand{\I}{\normalfont{\Romanbar{1}}} 
\newcommand{\II}{\normalfont{\Romanbar{2}}} 
 
\newcommand{\g}{g} 
\newcommand{\tr}{{\rm tr}} 
\newcommand{\di}{\mathop{\rm div}\nolimits} 
\newcommand{\supp}{\mathop{\rm supp}\nolimits} 
\newcommand{\A}{\mathbb{A}} 

\def\restriction#1#2{\mathchoice
	{\setbox1\hbox{${\displaystyle #1}_{\scriptstyle #2}$}
		\restrictionaux{#1}{#2}}
	{\setbox1\hbox{${\textstyle #1}_{\scriptstyle #2}$}
		\restrictionaux{#1}{#2}}
	{\setbox1\hbox{${\scriptstyle #1}_{\scriptscriptstyle #2}$}
		\restrictionaux{#1}{#2}}
	{\setbox1\hbox{${\scriptscriptstyle #1}_{\scriptscriptstyle #2}$}
		\restrictionaux{#1}{#2}}}
\def\restrictionaux#1#2{{#1\,\smash{\vrule height .8\ht1 depth .85\dp1}}_{\,#2}}

\definecolor{darkviolet}{rgb}{0.58, 0.0, 0.83}

\title[Large deformations of prestrained plates]{LDG approximation of large deformations \\ of prestrained plates}
\author{Andrea Bonito}
\address[Andrea Bonito]{Department of Mathematics, Texas A\&M University, College Station, TX 77845, USA.}
\email{bonito@tamu.edu}
\author{Diane Guignard}
\address[Diane Guignard]{Department of Mathematics and Statistics
\\ University of Ottawa,
Ottawa, ON K1N 6N5, Canada.}
\email{dguignar@uottawa.ca}
\author{Ricardo H. Nochetto}
\address[Ricardo H. Nochetto]{Department of Mathematics and Institute for Physical Science
		and Technology \\ University of Maryland,
		College Park, Maryland 20742, USA.}
\email{rhn@umd.edu}
\author{Shuo Yang}
\address[Shuo Yang]{Department of Mathematics \\ University of Maryland,
		College Park, Maryland 20742, USA.}
\email{shuoyang@umd.edu}

\date{\today}

\begin{document}
	
\maketitle

\begin{abstract} 
A reduced model for large deformations of prestrained plates consists of minimizing a second order bending energy subject to a nonconvex metric constraint. The former involves the second fundamental form of the middle plate and the later is a restriction on its first fundamental form. We discuss a formal derivation of this reduced model along with an equivalent formulation that makes it amenable computationally. We propose a local discontinuous Galerkin (LDG) finite element approach that hinges on the notion of reconstructed Hessian. We design discrete gradient flows to minimize the ensuing nonconvex problem and to find a suitable initial deformation. We present several insightful numerical experiments, some of practical interest, and assess various computational aspects of the approximation process.
\end{abstract}


\section{Introduction}

Natural and manufactured phenomena abound where thin materials develop internal stresses, deform out of plane and exhibit nontrivial 3d shapes. Nematic glasses \cite{modes2010disclination,modes2010gaussian}, natural growth of soft tissues \cite{goriely2005differential,yavari2010geometric} and manufactured polymer gels \cite{kim2012thermally,klein2007shaping,wu2013three} are chief examples. Such incompatible prestrained materials may be key constituents of micro-mechanical devices and be subject to actuation. A model postulates that these plates may reduce internal stresses by undergoing large out of plane deformations $\vu$ as a means to minimize an elastic energy $E[\vu]$ that measures the discrepancy between a reference (or target) metric $G$ and the orientation preserving realization $\vu$ of it. The strain tensor $\veps_{G}(\nabla\vu)$, given by  
\begin{equation}\label{eqn:3Dprestrain}
  \veps_{G}(\nabla\vu) := \frac12 \big( \nabla\vu^T \nabla\vu - G  \big),
\end{equation}
measures such discrepancy and yields the following elastic energy functional for prestrained isotropic materials in a 3d reference body $\mathcal{B}$ and without external forcing
\begin{equation}\label{E:prestrain-energy}
E[\vu] := \int_{\mathcal{B}}  \mu \Big| G^{-1/2} \veps_{G}(\nabla\vu) G^{-1/2} \Big|^2+
\frac{\lambda}{2}\tr\Big( G^{-1/2} \veps_{G}(\nabla\vu) G^{-1/2} \Big)^2,
\end{equation}
where $\mu,\lambda$ are the Lam\'e constants \cite{efrati2009,efrati2011hyperbolic,sharon2010mechanics}. A deformation $\vu:\mathcal{B}\to\mathbb{R}^3$ such that $\veps_{G}(\nabla\vu)=\mathbf{0}$ is called isometric immersion. If such a map exists, then the material can attain a stress-free equilibrium configuration, i.e., $E[\vu]=0$. However, the existence of an isometric immersion $\vu$ of class $H^2(\mathcal{B})$ for any given smooth metric $G$ is not guaranteed in general and it constitutes an outstanding problem in differential geometry. In the absence of such a map, the infimum of $E[\vu]$ is strictly positive and the material has a residual stress at free equilibria.

Slender elastic bodies are of special interest in many applications and our main focus. In this case, the 3d domain $\mathcal{B}$ can be viewed as a tensor product of a 2d domain $\Omega$, the midplane, and an interval of length $s$, namely
$
\Omega \times (-\frac{s}{2},\frac{s}{2}).
$
Developing dimensionally-reduced models as $s\to0$ is a classical endeavor in nonlinear elasticity. Upon rescaling $E[\vu]$ with a factor of the form $s^{-\beta}$, several $2d$ models can be derived in the limit $s\to0$. A geometrically nonlinear reduced energy was obtained formally by Kirchhoff in his seminal work of 1850. An ansatz-free rigorous derivation for isotropic materials was carried out in the influential work of Friesecke, James and M\"uller in 2002 \cite{frie2002b} via $\Gamma$-convergence for $\beta=3$. This corresponds to the bending regime of the nonlinear Kirchhoff plate theory.

If the target metric $G$ is the identity matrix, there is no in-plane stretching and shearing of the material leaving bending as the chief mechanism of deformation;
an excellent example examined in \cite{frie2002b} is the bending of a sheet of paper.
For a generic metric $G$ that does not depend on $s$ and is uniform across the thickness, Efrati, Sharon and Kupferman derived a 2d energy which decomposes into stretching and bending components \cite{efrati2009}; the former scales linearly in $s$ whereas the latter does it cubically. The first fundamental form of the midplane characterizes stretching while the second fundamental form accounts for bending. The thickness parameter $s$ appears in the reduced energy and determines the relative weight between stretching and bending.

The asymptotic limit $s\to0$ requires a choice of scaling exponent $\beta$. The bending regime $\beta=3$ has been studied by Lewicka and collaborators \cite{lewicka2011,lewicka2016}, while \cite{frie2006,bella2014metric,lewicka2010foppl,lewicka2015variational} discussed other exponents $\beta$. For instance, $\beta=5$ corresponds to the F\"oppl von K\'{a}rman plate theory, which is suitable for moderate deformations. Different energy scalings select specific asymptotic relations between the prestrain metric $G$ and deformations $\vu$. For instance, for $\beta=3$ and metrics $G$ of the form
\begin{equation}\label{E:prestrain-metric}
G(\vx',x_3)=G(\vx')=
\begin{bmatrix}
g(\vx') & \mathbf{0} \\
\mathbf{0} & 1
\end{bmatrix}
\quad \forall \, \vx'\in \Omega, \, x_3\in(-s/2,s/2),
\end{equation}
with $g\in\mathbb{R}^{2\times2}$ symmetric uniformly positive definite, the first fundamental form $\I[\vy]$ of parametrizations $\vy:\Omega\to\mathbb{R}^3$ of the midplane must satisfy the following pointwise metric constraint as $s\to0$
\begin{equation}\label{E:metric-constraint}
\I[\vy](\vx') = g(\vx') \quad \forall \, \vx'\in \Omega;
\end{equation}
this account for the stretching and shearing of the midplane.
Moreover, the scaled elastic energy $s^{-3} E[\vu]$ turns out to $\Gamma$-converge to the reduced bending energy
\begin{equation}\label{E:reduced-bending}
  E[\vy] = \frac{\mu}{12}\int_{\Omega}\Big|g^{-\frac{1}{2}} \, \II[\vy] \, g^{-\frac{1}{2}}\Big|^2+\frac{\lambda}{2\mu+\lambda}\tr\Big(g^{-\frac{1}{2}}\, \II[\vy] \, g^{-\frac{1}{2}} \Big)^2,
\end{equation}
which depends solely on the second fundamental form $\II[\vy]$ of $\vy$
in the absence of external forcing \cite{lewicka2011,lewicka2016}. It is known that
$E[\vy]>0$ provided that the Gaussian curvature of the surface $\vy(\Omega)$ does not vanish identically
\cite{lewicka2011,lewicka2016}. We illustrate this in Figure \ref{F:efrati}.
\begin{figure}[htbp]
	\begin{center}
		\includegraphics[width=5.5cm,trim=300 300 300 300, clip]{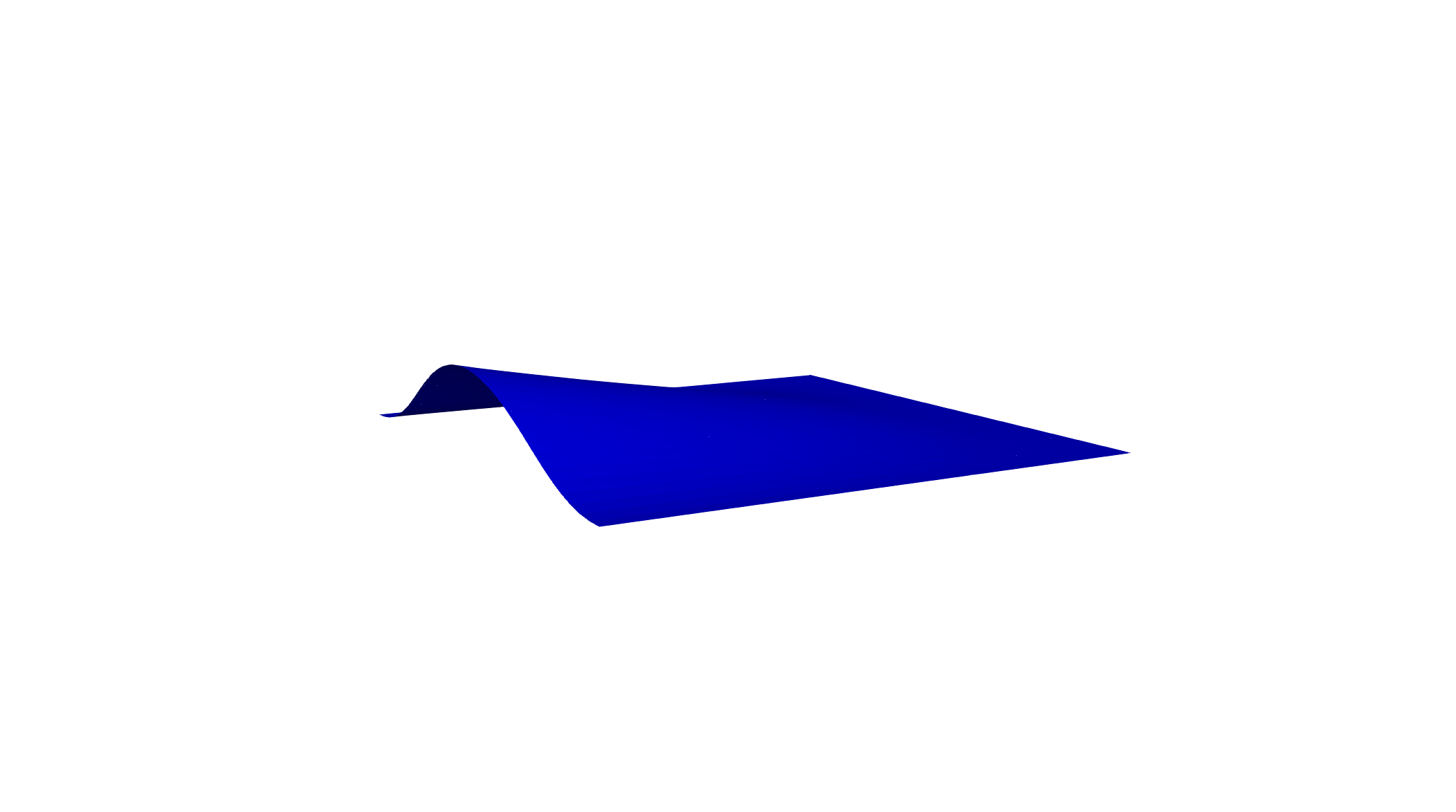}
		\includegraphics[width=5.5cm,trim=300 300 300 300, clip]{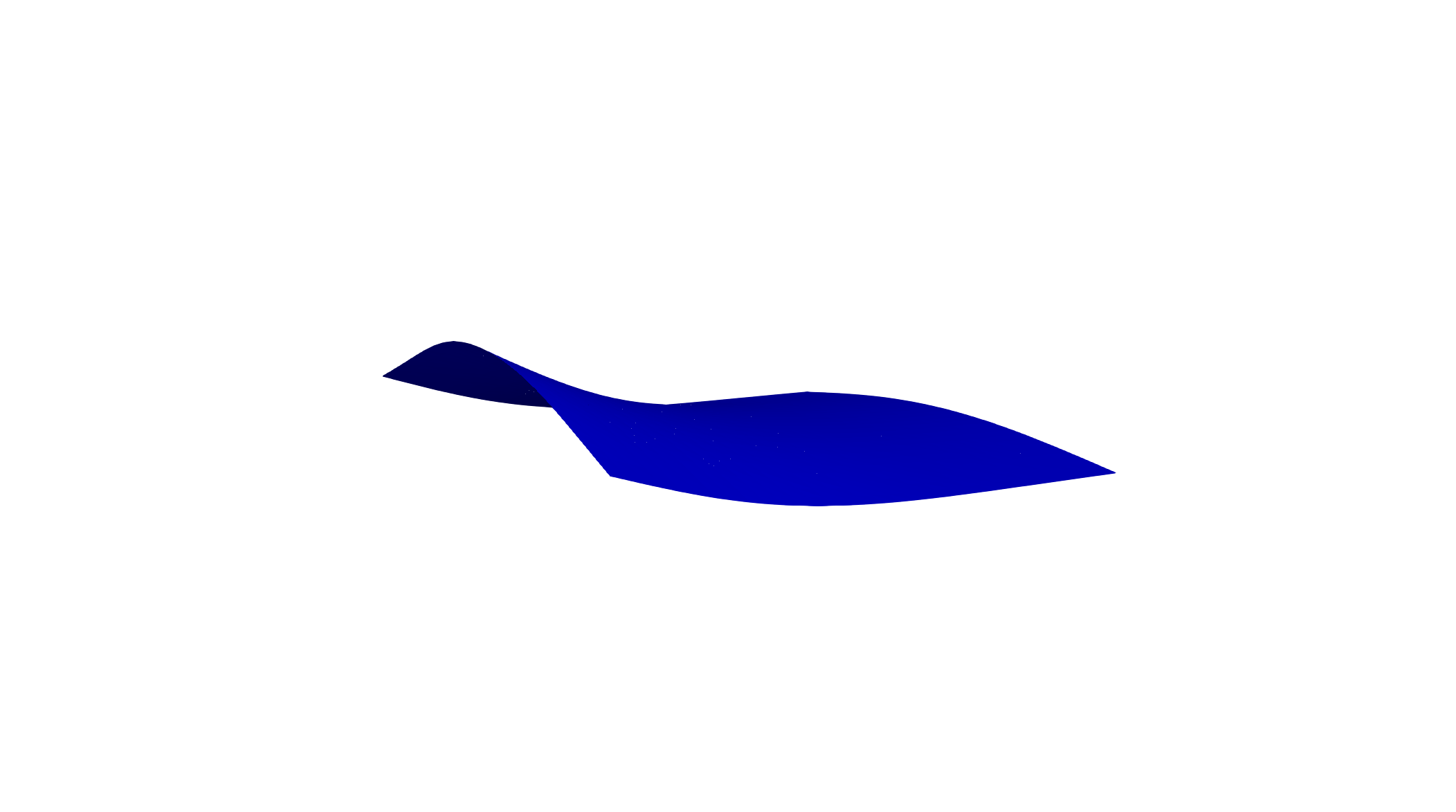}
		\label{F:efrati}
	\end{center}
        \caption{\small A trapezoidal-like plate is glued at the three edges of a square of unit size and is free at the remaining side, as suggested in \cite{efrati2009} (left). For $s$ small, the plate cannot sustain the in-plane compression and buckles up. The deformation $\vy(x_1,x_2)=(x_1,x_2,x_1^2(1-x_1)^2x_2^2)^T$ mimics this configuration and yields a target metric $g=\I[\vy]$;  $\vy$ and $g$ are thus compatible. Upon freeing the boundary conditions, the plate changes shape to a non-flat configuration with the same metric (right).}
\end{figure}

In this article, we present a numerical study of the minimization of \eqref{E:reduced-bending} subject to the constraint \eqref{E:metric-constraint} with either Dirichlet or \emph{free} boundary conditions. We start in Section \ref{sec:prob_statement} with a justification of \eqref{E:prestrain-energy} followed by a formal derivation of \eqref{E:metric-constraint} and \eqref{E:reduced-bending} as the asymptotic limit of $s^{-3} E[\vu]$ as $s\to0$. Moreover, we show an equivalent formulation that basically replaces the second fundamental form $\II[\vy]$ by the Hessian $D^2\vy$, which makes the constrained minimization problem amenable to computation. This derivation is, however, trickier than that in \cite{bartels2013,bonito2018} for single layer plates and \cite{bonito2015,bonito2017,bonito2020discontinuous} for bilayer plates.

The numerical treatment of the ensuing fourth order problem is a challenging and exciting endeavor. In \cite{bartels2013,bonito2015,bonito2017}, the discretization hinges on Kirchhoff elements for isometries $\vy$, i.e., $g=I_2\in\mathbb{R}^{2\times2}$ is the identity matrix. The approximation of $\vy$ in \cite{bonito2018,bonito2020discontinuous} relies on a discontinuous Galerkin (dG) method. In all cases, the minimization problem associated with the nonconvex contraint \eqref{E:metric-constraint} resorts to a discrete $H^2$-gradient flow approach. In addition, a $\Gamma$-convergence theory is developed in \cite{bartels2013,bonito2015,bonito2018}. In this paper, inspired by \cite{cockburn1998}, we design a {\it local discontinuous Galerkin} method (LDG) for $g\ne I_2$ that replaces $D^2\vy$ by a reconstructed Hessian $H_h[\vy_h]$ of the discontinuous piecewise polynomial approximation $\vy_h$ of $\vy$. Such discrete Hessian $H_h[\vy_h]$ consists of three distinct parts: the broken Hessian $D_h^2\vy_h$, the lifting $R_h([\nabla_h\vy_h])$ of the jump of the broken gradient $\nabla_h\vy_h$ of $\vy_h$, and the lifting $B_h([\vy_h])$ of the jumps of $\vy_h$ itself. Lifting operators were introduced in \cite{bassi1997} and analyzed in \cite{brezzi1999,brezzi2000}. The definition of $R_h$ and $B_h$ is motivated by the liftings of \cite{ern2010,ern2011} leading to discrete gradient operators.

It is worth pointing out prior uses of $H_h[\vy_h]$. Discrete Hessians were instrumental to study convergence of dG for the bi-Laplacian in \cite{pryer} and plates with isometry constraint in \cite{bonito2018}. In the present contribution, $H_h[\vy_h]$ makes its debut as a chief constituent of the numerical method. We introduce $H_h[\vy_h]$ in Section \ref{sec:method} along with the LDG approximation of \eqref{E:reduced-bending} and the metric defect $D_h[\vy_h]$ that relaxes \eqref{E:metric-constraint} and makes it computable. We also discuss two discrete $H^2$-gradient flows, one to reduce the bending energy \eqref{E:reduced-bending} starting from $\vy_h^0$, and the other to diminish the stretching energy and make $D_h[\vy_h^0]$ as small as possible. The former leads to Algorithm \ref{algo:main_GF} (gradient flow) and the latter to Algorithm \ref{algo:preprocess} (initialization) of Section \ref{sec:method}. We reserve Section~\ref{sec:solve} for implementation aspects of Algorithms \ref{algo:main_GF} and \ref{algo:preprocess}. We present several numerical experiments, some of practical interest, in Section~\ref{sec:num_res} that document the performance of the LDG approach and illustrate the rich variety of shapes achievable with the reduced model \eqref{E:metric-constraint}-\eqref{E:reduced-bending}. We close the paper with concluding remarks in Section \ref{S:conclusions}.

\section{Problem statement} \label{sec:prob_statement}

Let $\Omega_s:=\Omega\times(-s/2,s/2)\subset\mathbb{R}^3$ be a three-dimensional plate at rest, where $s>0$ denotes the thickness and $\Omega\subset\mathbb{R}^2$ is the (flat) midplane. Given a Riemannian metric $G:\Omega_s\rightarrow\mathbb{R}^{3\times 3}$ (symmetric uniformly positive definite matrix), we consider 3d deformations $\vu:\Omega_s\rightarrow\mathbb{R}^3$ driven by the strain tensor $\veps_{G}(\nabla\vu)$ of \eqref{eqn:3Dprestrain} that measures the discrepancy between $\nabla\vu^T\nabla\vu$ and $G$; hence, the 3d elastic energy $E[\vu]=0$  whenever $\veps_{G}(\nabla\vu)=\mathbf{0}$. We say that $G$ is the {\it reference (prestrained or target) metric.} An orientable deformation $\vu:\Omega_s\to\mathbb{R}^3$ of class $H^2(\Omega)$ satisfying $\veps_{G}(\nabla\vu)=\mathbf{0}$ is called an {\it isometric immersion.} We assume that $G$ does not depend on $s$ and is uniform throughout the thickness, as written in \eqref{E:prestrain-metric}
with $g:\Omega\rightarrow\mathbb{R}^{2\times 2}$ symmetric uniformly positive definite \cite{lewicka2011,efrati2009}. If $g^{1/2}$ denotes the square root of $g$, we have
\begin{equation} \label{def:Ginv}
G^{\frac{1}{2}}=
\begin{bmatrix}
g^{\frac{1}{2}} & \mathbf{0} \\
\mathbf{0} & 1
\end{bmatrix},
\quad
G^{-\frac{1}{2}}=
\begin{bmatrix}
g^{-\frac{1}{2}} & \mathbf{0} \\
\mathbf{0} & 1
\end{bmatrix}.
\end{equation}

In Section \ref{S:energy} we rederive, following \cite{efrati2009}, the elastic energy $E[\vu]$ advocated in \cite{lewicka2011,lewicka2016}. We reduce the 3d model to a 2d plate model in Section \ref{S:reduced-model}. To this end, we perform a formal asymptotic analysis as $s\to0$ but also consider the pre-asymptotic regime $s>0$. We discuss the notion of admissibility in Section \ref{S:admissibility} and derive an equivalent reduced energy better suited for computation in Section \ref{S:simplified-energy}.

We will use the following notation below. The $i^{th}$ component of a vector $\vv\in\mathbb{R}^n$ is denoted $v_i$ while for a matrix $A\in\mathbb{R}^{n\times m}$, we write $A_{ij}$ the coefficient of the $i^{th}$ row and $j^{th}$ column. 
The gradient of a scalar function is a column vector and for $\mathbf v:\mathbb R^m \rightarrow \mathbb R^n$, we set
$(\nabla \mathbf v)_{ij} := \partial_j v_i$, $i=1,..,n$, $j=1,...,m$. The Euclidean norm of a vector is denoted $|\cdot|$. For matrices $A,B\in\mathbb{R}^{n\times m}$, we write $A:B:=\tr(B^TA)=\sum_{i=1}^{n}\sum_{j=1}^{m}A_{ij}B_{ij}$ and $|A|:=\sqrt{A:A}$ the Frobenius norm of $A$. To have a compact notation later, for higher-order tensors we set
\begin{equation}\label{E:higher-order}
\mathbf{A}=(A_k)_{k=1}^n\in\mathbb{R}^{n\times m\times m}
~ \Rightarrow ~
\tr(\mathbf{A}) = \big( \tr(A_k) \big)_{k=1}^n,
\quad
|\mathbf{A}| = \left(\sum_{k=1}^n |A_k|^2 \right)^{\frac12}.
\end{equation}
Furthermore, we will frequently use the convention 
\begin{equation}\label{E:higher-order_B}
B\mathbf{A}B := (BA_kB)_{k=1}^3   \in \mathbb R^{3 \times 2 \times 2},
\end{equation}
for $\mathbf{A}  \in \mathbb R^{3 \times 2 \times 2}$ and $B \in \mathbb R^{2\times 2}$.
In particular, for $\vy:\mathbb R^2 \rightarrow \mathbb R^3$, we will often write
\begin{equation}\label{e:notation_terms_A}
g^{-1/2} \, D^2 \vy \, g^{-1/2} = \left(g^{-1/2} \, D^2 y_k \, g^{-1/2}\right)_{k=1}^3,
\end{equation}
which, combined with \eqref{E:higher-order}, yields
\begin{equation}\label{e:notation_terms}
\begin{split}
\big|g^{-1/2} \, D^2 \vy \, g^{-1/2}\big| &= \left( \sum_{k=1}^3 \big|g^{-1/2} \, D^2 y_k \, g^{-1/2} \big|^2 \right)^{1/2}, \\
\tr \big(g^{-1/2} \, D^2 \vy \, g^{-1/2} \big) &= \left( \tr \big(g^{-1/2} \, D^2 y_k \, g^{-1/2} \big)\right)_{k=1}^3.
\end{split}
\end{equation}
Finally, $\Id_n$ will denote the identity matrix in $\mathbb{R}^{n\times n}$.

\subsection{Elastic energy for prestrained plates}\label{S:energy}

We present, following \cite{efrati2009}, a simple derivation of the energy density $W(\nabla\vu \, G^{-1})$ for prestrained materials. This hinges on the well-established theory of hyperelasticity, and reduces to the classical St. Venant-Kirchhoff model provided $G=I_3$. Such model for isotropic materials reads
\begin{equation} \label{def:W_iso}
W(F):=\mu |\veps_{\Id}|^2+\frac{\lambda}{2}\tr(\veps_{\Id})^2, \quad \veps_{\Id}(F):=\frac{1}{2}\left(F^TF-\Id_3\right).
\end{equation}
Here, $F$ is the deformation gradient, $\veps_{\Id}$ is the Green-Lagrange strain tensor and $\lambda$ and $\mu$ are the (first and second) Lam\'e constants.  This implies
\begin{equation} \label{def:W2_iso}
D^2W(\Id_3)(F,F)=2\mu |e|^2+\lambda\tr(e)^2, \quad e:=\frac{F+F^T}{2}.
\end{equation}
We point out that in \cite{frie2002b}, the strain tensor $\veps_{\Id}=\veps_{\Id}(F)$ of \eqref{def:W_iso} is set to be $\veps_{\Id}(F)=\sqrt{F^TF}-\Id_3$, which yields the same relation \eqref{def:W2_iso}, and thus the same $\Gamma$-limit discussed below.

Given an arbitrary point $\vx_0\in\Omega_s$, we consider the linear transformation
$\vr_0(\vx) := G^{1/2}(\vx_0) (\vx-\vx_0)$; hence $\nabla\vr_0(\vx) = G^{1/2}(\vx_0)$.
The map $\vr_0$ can be viewed as a local re-parametrization of the deformed 3d elastic
body, and $\vz=\vr_0(\vx)$ is a new local coordinate system.
This induces the deformation $\vU(\vz) := \vu(\vx)$ and
\[
\vu = \vU \circ \vr_0
\quad\Rightarrow\quad
\nabla \vu (\vx) = \nabla_{\vz} \vU(\vz) \, G^{\frac12}(\vx_0),
\]
where $\nabla_{\vz}$ denotes the gradient with respect to the variable $\vz$.
The deviation of $\nabla\vu^T \nabla\vu$ from the reference metric $G$ at $\vx=\vx_0$ is thus given by \eqref{eqn:3Dprestrain}
\[
\veps_{G}(\nabla \vu) = \frac12 \big( \nabla\vu^T \nabla\vu - G \big)
= \frac12 G^{\frac12} \big( \nabla_{\vz}\vU^T \nabla_{\vz}\vU - I_3 \big) G^{\frac12}
= G^{\frac12} \veps_{\Id}(\nabla_{\vz}\vU)  G^{\frac12}.
\]
The energy density $W(\nabla_{\vz}\vU)$
at $\vz=\vr_0(\vx)$ with $\vx=\vx_0$ associated with $\veps_{\Id}(\nabla_{\vz}\vU)$, which minimizes when $\veps_{\Id}(\nabla_{\vz}\vU)$ vanishes, is governed by \eqref{def:W_iso} for isotropic materials according to the theory of hyperelasticity. 
What we need to do now is to rewrite this energy density in terms of $\nabla\vu$ at $\vx=\vx_0$, namely $W(\nabla_{\vz}\vU) = W(\nabla\vu \, G^{-1/2})$, whence
\begin{equation}\label{E:new-energydensity}
W(\nabla\vu \, G^{-1/2}) = \mu \Big| G^{-1/2} \, \veps_{G}(\nabla\vu) \, G^{-1/2} \Big|^2+
\frac{\lambda}{2}\tr\Big( G^{-1/2} \, \veps_{G}(\nabla\vu) \, G^{-1/2} \Big)^2.
\end{equation}
This motivates the definition of hyperelastic energy for prestrained materials
\begin{equation} \label{def:EG}
E[\vu] := \int_{\Omega_s} W \big(\nabla\vu(\vx) G(\vx)^{-\frac{1}{2}}\big)d\vx-\int_{\Omega_s}\vf_s(\vx)\cdot\vu(\vx) d\vx,
\end{equation}
where $\vf_s:\Omega_s \rightarrow \mathbb R^3$ is a prescribed forcing term and $W$ is given by \eqref{E:new-energydensity}.

Note that the pointwise decomposition $G(\vx_0)=\nabla \vr_0(\vx_0)^T\nabla \vr_0(\vx_0)$ is always possible because $G(\vx_0)$ is symmetric positive definite.
However, a global transformation $\vr$ such that $\nabla \vr^T\nabla\vr=G$ everywhere need not exist in general because $G$ is not required to be immersible in $\mathbb{R}^3$. This is referred to as {\it incompatible elasticity} in \cite{efrati2009}. Moreover, the infimum of $E[\vu]$ in \eqref{def:EG} should be strictly positive if the Riemann curvature tensor associated with $G$ does not vanish identically \cite{lewicka2011}.
 
\subsection{Reduced model}\label{S:reduced-model}

It is well-known that the case $E[\vu]\sim s$ corresponds to a stretching of the midplane $\Omega$ (membrane theory) while pure bending occurs when $E[\vu]\sim s^3$ (bending theory); see \cite{frie2006}. We examine now the formal asymptotic behavior of $s^{-3} E[\vu]$ as $s\to0$; see also  \cite{efrati2009}.

We start with the assumption \cite{frie2002a,lewicka2011,lewicka2016}
\begin{equation}\label{E:kirchhoff_quad}
\vu(\vx)=\vy(\vx')+x_3\alpha(\vx')\vnu(\vx')+\frac12 x_3^2\beta(\vx')\vnu(\vx') \qquad\forall \, \vx'\in\Omega, ~ x_3 \in (-s/2,s/2),
\end{equation}
where $\vy:\Omega\rightarrow \mathbb R^3$ describes the deformation of the mid-surface of the plate, $\vnu(\vx'):=\frac{\partial_1\vy(\vx')\times\partial_2\vy(\vx')}{|\partial_1\vy(\vx')\times\partial_2\vy(\vx')|}$ is the unit normal vector to the surface $\vy(\Omega)$ at the point $\vy(\vx')$, and $\alpha,\beta:\Omega\rightarrow\mathbb{R}$ are functions to be determined. Compared to the usual Kirchhoff-Love assumption
\begin{equation}\label{E:kirchhoff}
\vu(\vx',x_3) = \vy(\vx') + x_3 \, \vnu(\vx')
\qquad\forall \, \vx'\in\Omega, ~ x_3 \in (-s/2,s/2),
\end{equation}
\eqref{E:kirchhoff_quad} not only restricts fibers orthogonal to $\Omega$ to remain perpendicular to the surface $\vy(\Omega)$ but also allows such fibers to be inhomogeneously stretched. We rescale the forcing term in \eqref{def:EG} as follows
\begin{equation}\label{E:forcing}
\vf(\vx'):=\lim_{s\to 0^+} s^{-3}\int_{-s/2}^{s/2}\vf_s(\vx',x_3) \, dx_3
\qquad\forall \, \vx'\in\Omega,
\end{equation}
and assume the limit to be finite. However, for the asymptotics below we omit this term for simplicity from the derivation and focus on the energy density $W$ in \eqref{def:EG}.

Denoting by $\nabla'$ the gradient with respect to $\vx'$ and writing $\vb(\vx'):=\alpha(\vx')\vnu(\vx')$ and $\vd(\vx'):=\beta(\vx')\vnu(\vx')$, we have for all $\vx= (\vx',x_3)\in\Omega_s$
\begin{equation*}
\nabla\vu(\vx) = \left[ \nabla'\vy(\vx')+x_3\nabla' \vb(\vx')+\frac12x_3^2\nabla'\vd(\vx'), \vb(\vx')+x_3\vd(\vx')\right]\in\mathbb{R}^{3\times 3}.
\end{equation*}
Using the relations
\begin{equation*}
\vnu^T\vnu=1 \quad \mbox{and} \quad \vnu^T\nabla'\vy=\vnu^T\nabla'\vnu=\vd^T\nabla'\vnu=\vd^T\nabla'\vy=\vb^T\nabla'\vnu=\vb^T\nabla'\vy=\mathbf{0},
\end{equation*}
we easily get
\begin{align*}
  \nabla\vu^T\nabla\vu &=
\begin{bmatrix}
\nabla'\vy^T\nabla'\vy & \mathbf{0} \\
\mathbf{0} & \alpha^2
\end{bmatrix}
+x_3
\begin{bmatrix}
\nabla'\vy^T\nabla'\vb+\nabla'\vb^T\nabla'\vy & \nabla'\vb^T\vb \\
\vb^T\nabla'\vb & 2\alpha\beta
\end{bmatrix} \\
&+x_3^2
\begin{bmatrix}
\frac12(\nabla'\vy^T\nabla'\vd+\nabla'\vd^T\nabla'\vy)+\nabla'\vb^T\nabla'\vb & \frac12\nabla'\vd^T\vb+\nabla'\vb^T\vd \\
\frac12\vb^T\nabla'\vd+\vd^T\nabla'\vb & \beta^2
\end{bmatrix}+ h.o.t.
\end{align*}
Moreover, since
\begin{equation*}
|\vnu|^2=1, \quad \partial_j\vb=(\partial_j\alpha)\vnu+\alpha\partial_j\vnu \quad \mbox{and} \quad \vnu\cdot\partial_j\vy=0 \quad \mbox{for } j=1,2,
\end{equation*}
we have
\begin{equation*}
\nabla'\vb^T\nabla'\vy=\alpha\nabla'\vnu^T\nabla'\vy \quad \mbox{and} \quad \nabla'\vb^T\vb=\alpha\nabla'\alpha.
\end{equation*}
Therefore, the expression $2G^{-1/2} \veps_{G}(\nabla\vu) G^{-1/2}$ becomes
\begin{equation*}
  G^{-\frac{1}{2}}\nabla\vu^T\nabla\vu G^{-\frac{1}{2}}-\Id_3 =
  A_1+2x_3A_2+x_3^2A_3 + \mathcal{O}(x_3^3),
\end{equation*}
where
\begin{align*}
A_1 \! &:= \! \begin{bmatrix}
g^{-\frac{1}{2}} \, \I[\vy] \, g^{-\frac{1}{2}}-\Id_2 & \mathbf{0} \\
\mathbf{0} & \alpha^2-1
\end{bmatrix},
\\
A_2 \! &:= \! \begin{bmatrix}
- \alpha g^{-\frac{1}{2}} \, \II[\vy] \, g^{-\frac{1}{2}} & \frac12 \alpha g^{-\frac12}\nabla'\alpha \\
\frac12\alpha \nabla'\alpha^Tg^{-\frac12} & \alpha\beta
\end{bmatrix},
\\
A_3 \! &:= \! \begin{bmatrix}
g^{-\frac12}(\nabla'\vb^T\nabla'\vb+\frac12(\nabla'\vy^T\nabla'\vd+\nabla'\vd^T\nabla'\vy))g^{-\frac12} & \frac12g^{-\frac12}(\nabla'\vd^T\vb+2\nabla'\vb^T\vd) \\
\frac12(\nabla'\vd^T\vb+2\nabla'\vb^T\vd)^Tg^{-\frac12} & \beta^2
\end{bmatrix}
\end{align*}
are independent of $x_3$ and
\[
\I[\vy] = \nabla'\vy^T\nabla'\vy
\quad \mbox{and} \quad
\II[\vy] = -\nabla'\vnu^T\nabla'\vy
\]
are the first and second fundamental forms of $\vy(\Omega)$, respectively. To evaluate the two terms on the right-hand side of \eqref{E:new-energydensity}, we split them into powers of $x_3$. We first deal with the pre-asymptotic regime, in which $s>0$ is small, and next we consider the asymptotic regime $s\to0$.

\medskip\noindent
{\bf Pre-asymptotics.} To compute $s^{-3}\int_{\Omega_s} \big| G^{-\frac12} \veps_{G}(\nabla\vu) G^{-\frac12}\big|^2$, we first note that
\[
\Big|G^{-\frac12} \veps_{G}(\nabla\vu) G^{-\frac12} \Big|^2
\!= \! \frac14 |A_1|^2  +  x_3 A_1 \!:\! A_2  + \frac{x_3^2}{2} A_1 \!:\! A_3
 + x_3^2|A_2|^2
+  \mathcal{O}(x_3^3),
\]
all the terms with odd powers of $x_3$ integrate to zero on $[-s/2,s/2]$, and those terms hidden in $\mathcal{O}(x_3^3)$ integrate to an $\mathcal{O}(s)$ contribution after rescaling by $s^{-3}$. We next realize that
\begin{align*}
  s^{-3} \int_{-s/2}^{s/2}  dx_3 \int_\Omega |A_1|^2 d\vx'
& = s^{-2} \int_\Omega \big|A_1\big|^2 d\vx'
\\
s^{-3} \int_{-s/2}^{s/2}  x_3^2 \, dx_3 \int_\Omega A_1 \!:\! A_3  d\vx' 
& = \frac{1}{12} \int_\Omega A_1 \!:\! A_3  d\vx'
\\
s^{-3} \int_{-s/2}^{s/2} x_3^2 \,  dx_3 \int_\Omega |A_2|^2 d\vx' 
& = \frac{1}{12} \int_\Omega \big|A_2\big|^2 d\vx',
\end{align*}
and exploit that $s^{-3}\int_{\Omega_s} \big| G^{-\frac12} \veps_{G}(\nabla\vu) G^{-\frac12}|^2\le\Lambda$ independent of $s$ to find that
\[
\Big| \int_\Omega A_1 \!:\! A_3  d\vx'  \Big| \le
s \, \Big( s^{-2} \int_\Omega |A_1|^2d\vx'  \Big)^{\frac12}
\Big(\int_\Omega |A_3|^2d\vx'  \Big)^{\frac12} \le C \Lambda^{\frac12} s
\]
is a higher order term because $\int_\Omega |A_3|^2d\vx'\le C^2$. We thus obtain the expression
\[
s^{-3}\int_{\Omega_s} \big| G^{\frac12} \veps_{G}(\nabla\vu) G^{\frac12}\big|^2
=
\frac{1}{4s^2} \int_\Omega \big|A_1\big|^2 d\vx'
+ \frac{1}{12} \int_\Omega \big|A_2\big|^2  d\vx'
+ \mathcal{O}(s).
\]
We proceed similarly with the second term in \eqref{E:new-energydensity} to arrive at
\begin{align*}
\tr\big( G^{-\frac12} \veps_{G}(\nabla\vu) G^{-\frac12} \big)^2
= &\frac14 \, \tr(A_1)^2 + x_3 \, \tr(A_1) \, \tr(A_2)
+ \frac12 \, x_3^2 \, \tr(A_1) \, \tr(A_3)
\\& + x_3^2 \, \tr(A_2)^2 + \mathcal{O}(x_3^3),
\end{align*}
and
\begin{equation*}
s^{-3}\int_{\Omega_s} \tr\big( G^{-\frac12} \veps_{G}(\nabla\vu) G^{-\frac12} \big)^2 =
\frac{1}{4s^2} \int_\Omega \tr\big(A_1\big)^2 d\vx'
+ \frac{1}{12} \int_\Omega \tr\big(A_2\big)^2 d\vx' + \mathcal{O}(s).
\end{equation*}
In view of \eqref{E:new-energydensity} and \eqref{def:EG}, we deduce that the rescaled elastic energy $s^{-3} E[\vu]\approx  E_s[\vy] + E_b[\vy]$ for $s$ small, where the two leading terms are the {\it stretching energy}
\begin{equation}\label{E:stretching}
E_s[\vy] = \frac{1}{8s^2}
\int_\Omega \Big(2\mu \big|A_1\big|^2
+\lambda \tr\big(A_1\big)^2 \Big) d\vx'
\end{equation}
and the {\it bending energy}
\begin{equation}\label{E:bending}
E_b[\vy] = \frac{1}{24}
\int_{\Omega}\Big(2\mu \big|A_2\big|^2
+\lambda\tr \big(A_2\big)^2\Big) d\vx'
\end{equation}
with $A_1$ and $A_2$ depending on $\I[\vy]$ and $\II[\vy]$, respectively.

\medskip\noindent
{\bf Asymptotics.} We now let the thickness $s\to0$ and observe that for the scaled energy to remain uniformy bounded, the integrant of the stretching energy must vanish with a rate at least $s^2$. By definition of $A_1$, this implies that the parametrization $\vy$ must satisfy the metric constraint $g^{-\frac12} \, \I[\vy] \, g^{-\frac12} = I_2$, or equivalently $\vy$ is an {\it isometric immersion} of $g$
\begin{equation}\label{eqn:2Dprestrain}
\nabla'\vy^T\nabla'\vy = g \quad \mbox{a.e. in } \Omega,
\end{equation}
and $\alpha^2\equiv 1$. Since $E_s[\vy]=0$, we can take the limit for $s\to0$ and neglect the higher order terms
to obtain the following expression for the reduced elastic energy
\begin{equation}\label{E:reduced_w}
\lim_{s\rightarrow 0}\frac{1}{s^3}\int_{\Omega_s}W(\nabla\vu G^{-\frac{1}{2}})d\vx
= \frac{1}{24}\int_{\Omega}\Big(\underbrace{2\mu|A_2|^2+\lambda\tr(A_2)^2}_{=:w(\beta)}\Big)d\vx',
\end{equation}
where, using the definition of $A_2$, $w(\beta)$ is given by 
\begin{equation*}
w(\beta)=2\mu|g^{-\frac{1}{2}} \, \II[\vy] \, g^{-\frac{1}{2}}|^2+2\mu \beta^2+\lambda(-\tr(g^{-\frac{1}{2}} \, \II[\vy] \, g^{-\frac{1}{2}})+\beta)^2
\end{equation*}
because $\alpha^2\equiv 1$. In order to obtain deformations with minimal energies, we now choose $\beta=\beta(\vx')$ such that $w(\beta)$ is minimized. Since
\begin{equation*}
\frac{dw}{d\beta}=4\mu\beta+2\lambda\Big(-\tr\big(g^{-\frac{1}{2}} \, \II[\vy] \, g^{-\frac{1}{2}}\big)+\beta \Big) = 0 \quad \mbox{and} \quad \frac{d^2w}{d\beta^2}=4\mu+2\lambda>0,
\end{equation*}
we get
\begin{equation*}
\beta=\frac{\lambda}{2\mu+\lambda}\tr\big(g^{-\frac{1}{2}} \, \II[\vy] \, g^{-\frac{1}{2}}\big),
\end{equation*}
which gives 
\begin{equation*}
w(\beta)=2\mu \big|g^{-\frac{1}{2}} \, \II[\vy] \, g^{-\frac{1}{2}}\big|^2+\frac{2\mu\lambda}{\lambda+2\mu}\tr\big(g^{-\frac12}\II[\vy]g^{-\frac12}\big)^2.
\end{equation*}
Finally, the right-hand side of \eqref{E:reduced_w} has to be supplemented with the forcing term that we have ignored in this derivation but scales correctly owing to definition \eqref{E:forcing}. In the sequel, we relabel the bending energy $E_b[\vy]$ as $E[\vy]$, add the forcing and replace $\vx'$ by $\vx$ (and drop the notation $'$ on differential operators)
\begin{equation}\label{E:final-bending}
E[\vy] = \frac{\mu}{12}\int_{\Omega}\Big(\big|g^{-\frac{1}{2}}\II[\vy] g^{-\frac{1}{2}}\big|^2+\frac{\lambda}{2\mu+\lambda}\tr\big(g^{-\frac{1}{2}}\II[\vy] g^{-\frac{1}{2}}\big)^2\Big)d\vx
- \int_\Omega \vf\cdot\vy d\vx.
\end{equation}
This formal procedure has been justified via $\Gamma$-convergence in \cite{frie2002a,frie2002b} for isometries $\I[\vy]=I_2$ and in \cite[Corollary 2.7]{lewicka2011}, \cite[Theorem 2.1]{lewicka2016} for isometric immersions $\I[\vy]=g$. Moreover, as already observed in \cite{frie2002b}, we mention that using the Kirchhoff-Love assumption \eqref{E:kirchhoff} instead \eqref{E:kirchhoff_quad} yields a similar bending energy, namely we obtain \eqref{E:final-bending} but with $\lambda$ instead of $\frac{\mu\lambda}{2\mu+\lambda}$.
%

\subsection{Admissibility}\label{S:admissibility}
%
We need to supplement \eqref{E:final-bending} with suitable boundary conditions for $\vy$ for the minimization problem to be well-posed. For simplicity, we consider Dirichlet and free boundary conditions in this paper, but other types of boundary conditions are possible. Let $\Gamma_D \subset \partial \Omega$ be a (possibly empty) open set on which the following Dirichlet boundary conditions are imposed:
\begin{equation}\label{eq:dirichlet}
\vy=\vvarphi \quad \mbox{and} \quad \nabla\vy=\Phi \quad \mbox{on } \Gamma_D,
\end{equation}
where $\vvarphi:\Omega\rightarrow\mathbb{R}^3$ and $\Phi:\Omega\rightarrow\mathbb{R}^{3\times 2}$ are sufficiently smooth and $\Phi$ satisfies the compatibility condition $\Phi^T\Phi=g$ a.e. in $\Omega$. The set of {\it admissible} functions is
\begin{equation} \label{def:admiss}
\A(\vvarphi,\Phi):=\left\{\vy\in \V(\vvarphi,\Phi): \, \nabla\vy^T\nabla\vy=\g \,\, \mbox{a.e. in } \Omega\right\},
\end{equation}
where the affine manifold $\V(\vvarphi,\Phi)$ of $H^2(\Omega)$ is defined by
\begin{equation} \label{def:space_BC}
\V(\vvarphi,\Phi):=\left\{\vy\in [H^2(\Omega)]^3: \, \restriction{\vy}{\Gamma_D}=\vvarphi, \, \restriction{\nabla\vy}{\Gamma_D}=\Phi\right\}.
\end{equation}
Our goal is to obtain
\begin{equation} \label{prob:min_Eg}
\vy^*:=\textrm{argmin}_{\vy\in\A(\vvarphi,\Phi)} E(\vy),
\end{equation}
but this minimization problem is highly nonlinear and seems to be out of reach both analytically and geometrically. In fact, whether or not there exists a smooth {\it global} deformation $\vy$ from $\Omega\subset\mathbb{R}^n$ into $\mathbb{R}^N$ satisfying the metric constraint \eqref{eqn:2Dprestrain}, a so-called {\it isometric immersion}, is a long standing problem in differential geometry \cite{han2006isometric}. Note that $\nabla\vy$ is full rank if $\vy$ is an isometric immersion; if in addition $\vy$ is injective, then we say that $\vy$ is an {\it isometric embedding.} For $n=2$, Nash's theorem guarantees that an isometric embedding exists for $N=10$ (Nash proved it for $N=17$, while it was further improved to $N=10$ by Gromov \cite{gromov1986}). When $N=3$, as in our context, a given metric $g$ may or may not admit an isometric immersion. Some elliptic and hyperbolic metrics with special assumptions have isometric immersions in $\mathbb{R}^3$ \cite{han2006isometric}.
We assume implicitly below that $\A(\vvarphi,\Phi)$ is non-empty, thus there exists an isometric immersion that satisfies boundary conditions, but now we discuss an illuminating example in polar coordinates \cite{efrati2011hyperbolic, poznyak1995small}.

\medskip\noindent
{\bf Change of variables and polar coordinates.}
If $\vzeta=(\zeta_1,\zeta_2):\widetilde{\Omega}\to\Omega$ is a
change of variables $\vxi\mapsto\bx$
into Cartesian coordinates $\bx=(x_1,x_2)\in\Omega$ and $\vJ(\vxi)$ is the Jacobian matrix, then
the target metrics $\widetilde{g}(\vxi)$ and $g(\bx)=g(\vzeta(\vxi))$ satisfy
\begin{equation}\label{E:Jacobian}
\widetilde{g}(\vxi) = \vJ(\vxi)^T g(\vzeta(\vxi)) \vJ(\vxi),
\quad
\vJ(\vxi) =
\begin{bmatrix}
  \partial_{\xi_1} \zeta_1(\vxi) & \partial_{\xi_2} \zeta_1(\vxi)
  \\
  \partial_{\xi_1} \zeta_2(\vxi) & \partial_{\xi_2} \zeta_2(\vxi)
\end{bmatrix}.
\end{equation}
Let $\vxi=(r,\theta)$ indicate polar coordinates with $r\in I=[0,R]$ and $\theta\in[0,2\pi)$. If $g=I_2$ is the identity matrix (i.e., $\I[\vy]=I_2$) and $\eta(r)=r$, then
$\widetilde{g}(\vxi)$ reads
\begin{equation}\label{E:metric-eta}
  \widetilde{g}(r,\theta) =
  \begin{bmatrix}
    1 & 0 \\ 0 & \eta(r)^2
  \end{bmatrix}.
\end{equation}

We now show that some metrics of the form of \eqref{E:metric-eta} with $\eta(r)\ne r$ are still isometric immersible provided $\eta$ is sufficiently smooth. Consider the case
$|\eta'(r)|\le1$ along with the parametrization
\begin{equation}\label{E:embedding}
\widetilde{\vy}(r,\theta) = (\eta(r) \cos\theta, \eta(r) \sin\theta, \psi(r) )^T.
\end{equation}
Since $\partial_r\widetilde{\vy}\cdot\partial_\theta\widetilde{\vy}=0$ and $|\partial_\theta\widetilde{\vy}|^2 = \eta(r)^2$,
if $\psi$ satisfies $|\partial_r \widetilde{\vy}|^2 = \eta'(r)^2 + \psi'(r)^2 = 1$, we
realize that $\widetilde{\vy}$ is an isometric
embedding compatible with \eqref{E:metric-eta}. On the other hand, if $|\eta'(r)|\ge1$
and $a \ge \max_{r\in I} |\eta'(r)|$ is an integer, then the parametrization
\begin{equation}\label{E:immersion}
  \widetilde{\vy}(r,\theta) = \Big( \frac{\eta(r)}{a} \cos (a\theta),
  \frac{\eta(r)}{a} \sin (a\theta),
  \int_0^r \sqrt{1-\frac{\eta'(t)^2}{a^2}} dt \Big)^T
\end{equation}
is an isometric immersion compatible with \eqref{E:metric-eta} but not an
isometric embedding. We will construct in Section \ref{S:gel-discs} a couple of
isometric embeddings computationally. 

We also point out that \eqref{E:metric-eta} accounts for {\it shrinking}
if $0\le\eta(r)<r$ and {\it stretching} if $\eta(r)>r$. To see this,
let $\gamma_r(\theta) = (r,\theta)^T$, $\theta \in [0,2\pi)$, be the parametrization of a circle in
$\Omega$ centered at the origin and of radius $r$, and let
$\Gamma_r(\theta) = \widetilde{\vy}(\gamma_r(\theta))$ be its image on $\widetilde{\vy}(\widetilde{\Omega})=\vy(\Omega)$.
The length $\ell(\Gamma_r)$ satisfies
\[
\ell(\Gamma_r) \!=\! \int_0^{2\pi} \Big| \frac{d}{d\theta} \Gamma_r(\theta) \Big| d\theta
=\! \int_0^{2\pi} \sqrt{\gamma_r'(\theta)^T \widetilde{g}(r,\theta) \gamma_r'(\theta)} d\theta
=\! \int_0^{2\pi} \eta(r) d\theta = \ell(\gamma_r) \frac{\eta(r)}{r},
\]
and the ratio $\eta(r)/r$ acts as a shrinking/stretching parameter. 

\medskip\noindent
{\bf Gaussian curvature.}
Since $E[\vy]>0$ provided that the Gaussian curvature $\kappa = \det (\II[\vy])\det (\I[\vy])^{-1}$ of the surface $\vy(\Omega)$ does not vanish identically \cite{lewicka2011,lewicka2016}, it is instructive to find $\kappa$ for a deformation $\widetilde{\vy}$ so that $\I[\widetilde{\vy}]=\widetilde{g}$ is given by \eqref{E:metric-eta}. Since the formula for change of variables for $\II[\widetilde{\vy}]$ is the same as that in \eqref{E:Jacobian} for $\widetilde{g}=\I[\widetilde{\vy}]$, we realize that $\kappa$ is independent of the parametrization of the surface. According to Gauss's Theorema Egregium, $\kappa = \det (\II[\widetilde{\vy}])\det (\I[\widetilde{\vy}])^{-1}$ can be rewritten as an expression solely depending on $\I[\widetilde{\vy}]$. Do Carmo gives an explicit formula for $\kappa$ in case $\widetilde{g}=\I[\widetilde{\vy}]$ is diagonal \cite[Exercise 1, p.237]{carmo1976}, which reduces to
\begin{equation}\label{E:Gauss-curvature}
\kappa = -\frac{\eta''(r)}{\eta(r)}
\end{equation}
for $\widetilde{g}$ of the form \eqref{E:metric-eta}. Alternatively, we may express
$\II[\widetilde{\vy}]_{ij} = \partial_{ij} \widetilde{\vy}\cdot\widetilde{\vnu}$, where $\widetilde{\vnu}(r,\theta)$ is the unit normal vector to the surface $\widetilde{\vy}(\widetilde{\Omega})$ at the point $\widetilde{\vy}(r,\theta)$, in terms of the orthonormal basis
$\{\widetilde{\vnu},\partial_r\widetilde{\vy},\eta(r)^{-1}\partial_\theta\widetilde{\vy}\}$ as follows. First observe that
\begin{align*}
  |\partial_r\widetilde{\vy}|^2=1 
  \quad &\Rightarrow\quad
  \partial_{rr}\widetilde{\vy}\cdot\partial_r\widetilde{\vy}=0,
  \quad
  \partial_{\theta r}\widetilde{\vy} \cdot \partial_r\widetilde{\vy} = 0,
  \\
  |\partial_\theta \widetilde{\vy}|^2 = \eta^2(r)
  \quad &\Rightarrow\quad
  \partial_{r\theta}\widetilde{\vy}\cdot\partial_\theta \widetilde{\vy}=\eta(r)\eta'(r),
  \quad
  \partial_{\theta\theta}\widetilde{\vy} \cdot\partial_\theta\widetilde{\vy}=0,
  \\
  \partial_r\widetilde{\vy}\cdot\partial_\theta\widetilde{\vy} = 0
  \quad &\Rightarrow\quad
  \partial_{rr}\widetilde{\vy} \cdot \partial_\theta \widetilde{\vy} = 0.
\end{align*}
This yields
\[
\partial_{rr}\widetilde{\vy} = (\partial_{rr}\widetilde{\vy}\cdot\widetilde{\vnu}) \widetilde{\vnu},
\quad
\partial_{\theta\theta} \widetilde{\vy} = (\partial_{\theta\theta}\widetilde{\vy}\cdot\widetilde{\vnu}) \widetilde{\vnu}
+ (\partial_{\theta\theta}\widetilde{\vy}\cdot\partial_r\widetilde{\vy})\partial_r\widetilde{\vy},
\]
whence
\[
\II[\widetilde{\vy}]_{rr} \II[\widetilde{\vy}]_{\theta\theta} = (\partial_{rr}\widetilde{\vy}\cdot\widetilde{\vnu}) (\partial_{\theta\theta}\widetilde{\vy}\cdot\widetilde{\vnu}) = \partial_{rr}\widetilde{\vy} \cdot \partial_{\theta\theta} \widetilde{\vy}.
\]
We next differentiate $\partial_{rr}\widetilde{\vy} \cdot \partial_\theta \widetilde{\vy} = 0$ and
$\partial_{r\theta}\widetilde{\vy}\cdot\partial_\theta \widetilde{\vy}=\eta(r)\eta'(r)$
with respect to $\theta$ and $r$, respectively, to obtain
\[
\partial_{rr}\widetilde{\vy}\cdot\partial_{\theta\theta}\widetilde{\vy} =
\partial_{r\theta}\widetilde{\vy} \cdot \partial_{r\theta}\widetilde{\vy} - \eta'(r)^2 - \eta(r)\eta''(r).
\]
We finally notice that $\partial_{r\theta}\widetilde{\vy} = (\partial_{r\theta}\widetilde{\vy}\cdot\widetilde{\vnu})\widetilde{\vnu}
+ \frac{\eta'(r)}{\eta(r)} \partial_\theta \widetilde{\vy}$, whence
\[
\big(\II[\widetilde{\vy}]_{r\theta}\big)^2 = (\partial_{r\theta}\widetilde{\vy}\cdot\widetilde{\vnu})^2
= \partial_{r\theta}\widetilde{\vy} \cdot \partial_{r\theta}\widetilde{\vy} - \eta'(r)^2.
\]
Therefore, we have derived
$\det\II[\widetilde{\vy}]= \II[\widetilde{\vy}]_{rr} \II[\widetilde{\vy}]_{\theta\theta} - \big(\II[\widetilde{\vy}]_{r\theta}\big)^2
= -\eta(r)\eta''(r)$ and as $\det\I[\widetilde{\vy}]=\eta(r)^2$, we obtain \eqref{E:Gauss-curvature}. This expression will be essential in Section \ref{S:gel-discs}.

\subsection{Alternative energy}\label{S:simplified-energy}
%
The expression \eqref{E:final-bending} involves the second fundamental form $\II[\vy] = -\nabla \vnu^T \nabla \vy = (\partial_{ij}\vy\cdot\vnu)_{ij=1}^2$ and is too nonlinear to be practically useful. To render \eqref{prob:min_Eg} amenable to computation, we show now that $\II[\vy]$ can be replaced by the Hessian $D^2\vy$ without affecting the minimizers. This is the subject of next proposition, which uses the notation \eqref{e:notation_terms} for $g^{-1/2} D^2\vy g^{-1/2}$.

\begin{prop}[alternative energy] \label{prop:link_II_D2y}
  Let $\vy=(y_k)_{k=1}^3:\Omega\rightarrow\mathbb{R}^3$ be a sufficiently smooth orientable deformation and let $g=\I[\vy]$ and $\II[\vy]$ be the first and second fundamental forms of $\vy(\Omega)$. Then, there exist functions $f_1,f_2:\Omega\rightarrow\mathbb{R}_{\geq 0}$ depending only on $g$ and its derivatives, with precise definitions given in the proof, such that
\begin{equation} \label{eqn:link_part1}
\big|\g^{-\frac{1}{2}} \, D^2 \vy \, \g^{-\frac{1}{2}} \big|^2 = \big|\g^{-\frac{1}{2}} \, \II[\vy] \, \g^{-\frac{1}{2}} \big|^2 + f_1,
\end{equation}
and
\begin{equation} \label{eqn:link_part2}
\big|\tr \big(\g^{-\frac{1}{2}} \, D^2\vy \, \g^{-\frac{1}{2}} \big)\big|^2 = \tr \big(\g^{-\frac{1}{2}} \, \II[\vy] \, \g^{-\frac{1}{2}} \big)^2 + f_2.
\end{equation}	
\end{prop}
\begin{proof}
First of all, because $\vy$ is smooth and orientable, the second derivatives $\partial_{ij}\vy$ of the deformation $\vy$ can be (uniquely) expressed in the basis $\{\partial_1\vy,\partial_2\vy,\vnu\}$ as
\begin{equation} \label{eqn:Christoffel}
\partial_{ij}\vy = \sum_{l=1}^2\Gamma_{ij}^l \, \partial_l\vy+\II_{ij}[\vy] \, \vnu,
\end{equation}
where $\frac{\partial_1 \vy \times \partial_2 \vy}{|\partial_1 \vy \times \partial_2 \vy|}$ is the unit normal and $\Gamma_{ij}^l$ are the so-called Christoffel symbols of $\vy(\Omega)$. Since $\Gamma_{ij}^l$ are intrinsic quantitites, they can be computed in terms of the coefficients $g_{ij}$ of $g$ and their derivatives \cite{carmo1976}; they do not depend explicitly on $\vy$.

We start with the proof of relation (\ref{eqn:link_part1}). To simplify the notation, let us write $a=\g^{-\frac{1}{2}}$. Using (\ref{eqn:Christoffel}) we get
\begin{align*}
(a \, \II[\vy] \, a)_{ij} \, \vnu & =\sum_{m,n=1}^2a_{im} \big(\II_{mn}[\vy] \, \vnu \big) \, a_{nj} \\
	& = \sum_{m,n=1}^2a_{im}(\partial_{mn}\vy)a_{nj} - \sum_{m,n=1}^2a_{im}\left(\sum_{l=1}^2\Gamma_{mn}^l\partial_l\vy\right) a_{nj},
\end{align*}
or equivalently, rearranging the above expression,
\begin{equation*}
(a \,D^2 \vy \, a)_{ij}= (a \,\II[\vy] \, a)_{ij}\vnu + \sum_{m,n=1}^2a_{im}\left(\sum_{l=1}^2\Gamma_{mn}^l\partial_l\vy\right) a_{nj}.
\end{equation*} 
Since the unit vector $\vnu$ is orthogonal to both $\partial_1\vy$ and $\partial_2\vy$, the right-hand side is an $l_2$-orthogonal decomposition. Computing the square of the $l_2$-norms yields
\begin{equation} \label{eqn:part1_ij}
\sum_{k=1}^3(a \, D^2y_k \, a)_{ij}^2 = (a \, \II[\vy] \, a)_{ij}^2 + f_{ij}
\end{equation}
with
\begin{equation*}
f_{ij} := \sum_{l_1,l_2=1}^2g_{l_1l_2}\sum_{m_1,m_2,n_1,n_2=1}^2a_{im_1}a_{im_2}\Gamma_{m_1n_1}^{l_1}\Gamma_{m_2n_2}^{l_2}a_{n_1j}a_{n_2j}.
\end{equation*}
Functions $f_{ij}$ do not depend explicitly on $\vy$ but on $g$ and first derivatives of $g$. Therefore, summing \eqref{eqn:part1_ij} over $i,j$ from $1$ to $2$ gives \eqref{eqn:link_part1} with $f_1:=\sum_{i,j=1}^2f_{ij}$.

The proof of (\ref{eqn:link_part2}) is similar. Since $\tr(a\,\II[\vy] \, a)\,\vnu = \sum_{i=1}^2(a \, \II[\vy] \, a)_{ii} \,\vnu$ it suffices to take $i=j$ and sum over $i$ in the previous derivation to arrive at \eqref{eqn:link_part2} with
\begin{equation*}
f_2:=\sum_{l_1,l_2=1}^2g_{l_1l_2}\sum_{i_1,i_2,m_1,m_2,n_1,n_2=1}^2a_{i_1m_1}a_{i_2m_2}\Gamma_{m_1n_1}^{l_1}\Gamma_{m_2n_2}^{l_2}a_{n_1i_1}a_{n_2i_2}.
\end{equation*}
This completes the proof because $f_2$ does not dependent explicitly on $\vy$.
\end{proof}

\begin{remark}[alternative energy]
As stated, Proposition~\ref{prop:link_II_D2y} is valid for smooth deformations $\vy$ and metric $g$.
It turns out that for $\vy \in [H^2(\Omega)]^3$ and $g \in [H^1(\Omega) \cap L^\infty(\Omega)]^{2\times 2}$, the key relation \eqref{eqn:Christoffel} holds a.e. in $\Omega$ and so does the conclusion of Proposition~\ref{prop:link_II_D2y}.
For the interested reader, we refer to \cite{BGNY2020}.
\end{remark}

Proposition \ref{prop:link_II_D2y} (alternative energy) shows that the solutions of \eqref{prob:min_Eg} with the energy $E[\vy]$ given by \eqref{E:final-bending} are the same as those given by the energy
\begin{equation} \label{def:Eg_D2y}
E(\vy):=\frac{\mu}{12} \int_{\Omega}\left(\Big|\g^{-\frac{1}{2}} \, D^2\vy \, \g^{-\frac{1}{2}}\Big|^2+\frac{\lambda}{2\mu+\lambda} \left|\tr\big(\g^{-\frac{1}{2}} \, D^2\vy \, \g^{-\frac{1}{2}}\big)\right|^2\right)-\int_{\Omega}\vf\cdot\vy.
\end{equation}
The Euler-Lagrange equations characterizing local extrema $\vy\in[H^2(\Omega)]^3$ of \eqref{def:Eg_D2y}
\begin{equation} \label{equ:Gateaux}
\delta E[\vy;\vv]=0 \quad \forall \vv\in[H^2(\Omega)]^3,
\end{equation}
can be written in terms of the first variation of $E[\vy]$ in the direction $\vv$ given by
\begin{equation} \label{def:Diff_E_D2y}
\begin{split}
\delta E[\vy;\vv] :=& \frac{\mu}{6} \int_{\Omega}\big(\g^{-\frac{1}{2}} \, D^2\vy\, \g^{-\frac{1}{2}} \big): \big(\g^{-\frac{1}{2}} \, D^2\vv \, \g^{-\frac{1}{2}} \big) \\
& + \frac{\mu\lambda}{6(2\mu+\lambda)}
 \int_{\Omega}\tr \big(\g^{-\frac{1}{2}} \, D^2\vy \, \g^{-\frac{1}{2}} \big) \cdot
\tr \big(\g^{-\frac{1}{2}} \, D^2\vv \, \g^{-\frac{1}{2}} \big) -\int_{\Omega}\vf\cdot\vv.
\end{split}
\end{equation}
The presence of the trace term in \eqref{def:Diff_E_D2y} makes it problematic to find the governing partial differential equation hidden in \eqref{equ:Gateaux} (strong form).
However, when $\lambda=0$, integration by parts shows that $P_k:=\g^{-1} \, D^2y_k \, \g^{-1}\in\mathbb{R}^{2 \times 2}$ for $k=1,2,3$ satisfies
\begin{equation*}
  \delta E[\vy;\vv] = \frac{\mu}{6} \sum_{k=1}^3 \left( \int_{\Omega}\di\di P_k \,v_k \!-\! \int_{\partial \Omega} \di P_k\cdot\vn v_k +  \int_{\partial \Omega} P_k\vn \cdot\nabla v_k\right)
\!-\! \int_\Omega \vf \cdot \vv,
\end{equation*}
where $\vn$ is the outwards unit normal vector to $\partial\Omega$.
On the other hand, if 
$g=\Id_2$ in which case $\vy$ is an {\it isometry}, then $E[\vy]$ in \eqref{E:final-bending} and \eqref{def:Eg_D2y} are equal and reduce to
\begin{equation}\label{E:g=1}
E[\vy]=\frac{\alpha}{2}\int_{\Omega}|D^2\vy|^2-\int_{\Omega}\vf\cdot\vy, \quad \alpha:=\frac{\mu(\mu+\lambda)}{3(2\mu+\lambda)}
\end{equation}
thanks to the relations for isometries \cite{bartels2013,bonito2015,bonito2018}
\begin{equation}\label{E:bilaplacian}
|\II[\vy]|=|D^2\vy|=|\Delta\vy|=\tr(\II[\vy]).
\end{equation}
The strong form of the Euler-Lagrange equation for a minimizer of \eqref{E:g=1} reads
$\alpha \di\di D^2 \vy = \alpha \Delta^2 \vy = \vf$. This problem has been studied
numerically in \cite{bartels2013,bonito2018}. \looseness=-1

\section{Numerical scheme} \label{sec:method}

We propose here a {\it local discontinuous Galerkin} (LDG) method to approximate the solution of the problem \eqref{prob:min_Eg}. LDG is inspired by, and in fact improves upon, the previous dG methods \cite{bonito2020discontinuous,bonito2018} but they are conceptually different. LDG hinges on the explicit computation of a discrete Hessian $H_h[\vy_h]$ for the discontinuous piecewise polynomial approximation $\vy_h$ of $\vy$, which allows for a direct discretization of $E_h[\vy_h]$ in \eqref{def:Eg_D2y}, including the trace term. We refer to the companion paper \cite{BGNY2020} for a discussion of convergence of discrete global minimizers of $E_h$ towards those of $E$; a salient feature is that the stability of the LDG method is retained even when the penalty parameters are arbitrarily small.

We organize this section as follows. In Section \ref{subsec:LDG} we introduce the finite dimensional space $\V_h^k$ of discontinuous piecewise polynomials of degree $k\ge2$, along with the discrete Hessian $H_h[\vy_h]$. We also discuss the discrete counterparts $E_h$ and $\A_{h,\eps}^k(\vvarphi,\Phi)$ of the energy $E$ and the admissible set $\A(\vvarphi,\Phi)$, respectively. In Section \ref{subsec:GF}, we present a discrete gradient flow to minimize the energy $E_h$. Finally, in Section \ref{subsec:preprocess}, we show how to prepare suitable initial conditions for the gradient flow (preprocessing).

\subsection{LDG-type discretization} \label{subsec:LDG}
From now on, we assume that $\Omega \subset \mathbb R^2$ is a polygonal domain. Let $\{\Th\}_{h>0}$ be a shape-regular but possibly graded elements $\K$, either triangles or quadrilaterals, of diameter $h_{\K} := \textrm{diam}(\K)\leq h$. In order to handle hanging nodes (necessary for graded meshes based on quadrilaterals), we assume that all the elements within each domain of influence have comparable diameters. We refer to Sections 2.2.4 and 6 of Bonito-Nochetto \cite{BoNo} for precise definitions and properties. At this stage, we only point out that sequences of subdivisions made of quadrilaterals with at most one hanging node per side satisfy this assumption.

Let $\Eh=\Eh^0\cup\Eh^b$ denote the set of edges, where $\Eh^0$ stands for the set of interior edges and $\Eh^{b}$ for the set of boundary edges. We assume a compatible representation of the Dirichlet boundary $\Gamma_D$, i.e., if $\Gamma_D \not = \emptyset$ then $\Gamma_D$ is the union of (some) edges in $\Eh^{b}$ for every $h>0$, which we indicate with $\Eh^{D}$; note that $\Gamma_D$ and $\Eh^{D}$ are empty sets when dealing with a problem with free boundary conditions. Let $\Eh^a:=\Eh^0\cup\Eh^{D}$ the set of {\it active edges} on which jumps and averages will be computed. The union of these edges give rise to the corresponding skeletons of $\Th$
\begin{equation}\label{E:skeleton}
  \Gh^0 := \cup \big\{e: e\in\Eh^0 \big\},
    \quad
  \Gh^D := \cup \big\{e: e\in\Eh^D \big\},
     \quad
  \Gh^a := \Gh^0 \cup \Gh^D.
\end{equation}
If $h_e$ is the diameter of $e\in\Eh$, then $\h$ is the piecewise constant mesh function
\begin{equation} \label{eqn:mesh_function}
\h:\E_h \rightarrow \mathbb R_+, \qquad  \restriction{\h}{e}:= h_e \quad  \forall e\in\Eh.
\end{equation}
From now on, we use the notation $(\cdot,\cdot)_{L^2(\Omega)}$ and $(\cdot,\cdot)_{L^2(\Gh^a)}$ to denote the $L^2$ inner products over $\Omega$ and $\Gh^a$, and a similar notation for subsets of  $\Omega$ and $\Gh^a$.

\medskip\noindent
{\bf Broken spaces.}
For an integer $k\geq 0$, we let $\mathbb{P}_k$ (resp. $\mathbb{Q}_k$) be the space of polynomials of total degree at most $k$ (resp. of degree at most $k$ in the each variable).   
The reference unit triangle (resp. square) is denoted by $\widehat\K$ and for $\K\in \mathcal T_h$, we let $F_\K:\widehat\K \rightarrow \K$ be the generic map from the reference element to the physical element. 
When $\mathcal T_h$ is made of triangles the map is affine, i.e., $F_\K \in \mathbb [\mathbb P_1]^2$,  while $F_\K \in [\mathbb Q_1]^2$ when quadrilaterals are used.
  
If $k \ge 2$, the (\textit{broken}) finite element space $\V_h^k$ to approximate each component of the deformation $\vy$ (modulo boundary conditions) reads
\begin{equation} \label{def:Vhk_tri}
  \V_h^k:=\left\{v_h\in L^2(\Omega): \,\, \restriction{v_h}{\K}\circ F_{\K}\in\mathbb{P}_k
  \quad(\textrm{resp. } \mathbb{Q}_k) \quad \forall \K \in\Th \right\}
\end{equation}
if $\Th$ is made of triangles (resp. quadrilaterals).
We define the broken gradient $\nabla_h v_h$ of $v_h\in\V_h^k$ to be the gradient computed elementwise, and use similar notation for other piecewise differential operators such as the broken Hessian $D_h^2 v_h=\nabla_h\nabla_h v_h$.

We now introduce the jump and average operators. To this end, let $\vn_e$ be a unit normal to $e\in\Eh^0$ (the orientation is chosen arbitrarily but is fixed once for all), while for a boundary edge $e\in\Eh^b$, $\vn_e$ is the outward unit normal vector to $\partial\Omega$. For $v_h \in \V_h^k$ and $e \in \E_h^0$, we set
\begin{equation} \label{def:jump}
  \restriction{\jump{v_h}}{e} := v_h^{-}-v_h^+,
  \quad
  \restriction{\jump{\nabla_h v_h}}{e} := \nabla_h v_h^{-} - \nabla_h v_h^+,
  \quad
\end{equation}
where $v_h^{\pm}(\vx):=\lim_{s\rightarrow 0^+}v_h(\vx\pm s\vn_e)$ for $\vx \in e$. We compute the jumps componentwise provided the function $v_h$ is vector or matrix-valued. In what follows, the subindex $e$ is omitted when it is clear from the context.

In order to deal with Dirichlet boundary data $(\vvarphi,\Phi)$ we resort to a Nitsche approach; hence we do not impose essential restrictions on the discrete space $[\V_h^k]^3$. However, to simplify the notation later, it turns out to be convenient to introduce the discrete sets $\V_h^k(\vvarphi,\Phi)$ and $\V_h^k(\bz,\bz)$ which mimic the continuous counterparts $\V(\vvarphi,\Phi)$ and $\V(\bz,\bz)$ but coincide with $[\V_h^k]^3$.
In fact, we say that $\vv_h\in[\V_h^k]^3$ belongs to $\V_h^k(\vvarphi,\Phi)$ provided the boundary jumps of $\vv_h$ are defined to be
\begin{equation}\label{E:bd-jumps}
  [\vv_h]_e := \vv_h - \vvarphi,
  \quad
  [\nabla_h \vv_h]_e := \nabla_h \vv_h - \Phi,
  \quad \forall \, e\in\Eh^D.
\end{equation}
We stress that $\|[\vv_h]\|_{L^2(\Gamma_h^D)} \to 0$ and $\|[\nabla_h \vv_h]\|_{L^2(\Gamma_h^D)} \to 0$ imply $\vv_h\to \vvarphi$ and $\nabla_h \vv_h \to\Phi$ in $L^2(\Gamma_D)$ as $h\to0$; hence the connection between $\V_h^k(g,\Phi)$ and $\V(g,\Phi)$. Therefore, we emphasize again that the sets $[\V_h^k]^3$ and $\V_h^k(\vvarphi,\Phi)$ coincide but the latter carries the notion of boundary jump, namely
\begin{equation}\label{discrete-set}
  \V_h^k (\vvarphi,\Phi) := \Big\{ \vv_h\in [\V_h^k]^3: \
  [\vv_h]_e, \, [\nabla_h \vv_h]_e \text{ given by \eqref{E:bd-jumps} for all } e\in\Eh^D   \Big\}.
\end{equation}
When free boundary conditions are imposed, i.e., $\Gamma_D = \emptyset$, then we do not need to distinguish between $\V^k_h(\vvarphi,\Phi)$ and $[\V_h^k]^3$. However, we keep the notation $\V^k_h(\vvarphi,\Phi)$ in all cases thereby allowing for a uniform presentation.
 
We define the {\it average} of $v_h \in \V_h^k$ across an edge $e\in \Eh$ to be
\begin{equation} \label{def:avrg}
\avrg{v_h}_{e} := 
\left\{\begin{array}{ll}
\frac{1}{2}(v_h^{+}+v_h^{-}) & e\in\Eh^0 \\
v_h^{-} & e\in\Eh^b,
\end{array}\right. 
\end{equation}
and apply this definition componentwise to vector and matrix-valued functions.
As for the jump notation, the subindex $e$ is drop when it is clear from the context.

\medskip\noindent
{\bf Discrete Hessian.}
To approximate the elastic energy \eqref{def:Eg_D2y}, we propose an LDG approach. 
Inspired by \cite{bonito2018,pryer}, the idea is to replace the Hessian $D^2\vy$ by a discrete Hessian $H_h[\vy_h]\in\left[L^2(\Omega)\right]^{3\times 2\times 2}$ to be defined now. To this end, let $l_1,l_2$ be non-negative integers (to be specified later) and consider two {\it local lifting operators} $r_e:[L^2(e)]^2\rightarrow[\V_h^{l_1}]^{2\times 2}$ and $b_e:L^2(e)\rightarrow[\V_h^{l_2}]^{2\times 2}$ defined for $e\in\Eh^a$ by
\begin{gather} 
  r_e(\vphi) \in [\V_h^{l_1}]^{2\times 2}: \,
  \int_{\omega_e}r_e(\vphi):\tau_h = \int_e\avrg{\tau_h}\vn_e\cdot\vphi \quad \forall \tau_h\in [\V_h^{l_1}]^{2\times 2}\label{def:lift_re},
  \\
  b_e(\phi) \in [\V_h^{l_2}]^{2\times 2}: \,
\int_{\omega_e} b_e(\phi):\tau_h = \int_e\avrg{\di \tau_h}\cdot\vn_e\phi \quad \forall \tau_h\in [\V_h^{l_2}]^{2\times 2} \label{def:lift_be}.
\end{gather}
It is clear that $\supp(r_e(\vphi))=\supp(b_e(\phi))=\omega_e$, where $\omega_e$ is the patch associated with $e$ (i.e., the union of two elements sharing $e$ for interior edges $e\in\Eh^0$ or just one single element for boundary edges $e\in\Eh^b$).
We extend $r_e$ and $b_e$ to $[L^2(e)]^{3\times 2}$ and $[L^2(e)]^3$, respectively, by component-wise applications.

The corresponding {\it global lifting operators} are then given by
\begin{equation}\label{E:global-lifting}
  R_h := \sum_{e\in\Eh^a} r_e : [L^2(\Gh^a)]^2 \rightarrow [\V_h^{l_1}]^{2\times 2},
  \quad
  B_h := \sum_{e\in\Eh^a} b_e : L^2(\Gh^a) \rightarrow [\V_h^{l_2}]^{2\times 2}.
\end{equation}
This construction is simpler than that in \cite{bonito2018} for quadrilaterals.
We now define the {\it discrete Hessian operator} $H_h:\V_h^k(\vvarphi,\Phi)\rightarrow\left[L^2(\Omega)\right]^{3\times 2\times 2}$ to be
\begin{equation} \label{def:discrHess}
H_h[\vv_h] := D_h^2 \vv_h -R_h(\jump{\nabla_h\vv_h})+B_h(\jump{\vv_h}).
\end{equation}
For a given polynomial degree $k\ge2$, a natural choice for the degree of the liftings is $l_1=l_2=k-2$ for triangular elements and $l_1=l_2=k$ for quadrilateral elements. However, any nonnegative values for $l_1$ and $l_2$ are suitable. We anticipate that in the numerical experiments presented in Section~\ref{sec:num_res}, we use $l_1=l_2=k$ with $k=2$.

We refer to \cite{BGNY2020} for properties of $H_h[\vy_h]$ but we point out one now to justify its use. Let $\Gamma_D \not = \emptyset$ and data $(\vvarphi,\Phi)$ be sufficiently smooth, and let $\{ \vy_h \}_{h>0} \subset  \V_h^k(\vvarphi,\Phi)$ satisfy
\begin{equation}\label{e:def_triple}
\|\vy_h\|_{H_h^2(\Omega)}^2:=\| D^2_h \vy_h \|_{L^2(\Omega)}^2 + \| \h^{-\frac{1}{2}} \jump{\nabla_h \vy_h }\|_{L^2(\Gh^a)}^2 + \| \h^{-\frac{3}{2}} \jump{\vy_h }\|_{L^2(\Gh^a)}^2 \leq \Lambda
\end{equation}
for a constant $\Lambda$ independent of $h$. If $\vy_h$ converges in $[L^2(\Omega)]^3$ to a function $\vy \in [H^2(\Omega)]^3$, then $H_h[\vy_h]$ converges weakly to $D^2 \vy$ in $[L^2(\Omega)]^{3\times 2 \times 2}$. We also refer to \cite{pryer,BGNY2020} for similar results for the Hessian and to \cite{ern2011} for the gradient operator.

\medskip\noindent  
{\bf Discrete energies.}
We are now ready to introduce the discrete energy on $\V_h^k(\vvarphi,\Phi)$
\begin{equation} \label{def:Eh}
  \begin{aligned}
E_h[\vy_h] := & \frac{\mu}{12} \int_{\Omega} \Big|\g^{-\frac{1}{2}} \, H_h[\vy_{h}] \, \g^{-\frac{1}{2}} \Big|^2 \\
& + \frac{\mu\lambda }{12(2\mu+\lambda)} \int_{\Omega} \Big|\tr\big(\g^{-\frac{1}{2}} \, H_h[\vy_h] \, \g^{-\frac{1}{2}}\big) \Big|^2  \\
	& +\frac{\gamma_1}{2}\|\h^{-\frac{1}{2}}\jump{\nabla_h\vy_h}\|_{L^2(\Gh^a)}^2+\frac{\gamma_0}{2}\|\h^{-\frac{3}{2}}\jump{\vy_h}\|_{L^2(\Gh^a)}^2 -\int_{\Omega}\vf\cdot\vy_h ,
  \end{aligned}
\end{equation}
where $\gamma_0,\gamma_1>0$ are stabilization parameters; recall the notation \eqref{E:higher-order} and \eqref{E:higher-order_B}.
One of the most attractive feature of the LDG method is that $\gamma_0,\gamma_1$ are not required to be sufficiently large as is typical for interior penalty methods \cite{bonito2018}. We refer to Section~\ref{sec:num_res} for numerical  investigations of this property
and to \cite{BGNY2020} for theory.

Note that the Euler-Lagrange equation $\delta E_h[\vy_h;\vv_h]=0$ in the direction $\vv_h$ reads
\begin{equation}\label{E:E-L}
a_h(\vy_h,\vv_h) = F(\vv_h) \quad\forall \, \vv_h\in\V_h^k(\mathbf{0},\mathbf{0}),
\end{equation}
where
\begin{equation} \label{def:form_ah}
\begin{split}
a_h(\vy_h,\vv_h) := & \frac{\mu}{6} \int_{\Omega} \left(\g^{-\frac{1}{2}}H_h[\vy_{h}] \g^{-\frac{1}{2}}\right):\left(\g^{-\frac{1}{2}}H_h[\vv_{h}] \g^{-\frac{1}{2}}\right)  \\
&  +\frac{\mu\lambda }{6(2\mu+\lambda)} \int_{\Omega}\tr\left(\g^{-\frac{1}{2}}H_h[\vy_{h}] \g^{-\frac{1}{2}}\right) \cdot \tr\left(\g^{-\frac{1}{2}}H_h[\vv_{h}] \g^{-\frac{1}{2}}\right) \\
& +\gamma_1 \big(\h^{-1}\jump{\nabla_h\vy_h},\jump{\nabla_h\vv_h} \big)_{L^2(\Gh^a)}+\gamma_0 \big(\h^{-3}\jump{\vy_h},\jump{\vv_h} \big)_{L^2(\Gh^a)},
\end{split}
\end{equation}
and
\begin{equation}  \label{def:form_Fh}
F(\vv_h) := \int_{\Omega}\vf\cdot\vv_h;
\end{equation}
compare with \eqref{equ:Gateaux} and \eqref{def:Diff_E_D2y}.

We reiterate that finding the strong form of  \eqref{def:Diff_E_D2y} is problematic because of the presence of the trace term. Yet, it is a key ingredient in the design of discontinuous Galerkin methods such as the interior penalty method and raises the question how to construct such methods for \eqref{def:Diff_E_D2y}.
The use of reconstructed Hessian in \eqref{def:form_ah} leads to a numerical scheme without resorting to the strong form of the equation.

\medskip\noindent  
{\bf Constraints.}
%
We now discuss how to impose the Dirichlet boundary conditions \eqref{eq:dirichlet} and the metric constraint \eqref{eqn:2Dprestrain} discretely. The former is enforced via the Nitsche approach and thus is not included as a constraint in the discrete admissible set as in \eqref{def:admiss}; this turns out to be advantageous for the analysis of the method \cite{bonito2018}. The latter is too strong to be imposed on a polynomial space. Inspired by \cite{bonito2018}, we define the {\it metric defect} as
\begin{equation} \label{def:PD}
D_h[\vy_h] := \sum_{\K\in\Th}\left|\int_{\K} \left(\nabla\vy_h^T\nabla\vy_h-g\right)\right|
\end{equation}
and, for a positive number $\eps$, we define the {\it discrete admissible set} to be
\begin{equation*}
\A_{h,\eps}^k:=\Big\{\vy_h\in \V^k_h(\vvarphi,\Phi): \quad D_h[\vy_h]\leq \eps\Big\}.
\end{equation*}
Therefore, the discrete minimization problem, discrete counterpart of \eqref{prob:min_Eg}, reads
\begin{equation} \label{prob:min_Eg_h}
\min_{\vy_h\in\A_{h,\eps}^k} E_h[\vy_h].
\end{equation}
Problem \eqref{prob:min_Eg_h} is nonconvex due to the structure of $\A_{h,\eps}^k$. Its solution is non-trivial and is discussed next.

\subsection{Discrete gradient flow}  \label{subsec:GF}
To find a local minimizer $\vy_h$ of $E_h[\vy_h]$ within $\A_{h,\eps}^k$, we design a discrete gradient flow associated with the discrete $H^2$-norm on $\V_h^k(\mathbf{0},0)$
\begin{equation}\label{def:H2metric}
  \begin{aligned}
(\vv_h,\mathbf{w}_h)_{H_h^2(\Omega)} := & \sigma (\vv_h,\mathbf{w}_h)_{L^2(\Omega)} + (D^2_h\vv_h,D^2_h\mathbf{w}_h)_{L^2(\Omega)} \\
& +(\h^{-1}\jump{\nabla_h\vv_h},\jump{\nabla_h\mathbf{w}_h})_{L^2(\Gh^a)}+(\h^{-3}\jump{\vv_h},\jump{\vw_h})_{L^2(\Gh^a)}, 
  \end{aligned}
\end{equation}
where $\sigma=0$ if $\Gamma_D \not = \emptyset$ and $\sigma>0$ if $\Gamma_D = \emptyset$. The latter corresponds to free boundary conditions and guarantees that $(\cdot,\cdot)_{H_h^2(\Omega)}$ is a scalar product \cite{bonito2020discontinuous,BGNY2020}.

Given an initial guess  $\vy_h^0\in\A_{h,\eps}^k$ and a pseudo-time step $\tau>0$, we  compute iteratively  $\vy_h^{n+1}:=\vy_h^{n}+\delta\vy_h^{n+1} \in \V_h^k(\vvarphi,\Phi)$ that minimizes the functional
\begin{equation}\label{gf:minimization}
  \vw_h \,\,\mapsto\,\,\frac{1}{2\tau}\|\vw_h-\vy_h^n\|_{H_h^2(\Omega)}^2+E_h[\vw_h]
  \quad\forall \, \vw_h\in\V_h^k(\vvarphi,\Phi),
\end{equation}
under the following \textit{linearized metric constraint} for the increment $\delta\vy_h^{n+1}$
\begin{equation}\label{gf:constraint}
L_T[\vy_h^n;\delta\vy_h^{n+1}] :=
\int_T(\nabla\delta\vy_h^{n+1})^T\nabla\vy_h^n+(\nabla\vy_h^n)^T\nabla\delta\vy_h^{n+1}=0 \quad \forall \K\in\Th.
\end{equation}
 
The proposed strategy is summarized in Algorithm \ref{algo:main_GF}.
\RestyleAlgo{boxruled}
\begin{algorithm}[htbp]
	\SetAlgoLined
	Given a target metric defect $\eps>0$, a pseudo-time step $\tau>0$ and a target tolerance $tol$\;
	Choose initial guess $\vy_h^0\in\A_{h,\eps}^k$\;
	\While{$\tau^{-1}|E_h[\vy_h^{n+1}]-E_h[\vy_h^{n}]|>$tol}{
		\textbf{Solve} \eqref{gf:minimization}-\eqref{gf:constraint} for $\delta\vy_h^{n+1}\in\V^k_h(\mathbf{0},\mathbf{0})$\;
		\textbf{Update} $\vy_h^{n+1} = \vy_h^{n}+\delta\vy_h^{n+1}$\;
	}
	\caption{(discrete-$H^2$ gradient flow) Finding local minima of $E_h$} \label{algo:main_GF}
\end{algorithm}

\noindent
We refer to Section~\ref{sec:solve} for a discussion on the implementation of Algorithm~\ref{algo:main_GF}.
We show in \cite{BGNY2020} that the discrete gradient flow satisfies the following properties:
\begin{itemize}
\item \textbf{Energy decay}: If $\delta\vy_h^{n+1}$ is nonzero, then we have 
\begin{equation}\label{eqn:control_energy}
E_h[\vy_h^{n+1}] < E_h[\vy_h^n].
\end{equation}
\item \textbf{Control of metric defect}: If $D_h[\vy_h^0]\le\eps_0$ and $E_h[\vy_h^0]<\infty$, then all the iterates $\vy_h^n$ satisfy $\vy_h^n \in \A^k_{h,\eps}$, i.e.,
\begin{equation} \label{eqn:control_defect}
D_h[\vy_h^n] \leq \eps := \eps_0 + \tau \Big(c_1 E_h[\vy_h^0] + c_2 \big(\|\vvarphi\|_{H^1(\Omega)}^2
  + \|\Phi\|_{H^1(\Omega)}^2 + \|\vf\|_{L^2(\Omega)}^2 \big)\Big), \!
\end{equation}
where $c_1,c_2$ depend on $\Omega$ if $\Gamma_D \not = \emptyset$ and also on $\sigma$ if $\Gamma_D = \emptyset$ but are independent of $n$, $h$ and $\tau$. Moreover, $c_2=0$ when $\Gamma_D = \emptyset$, as we assume $\vf=\mathbf{0}$ in the free boundary case.  
\end{itemize}
These two properties imply that the energy $E_h$ decreases at each step of Algorithm \ref{algo:main_GF} until a local extrema of $E_h$ restricted to $\A^k_{h,\eps}$ is attained.
  
\subsection{Initialization}  \label{subsec:preprocess}
The choice of an initial deformation $\vy^0_h$ is a very delicate matter. On the one hand, we need $\eps_0$ in \eqref{eqn:control_defect} as small as possible because the discrete gradient flow cannot improve upon the initial metric defect $D_h[\vy^0_h]\le\eps_0$. On the other hand, the only way to compensate for a large initial energy $E_h[\vy_h^0]$ is to take very small fictitious time steps $\tau$ that may entail many iterations of the gradient flow to reduce the energy. The value of $E_h[\vy_h^0]$ is especially affected by the mismatch between the Dirichlet boundary data $(\vvarphi,\Phi)$ and the trace of $\vy^0_h$ and $\nabla_h\vy^0_h$ that enter via the penalty terms in \eqref{def:form_ah} of LDG. Therefore, the role of the initialization process is to construct $\vy^0_h$ with $\eps_0$ relatively small and $E_h[\vy_h^0]$ of moderate size upon matching the boundary data $(\vvarphi,\Phi)$ as well as possible whenever $\Gamma_D \ne \emptyset$.

Notice that in some special cases, it is relatively easy to find such a $\vy^0_h$. For instance, when $g=\Id_2$ and $\Gamma_D \ne \emptyset$, this has been achieved in \cite{bonito2018} with a flat surface and a continuation technique.
For $g\ne\Id_2$ immersible, i.e., for which there exists a deformation $\vy\in[H^2(\Omega)]^3$ such that $\I[\vy]=g$, finding a good approximation $\vy_h^0$ of $\vy$ remains problematic and is the subject of this section.

\medskip\noindent
{\bf Metric preprocessing.}\label{sss:prestrain}
We recall that the stretching energy $E_s[\vy]$ of \eqref{E:stretching} must vanish for the asymptotic bending limit to make sense. We can monitor the deviation of $E_s[\vy]$ from zero to create a suitable $\vy_h^0$. Upon setting $\alpha^2=1$, we first observe that, since $g$ is uniformly positive definite, the first term in \eqref{E:stretching} satisfies
\[
\int_\Omega \big|g^{-\frac12} \, \I[\vy] \, g^{-\frac12} - \Id_2 \big|^2
\approx
\int_\Omega \big| \I[\vy] - g \big|^2 =
\int_\Omega \big| \nabla\vy^T \nabla\vy - g \big|^2;
\]
the same happens with the second term. We thus consider the discrete energy
\begin{equation} \label{def:E_prestrain_PP}
\widetilde{E}_h [\widetilde \vy_h] := \frac{1}{2}\int_{\Omega} \big|\nabla_h\widetilde \vy_h^T\nabla_h\widetilde \vy_h-g \big|^2
\end{equation}
and propose a discrete $H^2$-gradient flow to reduce it similar to that in Section \ref{subsec:GF}. We proceed recursively: given $\widetilde \vy^n_h$ we compute $\widetilde \vy^{n+1}_h:=\widetilde \vy^n_h+\delta \widetilde\vy_h^{n+1}$ by seeking the increment $\delta \widetilde\vy_h^{n+1}\in\V^k_h(\mathbf{0},\mathbf{0})$ that satisfies for all $\vv_h\in\V^k_h(\mathbf{0},\mathbf{0})$
\begin{equation}\label{pre-gf:variation2}
\widetilde{\tau}^{-1}(\delta\widetilde \vy_h^{n+1},\vv_h)_{H_h^2(\Omega)}+s_h(\widetilde \vy_h^n;\delta\widetilde \vy_h^{n+1},\vv_h)=
-s_h(\widetilde \vy_h^n; \widetilde \vy^n_h,\vv_h),
\end{equation}
where $\widetilde{\tau}$ is a pseudo time-step parameter, not necessarily the same as $\tau$ in Algorithm~\ref{algo:main_GF}, and $s_h(\widetilde \vy_h^n;\cdot,\cdot)$ is the variational derivative of $\widetilde{E}_h$ linealized at $\widetilde \vy_h^n$
\begin{equation}\label{pre-gf:bilinear}
s_h(\widetilde \vy_h^n; \vw_h,\vv_h):=\int_{\Omega}\Big(\nabla_h\vv_h^T\nabla_h\vw_h+\nabla_h\vw_h^T\nabla_h\vv_h\Big):\Big((\nabla_h\widetilde \vy_h^n)^T\nabla_h\widetilde \vy_h^n-g \Big).
\end{equation}
This flow admits a unique solution at each step because the left-hand side of \eqref{pre-gf:variation2} is coercive, namely
\begin{equation}\label{E:coercive}
  \|\vv_h\|_{H_h^2(\Omega)}^2 \lesssim \widetilde{\tau}^{-1} (\vv_h,\vv_h)_{H_h^2(\Omega)}
  + s_h(\widetilde \vy_h^n; \vv_h,\vv_h)
  \quad\forall\, \vv_h\in\V^k_h(\mathbf{0},\mathbf{0}).
\end{equation}
Moreover, this flow stops whenever either of the following two conditions is met: 
\begin{itemize}
	\item the prestrain defect $D_h$ reaches a prescribed value $\tilde \eps_0$, i.e, $D_h[\widetilde \vy_h^{n+1}]\le \widetilde \eps_0$;
	\item the energy $\widetilde E_h$ becomes stationary, i.e., $\widetilde \tau^{-1}|\tilde E_h[\widetilde \vy^{n+1}]-\widetilde E_h[\widetilde \vy^{n}] |\leq \widetilde{tol}$.
\end{itemize}
Monotone decay of $\widetilde{E}_h$ in \eqref{def:E_prestrain_PP} is not guaranteed by the flow because of the evaluation of $\delta\widetilde{E}_h$ at $\widetilde\vy_h^n$.
However, in all the numerical experiments proposed in Section~ \ref{sec:num_res}, the latter property is observed for $\widetilde \tau$ sufficiently small.
Upon choosing suitable parameters $\widetilde \eps_0$ and $\widetilde{\tau}$, this procedure produces initial configurations $\vy_h^0$ with small metric defect $D_h[\vy_h^0]$, but it has one important drawback: {\it flat configurations are local minimizers of \eqref{def:E_prestrain_PP} irrespective of $g$}. To see this, suppose that the current iterate $\widetilde{\vy}_h^n$ of \eqref{pre-gf:variation2} is flat, i.e., $\widetilde{\vy}_h^n=(y_1,y_2,0)$, and let $\delta\widetilde{\vy}_h^{n+1}=(d_1,d_2,d_3)\in\V^k_h(\mathbf{0},\mathbf{0})$, where the functions $y_i$ and $d_i$ depend on $(x_1,x_2)\in\Omega$. Take now $\vv_h=(0,0,\phi)\in \V^k_h(\mathbf{0},\mathbf{0})$ and note that
\begin{equation*}
  \nabla\vv_h^T \nabla\widetilde{\vy}_h^n =
  \begin{bmatrix}
   0 & 0 & \partial_1\phi \\ 0 & 0 & \partial_2\phi
  \end{bmatrix}
  \begin{bmatrix}
   \partial_1 y_1 & \partial_2 y_1 \\ \partial_1 y_2 & \partial_2 y_2 \\ 0 & 0
  \end{bmatrix}
  = \mathbf{0} = (\nabla\widetilde{\vy}_h^n)^T \nabla\vv_h,
\end{equation*}
whence the right-hand side of \eqref{pre-gf:variation2} vanishes. Since
$(\delta\widetilde{\vy}_h^{n+1},\vv_h)_{H_h^2(\Omega)} = (d_3,\phi)_{H_h^2(\Omega)}$,
taking $\phi=d_3$ and utilizing \eqref{E:coercive} we deduce $d_3=0$
because, as already pointed out, $\|\cdot\|_{H_h^2(\Omega)}:= (\cdot,\cdot)^{1/2}_{H_h^2(\Omega)}$ defines a norm on $\V^k_h(\mathbf{0},\mathbf{0})$.  
This shows that the next iterate
$\widetilde{\vy}_h^{n+1}$ of  \eqref{gf:minimization} is also flat and we need another mechanism to deform a flat surface out of plane provided $g$ does not admit a flat immersion. We discuss this next.

A second drawback of \eqref{gf:minimization} is that the stretching energy $\widetilde{E}_h$ is just first order and cannot accommodate the Dirichlet boundary condition $\nabla\vy=\Phi$ on $\Gamma_D$. We again need an additional preprocessing of the boundary conditions which we present next.

\medskip\noindent
{\bf Boundary conditions preprocessing.}\label{sss:bc_d}
We pretend that $g=\Id_2$ momentarily, and rely on \eqref{E:g=1} and \eqref{E:bilaplacian} to consider the bi-Laplacian problem provided $\Gamma_D\ne\emptyset$
\begin{equation}\label{bi-Laplacian}
  \Delta^2\widehat \vy = \widehat \vf\quad \mbox{in } \Omega, \quad
  \widehat \vy = \vvarphi \quad \mbox{on } \Gamma_D, \quad
  \nabla\widehat \vy = \Phi \quad \mbox{on } \Gamma_D,
\end{equation}  
where typically $\widehat\vf=\mathbf{0}$.
This vector-valued problem is well-posed and gives, in general, a non-flat surface $\widehat{\vy}(\Omega)$. We use the LDG method with boundary conditions imposed \emph{\`a la Nitsche} to approximate the solution $\widehat\vy\in\V(\vvarphi,\Phi)$ of \eqref{bi-Laplacian}:
\begin{equation}\label{bi-Laplacian-system}
\widehat \vy_h\in\V^k_h(\vvarphi,\Phi): \quad
c_h(\widehat \vy_h,\vv_h) = (\widehat\vf,\vv_h)_{L^2(\Omega)} \quad\forall\,\vv_h\in\V^k_h(\mathbf{0},\mathbf{0}).
\end{equation}
Here, $c_h(\widehat \vy_h,\vv_h)$ is defined similarly to \eqref{def:form_ah} using the discrete Hessian \eqref{def:discrHess}, i.e.,
\begin{equation}\label{bi-Laplacian_bilinear}
\begin{split}
c_h(\vw_h,\vv_h) := & \int_{\Omega}H_h[\vw_h] : H_h[\vv_h]  \\
&+\widehat\gamma_1(\h^{-1}\jump{\nabla_h\vw_h},\jump{\nabla_h\vv_h})_{L^2(\Gh^a)}
+\widehat\gamma_0(\h^{-3}\jump{\vw_h},\jump{\vv_h})_{L^2(\Gh^a)},
\end{split}
\end{equation}
where $\widehat{\gamma}_0$ and $\widehat{\gamma}_1$ are positive penalty parameters that may not necessarily be the same as their counterparts $\gamma_0$ and $\gamma_1$ used in the definition of $E_h$. Then $\widehat\vy_h$ satisfies (approximately) the given boundary conditions on $\Gamma_D$ and $\widehat\vy_h(\Omega)$ is, in general, non-flat.

Instead, if $\Gamma_D=\emptyset$ (free boundary condition), then an obvious choice is $\widehat{\vy}=(id,0)^T$, where $id(x)=x$ for $x \in \Omega$, but the surface $\widehat{\vy}(\Omega)=\Omega\times0$ is flat. To get a surface out of plane, we consider a somewhat ad-hoc procedure: we solve \eqref{bi-Laplacian} with a fictitious forcing $\widehat\vf\ne\mathbf{0}$ supplemented with the Dirichlet boundary condition $\vvarphi(x)=(x,0)^T$ for $x \in \partial\Omega$ but obviating $\Phi$ and jumps of $\nabla_h\widehat{\vy}_h$ on $\Gh^b$ in \eqref{bi-Laplacian_bilinear}. This corresponds to enforcing discretely a variational (Neumann) boundary condition $\Delta\widehat{\vy} = 0$ on $\partial\Omega$.

We summarize the previous discussion of preprocessing in Algorithm \ref{algo:preprocess},
which
\RestyleAlgo{boxruled}
\begin{algorithm}[htbp]
\SetAlgoLined
	Given $\widetilde{tol}$ and $\widetilde\eps_0$\;
	\eIf{$\Gamma_D\ne\emptyset$ (Dirichlet boundary condition)}{
		\textbf{Solve} \eqref{bi-Laplacian-system} for $\widehat \vy_h\in\V_h^k(\vvarphi,\Phi)$ with $\widehat\vf=\mathbf{0}$\;
		}{\textbf{Solve} \eqref{bi-Laplacian-system} for $\widehat \vy_h$ with $\widehat\vf\ne\mathbf{0}$, $\vvarphi=(id,0)$ and without $\Phi$\;}{
		Set $\widetilde{\vy}_h^0=\widehat \vy_h$\;
		}
 \While{$\widetilde \tau^{-1}\big|E_h[\widetilde{\vy}_h^{n+1}]-E_h[\widetilde{\vy}_h^{n}]\big|> \widetilde{tol}$ \emph{ and } $D_h[\widetilde{\vy}_h^{n+1}] > \widetilde\eps_0$}{
  \textbf{Solve} \eqref{pre-gf:variation2}  for $\delta\widetilde{\vy}_h^{n+1}\in\V^k_h(\mathbf{0},\mathbf{0})$\;
  \textbf{Update} $\widetilde{\vy}_h^{n+1} = \widetilde{\vy}_h^{n}+\delta\widetilde{\vy}_h^{n+1}$ \;
  }
 \textbf{Set} $\vy_h^0=\widetilde{\vy}_h^{n+1}$.
 \caption{Initialization step for Algorithm \ref{algo:main_GF}.} \label{algo:preprocess}
\end{algorithm}
consists of two separate steps: the \textit{boundary conditions} and \textit{metric} preprocessing steps. 
When $\Gamma_D \not = \emptyset$ (Dirichlet boundary condition), the former constructs a solution $\widehat\vy_h$ to \eqref{bi-Laplacian-system} with $\widehat\vf=\mathbf{0}$, whence $\widehat\vy_h \approx\vvarphi$ and $\nabla_h\widehat\vy_h \approx \Phi$ on $\Gamma_D$. Instead, when $\Gamma_D = \emptyset$ (free boundary condition), $\widehat\vy_h$ solves \eqref{bi-Laplacian-system} again but now with  $\widehat\vf\ne\mathbf{0}$ and a suitable boundary condition for $\vy_h$ on $\partial\Omega$ that guarantee $\widehat\vy_h(\Omega)$ is non-flat. The output of this step is then used as an initial guess for the metric preprocessing step \eqref{pre-gf:variation2}.

It is conceivable that more efficient or physically motivated algorithms could be designed to construct initial guesses. We leave these considerations for future research. As we shall see in Section~\ref{sec:num_res}, different initial deformations can lead to different equilibrium configurations corresponding to distinct local minima of the energy $E_h$ in \eqref{def:Eh}. These minima are generally physically meaningful.

\section{Implementation}  \label{sec:solve}

We make a few comments on the implementation of the gradient flow \eqref{gf:minimization}-\eqref{gf:constraint}, built in Algorithm \ref{algo:main_GF}, and the resulting linear algebra solver used at each step. 

\subsection{Linear constraints}
We start by discussing how the linearized metric constraint \eqref{gf:constraint} is enforced using piecewise constant Lagrange multipliers in the space
\begin{equation*}
\Lambda_h:=\left\{\lambda_h:\Omega\to\mathbb{R}^{2\times2}: \,\, \lambda_h^T=\lambda_h, \,\, \lambda_h\in\big[\V_h^0\big]^{2\times 2}\right\}.
\end{equation*}
We define the bilinear form $b_h^n$ for any $(\vv_h,\vmu_h)\in\V^k_h(\mathbf{0},\mathbf{0})\times\Lambda_h$ to be
\begin{equation}\label{def:bh}
b_h^n(\vv_h,\vmu_h):=\sum_T\int_T(\nabla \vv_h^T\nabla\vy_h^n+(\nabla\vy_h^n)^T\nabla\vv_h):\vmu_h.
\end{equation}
We observe that $b_h^n$ depends on $\vy_h^n$ and that $b_h^n(\delta\vy_h^{n+1},\vmu_h)=0$ for all $\vmu_h\in\Lambda_h$ implies \eqref{gf:constraint}, i.e., $L_T[\vy_h^n;\delta\vy_h^{n+1}]=0$ for all $T\in\Th$. Therefore, recalling the forms $a_h$ and $F_h$ in \eqref{def:form_ah} and \eqref{def:form_Fh}, the augmented system for the Euler-Lagrange equation \eqref{E:E-L} incorporating the gradient flow step and the linearized metric constraint reads: seek $(\delta\vy_h^{n+1},\vla_h^{n+1})\in\V^k_h(\mathbf{0},\mathbf{0})\times \Lambda_h$ such that
\begin{equation}\label{gf:system}
\begin{aligned}
  \tau^{-1}(\delta\vy_h^{n+1},\vv_h)_{H_h^2(\Omega)} \!+\! a_h(\delta\vy_h^{n+1},\vv_h)
  \!+\! b_h^n(\vv_h,\vla_h^{n+1})& \!=\! F_h(\vv_h) \!-\! a_h(\vy_h^n,\vv_h)  \\
b_h^n(\delta\vy_h^{n+1},\vmu_h) & \!=\! 0
\end{aligned}
\end{equation}
for all $(\vv_h,\vmu_h)\in\V^k_h(\mathbf{0},\mathbf{0})\times \Lambda_h$. 
Since $\vy_h^n\in\V_h^k(\vvarphi,\Phi)$, whence $\vy_h^{n+1}=\vy_h^n+\delta\vy_h^{n+1}\in\V_h^k(\vvarphi,\Phi)$, the effect of the Dirichlet boundary data $(\vvarphi,\Phi)$ is implicitly contained in $a_h(\vy_h^n,\vv_h)$ when $\Gamma_D$ is not empty.
   
\subsection{Solvers}
Let $\{\vvarphi_h^i\}_{i=1}^N$ be a basis for $\V^k_h(\mathbf{0},\mathbf{0})$ and let $\{\vpsi_h^i\}_{i=1}^M$ be a basis for $\Lambda_h$. The discrete problem \eqref{gf:system} is a {\it saddle-point problem} of the form
\begin{equation}\label{discrete_system}
\begin{bmatrix}
A       & B_n^T \\
B_n       & 0
\end{bmatrix}
\begin{bmatrix}
\vdY_h^{n+1} \\
\vLambda_h^{n+1}
\end{bmatrix}
=
\begin{bmatrix}
\vF_n  \\
\mathbf{0}
\end{bmatrix}.
\end{equation}
Here, $(\vdY_h^{n+1},\vLambda_h^{n+1})$ are the nodal values of $(\delta\vy_h^{n+1},\vla_h^{n+1})$ in these bases, $A=(A_{ij})_{i,j=1}^N\in\mathbb{R}^{N\times N}$ is the matrix corresponding to the first two terms of \eqref{gf:system}
\begin{equation*}
  A_{ij}:=\tau^{-1}(\vvarphi_h^j,\vvarphi_h^i)_{H_h^2(\Omega)} + \widetilde A_{ij} \quad \mbox{with} \quad \widetilde A_{ij}:=a_h(\vvarphi_h^j,\vvarphi_h^i), \quad i,j=1,\ldots,N,
\end{equation*}
while the matrix $B_n\in\mathbb{R}^{M\times N}$ corresponds to the bilinear form $b_h^n$ and is given by
\begin{equation*}
(B_n)_{ij}:=b_h^n(\vvarphi_h^j,\vpsi_h^i) \quad i=1,\ldots,M, \, j=1,\ldots,N.
\end{equation*}
The vector $\vF_n\in\mathbb{R}^N$ accounts for the right-hand-side of \eqref{gf:system}. It reads $\vF_n=\vF+\vL-\widetilde A \vY^n$, where $\vY^n$ contains the nodal values of $\vy_h^n$ in the basis $\{\vvarphi_h^i\}_{i=1}^N$ while $\vF=(F_i)_{i=1}^N$ and $\vL=(L_i)_{i=1}^N$ are defined by
\begin{equation*}
F_i:= F_h(\vvarphi_h^i) \quad \mbox{and} \quad L_i:= -a_h(\bar \vo,\vvarphi_h^i), \quad i=1,\ldots,N.
\end{equation*}
Here, $\bar \vo$ denotes the zero function in the space $\V_h(\vvarphi,\Phi)$ and $\vL$ contains the liftings of the boundary data. Since $B_n$ and $\vF_n$ depend explicitly on the current deformation $\vy_h^n$, they have to be re-computed at each iteration of Algorithm \ref{algo:main_GF} (gradient flow). In contrast, the matrices $A$ and $\widetilde A$ and the vector $\vL$, which are the most costly to assemble because of the reconstructed Hessians, are independent of the iteration number $n$ and can thus be computed once for all.

More precisely, to compute the element-wise contribution on a cell $T$, the discrete Hessian \eqref{def:discrHess} of each basis function associated with $T$ along with those associated with the neighboring cells is computed. Recall that for any interior edge $e\in\Eh^i$, the support of the liftings $r_e$ and $b_e$  in \eqref{def:lift_re} and \eqref{def:lift_be} is the union of the two cells sharing $e$ as an edge. We employ direct solvers for these small systems. We proceed similarly for the computation of the liftings of the boundary data $\vvarphi$ and $\Phi$. Once the discrete Hessians are computed, the rest of the assembly process is standard. Incidentally, we note that the proposed LDG approach couples the degree of freedom (DoFs) of all neighboring cells (not only the cell with its neighbors). As a consequence, the sparsity pattern of LDG is slightly larger than it for a standard symmetric interior penalty dG (SIPG) method. However, the stability properties of LDG are superior to those of SIPG \cite{BGNY2020}.

System \eqref{discrete_system} can be solved using the \textit{Schur complement method}. Denoting $S_n:=B_nA^{-1}B_n^T$ the Schur complement matrix, the first step determines $\vLambda_h^{n+1}$ satisfying
\begin{equation} \label{eqn:Schur}
S_n\vLambda_h^{n+1}=B_nA^{-1}\vF_n,
\end{equation}
followed by the computation of $\delta \vY_h^{n+1}$ solving
\begin{equation}\label{E:deltaY}
A\delta \vY_h^{n+1}=\vF_n-B_n^T\vLambda_h^{n+1}. 
\end{equation}
Because the matrix $A$ is independent of the iterations, we pre-compute its LU decomposition once for all and use it whenever the action of $A^{-1}$ is needed in \eqref{eqn:Schur} and \eqref{E:deltaY}. 
Furthermore, a conjugate gradient algorithm is utilized to compute $\vLambda_h^{n+1}$ in \eqref{eqn:Schur} to avoid assembling $S_n$. The efficiency of the latter depends on the condition number of the matrix $S_n$, which in turn depends on the inf-sup constant of the saddle-point problem \eqref{discrete_system}. Leaving aside the preprocessing step, we observe in practice that solving the Schur complement problem \eqref{eqn:Schur} is the most time consuming part of the simulation. Finally, we point out that the stabilization parameters $\gamma_0$ and $\gamma_1$ influence the number of Schur complement iterations: more iterations of the gradient conjugate algorithm are required for larger stabilization parameter values. We refer to Tables \ref{tab:iso_vertical_load_0025} and \ref{tab:iso_vertical_load_0025_LDG_SIPG} below for more details.

\section{Numerical experiments}  \label{sec:num_res}

In this section, we present a collection of numerical experiments to illustrate the performance of the proposed methodology. We consider several prestrain tensors $g$, as well as both $\Gamma_D\neq\emptyset$ (Dirichlet boundary condition) and $\Gamma_D=\emptyset$ (free boundary condition). The Algorithms \ref{algo:main_GF} and \ref{algo:preprocess} are implemented using the \textrm{deal.ii} library \cite{bangerth2007} and the visualization is performed with \textrm{paraview} \cite{ayachit2015}. The color code is the following: (multicolor figures) dark blue indicates the lowest value of the deformation's third component while dark red indicate the largest value of the deformation's third component; (unicolor figures) magnitude of the deformation's third component.

For all the simulations, we fix the polynomial degree $k$ of the deformation $\vy_h$ and  $l_1,l_2$ for the two liftings of the discrete Hessian $H_h[\vy_h]$ to be
\[
k=l_1=l_2=2.
\]
Moreover, unless otherwise specified, we set the Lam\'e coefficients to $\lambda=8$ and $\mu=6$, and the stabilization parameters for \eqref{def:form_ah} and \eqref{bi-Laplacian_bilinear} to be
\[
\gamma_0=\gamma_1=1,
\qquad
\widehat \gamma_0=\widehat \gamma_1=1.
\]
In striking contrast to \cite{bonito2018,bonito2020discontinuous}, these parameters do not need to be large for stability purposes. When $\Gamma_D=\emptyset$, we set $\epsilon=1$ in \eqref{def:H2metric}. Finally, we choose $tol=10^{-6}$ for the stopping criteria in Algorithm \ref{algo:main_GF} (gradient flow).

To record the energy $E_h$ and metric defect $D_h$ after the three key procedures described in Section \ref{sec:method}, we resort to the following notation:
{\it BC PP} (boundary conditions preprocessing);
{\it Metric PP} (metric preprocessing);
{\it Final}  (gradient flow).
 
\subsection{Vertical load and isometry constraint}\label{S:vertical-load}
This first example has been already investigated in \cite{bartels2013,bonito2018}.
We consider the square domain $\Omega=(0,4)^2$, the metric $g=\Id_2$ (isometry) and a vertical load $\vf=(0,0,0.025)^T$. Moreover, the plate is clamped on $\Gamma_D=\{0\}\times[0,4]\cup[0,4]\times\{0\}$, i.e., we prescribe the Dirichlet boundary condition \eqref{eq:dirichlet} with
\begin{equation*}
\vvarphi(x_1,x_2)=(x_1,x_2,0)^T, \quad \Phi=[I_2,\mathbf{0}]^T \qquad (x_1,x_2) \in \Gamma_D.
\end{equation*}
Finally, we set the Lam\'e constant $\lambda=0$ thereby removing the trace term in \eqref{def:Eh}.

No preprocessing step is required because the flat plate, which corresponds to the identity deformation $\vy_h^0(\Omega)=\Omega$, satisfies the metric constraint and the boundary conditions. For the discretization of $\Omega$, we use $\ell=0,1,2,\cdots$ to denote the refinement level and consider uniform partitions $\mathcal T_\ell$ consisting of squares $\K$ of side-length $4/2^\ell$ and diameters $h_{\K}=h=\sqrt{2}/2^{\ell-2}$. The pseudo-time step used for the discretization of the gradient flow is chosen so that $\tau=h$. 
The discrete energy $E_h[\vy_h]$ and metric defect $D_h[\vy_h]$ for $\ell=3,4,5$ are report in Table \ref{tab:iso_vertical_load_0025} along with the number of gradient flow iterations (GF Iter) required to reach the targeted stationary tolerance and the range of number of iterations (Schur Iter) needed to solve the Schur complement problem \eqref{eqn:Schur}. Note that in this case we have $D_h[\vy_h^0]=0$, namely $\vy_h^0\in\mathbb{A}_{h,\eps_0}^k$ with $\eps_0=0$.

\begin{table}[htbp]
\begin{center}
\begin{tabular}{|c|c|c|c|c|c|c|}
\hline
Nb. cells & DoFs & $\tau=h$ & $E_h$ & $D_h$ & GF Iter &  Schur Iter \\
\hline
64 & 1920 & $\sqrt{2}/2$ & -1.002E-2 & 1.062E-2 & 11 & [60,65] \\ \hline
256 & 7680 & $\sqrt{2}/4$ & -9.709E-3 & 5.967E-3 & 17 & [85,101] \\ \hline
1024 & 30720 & $\sqrt{2}/8$ & -8.762E-3 & 2.962E-3 & 28 & [118,148] \\
\hline
\end{tabular}
\vspace{0.3cm}
\caption{Effect of the numerical parameters $h$ and $\tau=h$ on the energy and prestrain defect for the vertical load example using $\gamma_0=\gamma_1=1$.
As expected \cite{bartels2013,bonito2015,bonito2018}, we observe that $D_h[\vy_h]$ is $\mathcal{O}(h)$. The number of iterations needed by the gradient flow and for each Schur complement solver increases with the resolution. 
} \label{tab:iso_vertical_load_0025}
\end{center}
\end{table}

We point out that the SIPG method analyzed in \cite{bonito2018} requires $\gamma_0=5000$ and $\gamma_1=1100$ in this example. We report in Table \ref{tab:iso_vertical_load_0025_LDG_SIPG} the performance of both methods with this choice of stabilization parameters but using the definition of the mesh function \eqref{eqn:mesh_function} rather than $\h(\vx)=\max_{\K\in\mathcal T}h_T$ as in \cite{bonito2018}.

\begin{table}[htbp]
\begin{center}
\begin{tabular}{|c|c|c|c|c|c|c|c|c|}
\cline{2-9}
\multicolumn{1}{c|}{ } & \multicolumn{4}{c|}{LDG} & \multicolumn{4}{c|}{SIPG} \\
\hline
$\tau=h$ & $E_h$ & $D_h$ & GF Iter &  Schur Iter & $E_h$ & $D_h$ & GF Iter & Schur Iter \\
\hline		
$\sqrt{2}/2$ & -8.28E-3 & 7.71E-3 & 7 & [302,321] & -8.30E-3 & 7.72E-3 & 7 & [284,307] \\
\hline
$\sqrt{2}/4$ & -6.63E-3 & 3.45E-3 & 14 & [557,605] & -6.64E-3 & 3.46E-3 & 13 & [556,600] \\
\hline
$\sqrt{2}/8$ & -4.88E-3 & 1.34E-3 & 37 & [788,831] & -4.90E-3 & 1.34E-3 & 35 & [787,833] \\
\hline
\end{tabular}
\vspace{0.3cm}
\caption{Comparison of the LDG and SIPG methods using the penalization parameters $\gamma_0=5000$, $\gamma_1=1100$ required by the SIPG. The results are similar.} \label{tab:iso_vertical_load_0025_LDG_SIPG}
\end{center}	
\end{table}

Based on Table \ref{tab:iso_vertical_load_0025_LDG_SIPG}, we see that the two methods give similar results. The advantage of the LDG approach is that there is no constraint on the stabilization parameters $\gamma_0$ and $\gamma_1$ other than being positive. In contrast, the coercivity of the energy discretized with the SIPG method requires {$\gamma_0$ and $\gamma_1$ to be} sufficiently large (depending on the maximum number of edges of the elements in the subdivision $\mathcal T$ and the constant in the trace inequality) \cite{bonito2018}. For instance, the choice $\gamma_0=\gamma_1=1$ for the SIPG method yields an unstable scheme and the problem \eqref{gf:system} becomes singular
after a few iterations of the gradient flow. Moreover, the large values of $\gamma_0,\gamma_1$ are mainly dictated by the penalty of the boundary terms in $E_h[\vy_h^0]$ and the need to produce moderate values of $E_h[\vy_h^0]$ to prevent very small time steps $\tau$ in \eqref{eqn:control_defect}. Furthermore, within each gradient flow iteration, the solution of the Schur complement problem \eqref{eqn:Schur} using the LDG approach with $\gamma_0=\gamma_1=1$ requires less than a fifth of the iterations (Schur Iter) for SIPG with $\gamma_0=5000$ and $\gamma_1=1100$ at the expense of slightly larger number of iterations of the gradient flow (GF Iter); compare Tables \ref{tab:iso_vertical_load_0025} and \ref{tab:iso_vertical_load_0025_LDG_SIPG}. This documents a superior performance of LDG relative to SIPG.

Note that there is an artificial displacement along the diagonal $x_1+x_2=4$ \cite{bonito2018,bartels2013} for this example, which does not correspond to the actual physics of the problem, namely $y = 0$ for $x_1+x_2\le4$. 
The artificial displacements obtained by the two methods for various meshes are compared in Figure \ref{fig:deformation_accross_diagonal} and Table \ref{T:defection}.

\begin{figure}[htbp]
	\begin{center}
		\includegraphics[width=6.0cm]{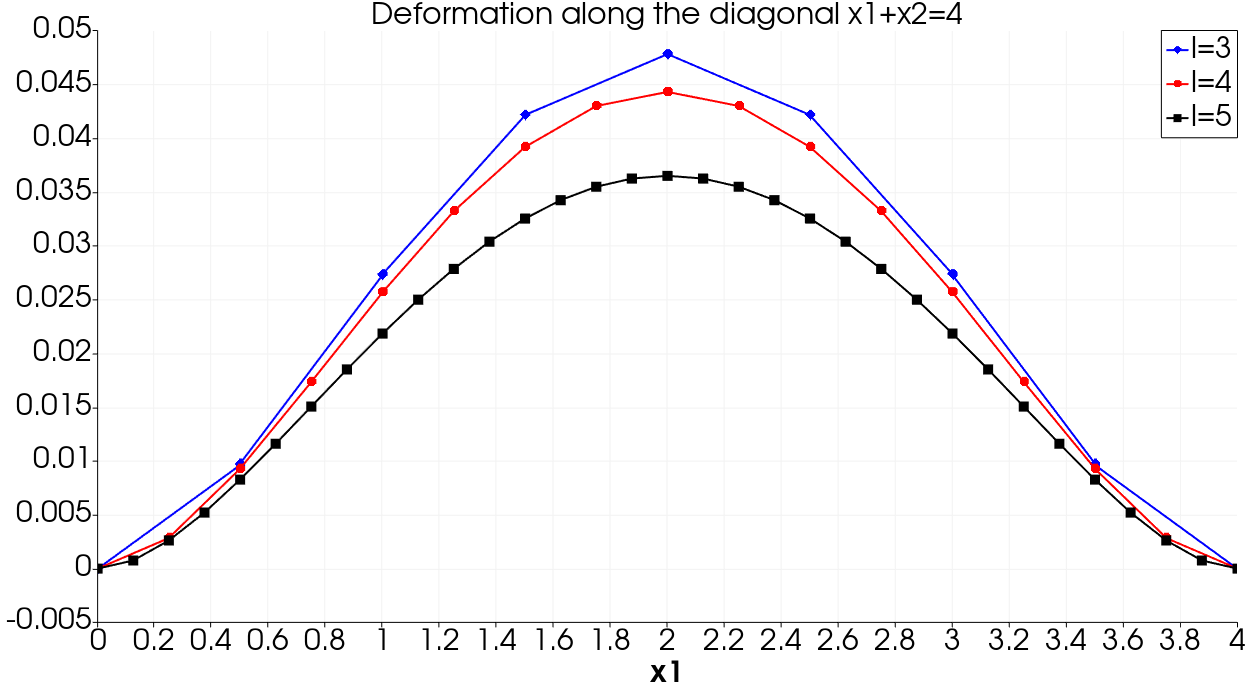} \\
		\includegraphics[width=6.0cm]{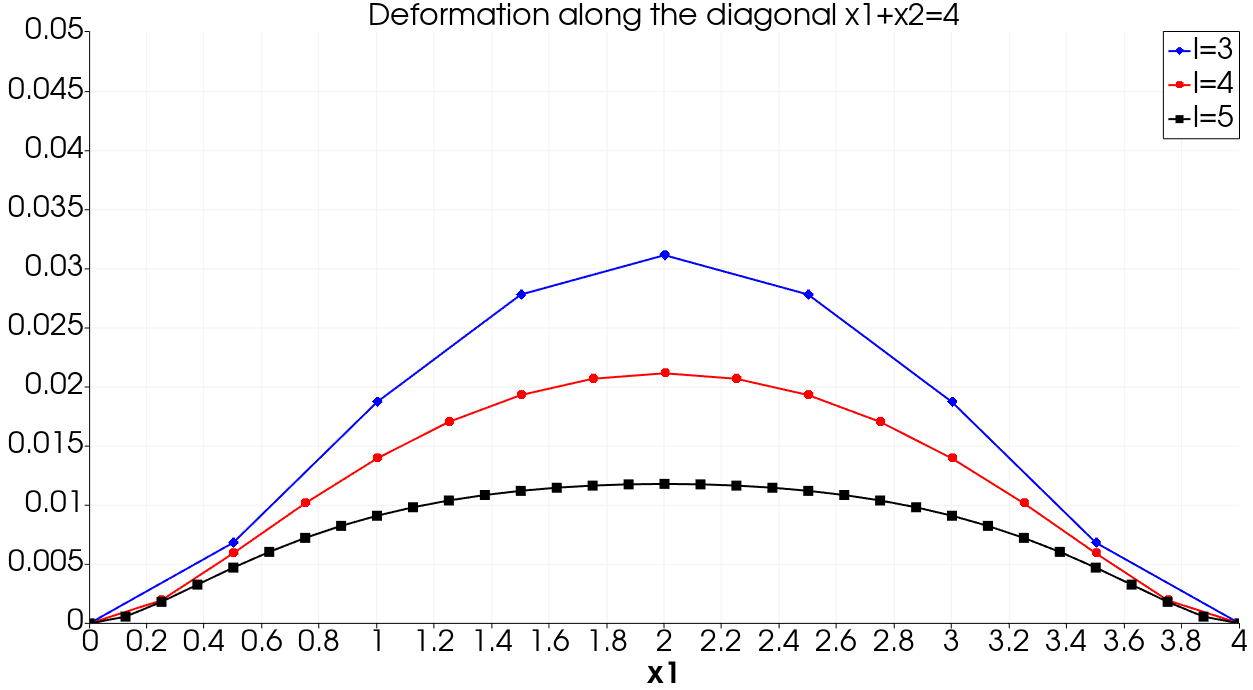}
		\includegraphics[width=6.0cm]{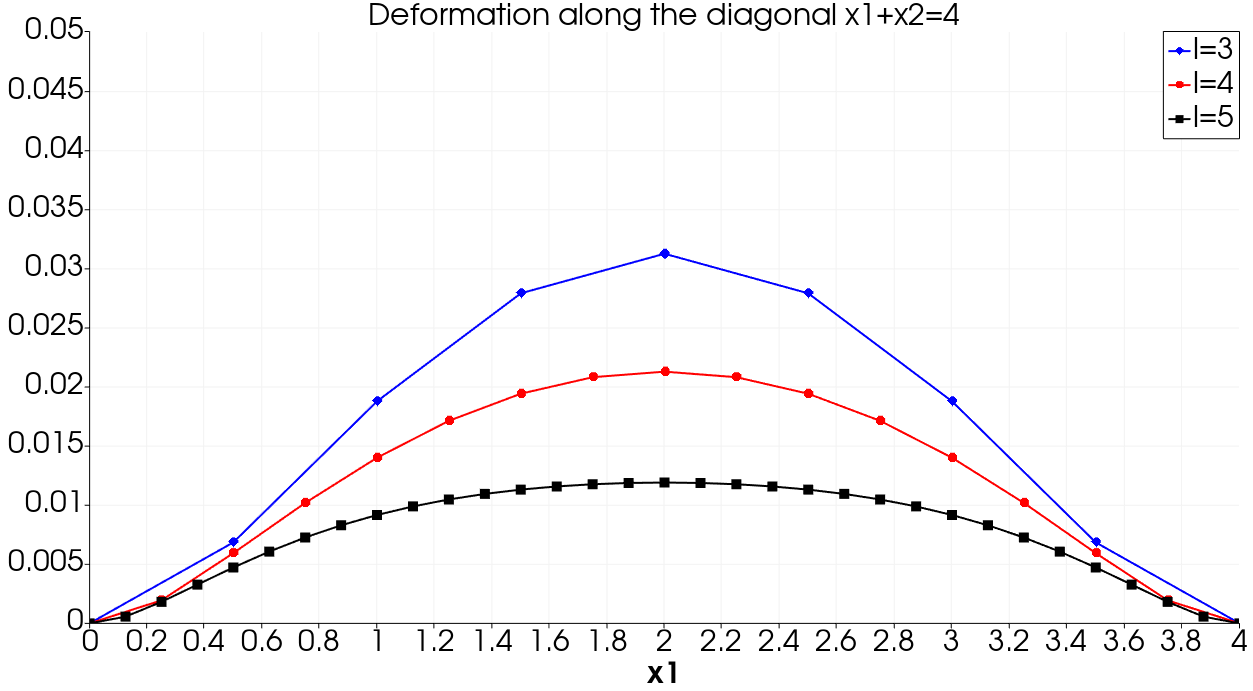}
		\caption{Deformation along the diagonal $x_1+x_2=4$. Top: LDG with $\gamma_0=\gamma_1=1$; bottom-left: LDG with $\gamma_0=5000$ and $\gamma_1=1100$; bottom-right: SIPG with $\gamma_0=5000$ and $\gamma_1=1100$. The deflection is slightly larger when $\gamma_0=\gamma_1=1$ while both methods yield similar results when $\gamma_0=5000$ and $\gamma_1=1100$; see Table \ref{T:defection}.}\label{fig:deformation_accross_diagonal}
	\end{center}
\end{figure}
\begin{table}[htbp]
\begin{center}
	\begin{tabular}{|l|c|c|c|}
		\cline{2-4}
		\multicolumn{1}{c|}{ } & \multicolumn{2}{c|}{LDG} & \multicolumn{1}{c|}{SIPG} \\
		\hline
		$\sharp$ ref. & $\qquad\gamma_0=\gamma_1=1\qquad$ & $\gamma_0=5000,\gamma_1=1100$ & $\gamma_0=5000,\gamma_1=1100$ \\
		\hline
		$l=3$ & 0.0478 & 0.0311 & 0.0312 \\
		\hline
		$l=4$ & 0.0443 & 0.0211 & 0.0213 \\
		\hline
		$l=5$ & 0.0365 & 0.0118 & 0.0119 \\
		\hline
	\end{tabular}
\vspace{0.3cm}
\caption{Deflection $y_3$ along the diagonal $x_1+x_2=4$ for both LDG and SIPG
}\label{T:defection}
\end{center}
\end{table}

\subsection{Rectangle with \emph{cylindrical} metric}
The domain is the rectangle $\Omega=(-2,2)\times(-1,1)$ and the Dirichlet boundary is $\Gamma_D=\{-2\}\times(-1,1)\cup\{2\}\times(-1,1)$. The mesh $\Th$ is uniform and made of 1024 rectangular cells of diameter $h_{\K}=h=\sqrt{5}/4$ (30720 DoFs) and the pseudo time-step is fixed to $\tau = 0.1$.

\subsubsection{\bf \emph{One mode}} \label{sec:one_mode}
We first consider the immersible metric
\begin{equation} \label{def:metric_one_mode}
g(x_1,x_2) =
\begin{bmatrix}
1+\frac{\pi^2}{4}\cos\left(\frac{\pi}{4}(x_1+2)\right)^2 & 0 \\
0 & 1
\end{bmatrix}
\end{equation}
for which
\begin{equation}\label{eq:y_exacy_one}
\vy(x_1,x_2) = (x_1,x_2,2\sin(\frac{\pi}{4}(x_1+2)))^T
\end{equation}
is a compatible deformation (isometric immersion), i.e., $\I[\vy]=g$. We impose the boundary conditions $\vvarphi = \vy|_{\Gamma_D}$ and $\Phi = \nabla \vy|_{\Gamma_D}$, so that $\vy\in\V(\vvarphi,\Phi)$ is an admissible deformation and also a global minimizer of the energy.

To challenge our algorithm, we start from a flat initial plate and obtain an admissible initial deformation $\vy_h^0$ using the two preprocessing steps (BC PP and Metric PP) in Algorithm \ref{algo:preprocess} with parameters
\[
\widetilde \tau = 0.05, \quad \widetilde\eps_0=0.1 \quad \mbox{and} \quad \widetilde{tol}=10^{-6}.
\]
The deformation obtained after applying Algorithms \ref{algo:preprocess} and  \ref{algo:main_GF} are displayed in Figure \ref{fig:cylinder_one_mode}. Moreover, the corresponding energy and prestrain defect are reported in Table \ref{tab:cylinder_one_mode}. Notice that the target metric defect $\widetilde\eps_0$ is reached in 49 iterations while 380 iterations of the gradient flow are needed to reach the stationary deformation. 

\begin{figure}[htbp]
	\begin{center}
		\includegraphics[width=4.0cm]{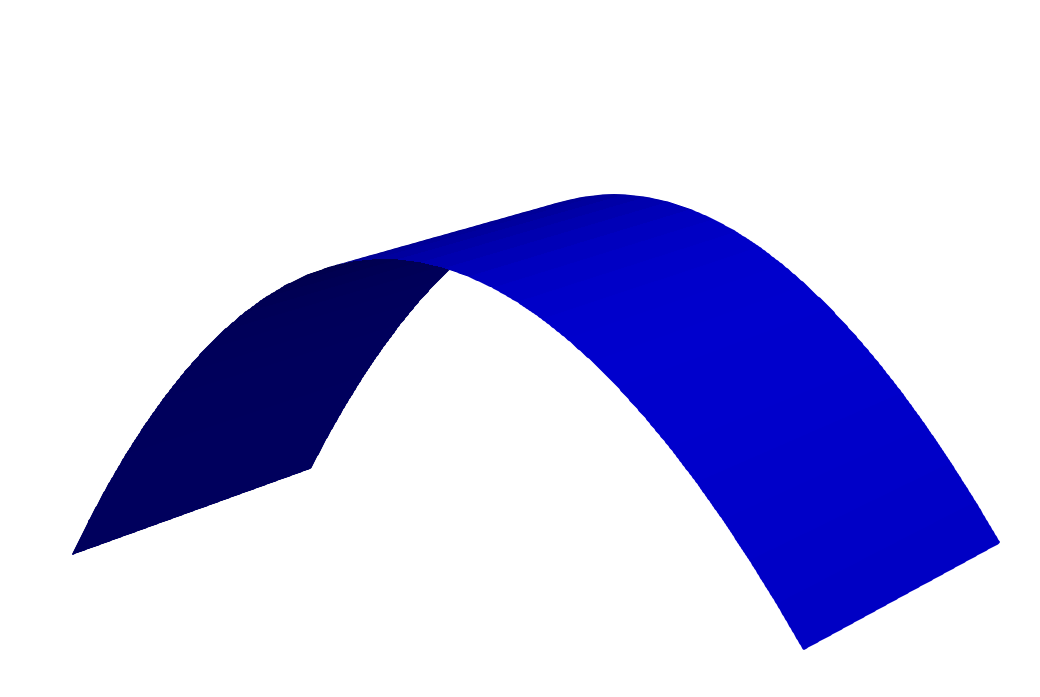}
		\includegraphics[width=4.0cm]{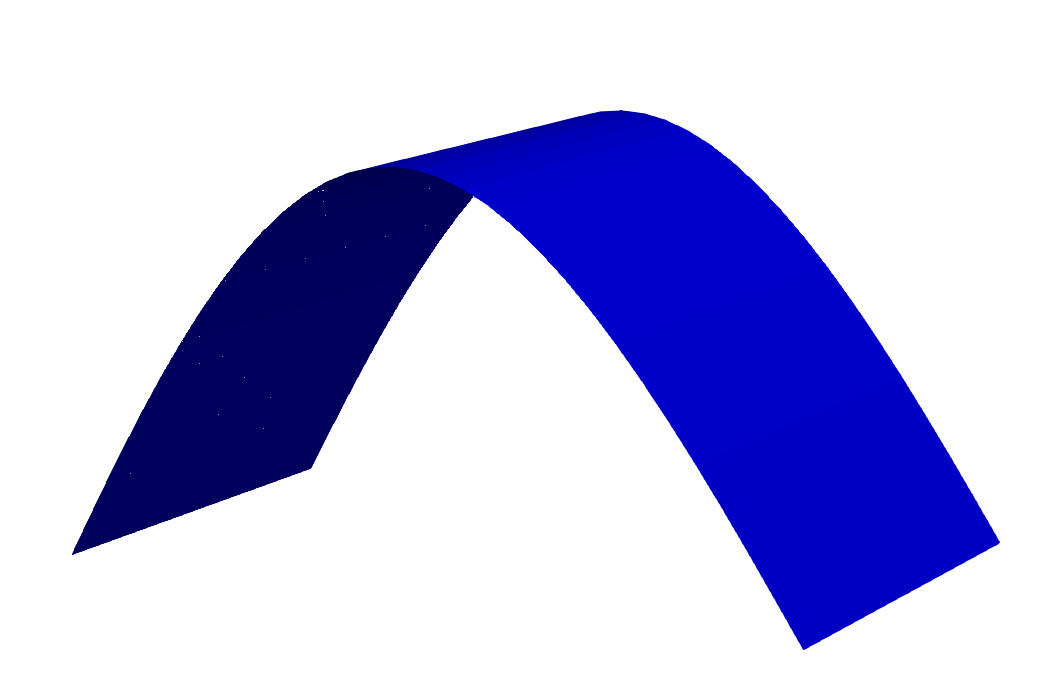}
		\includegraphics[width=4.0cm]{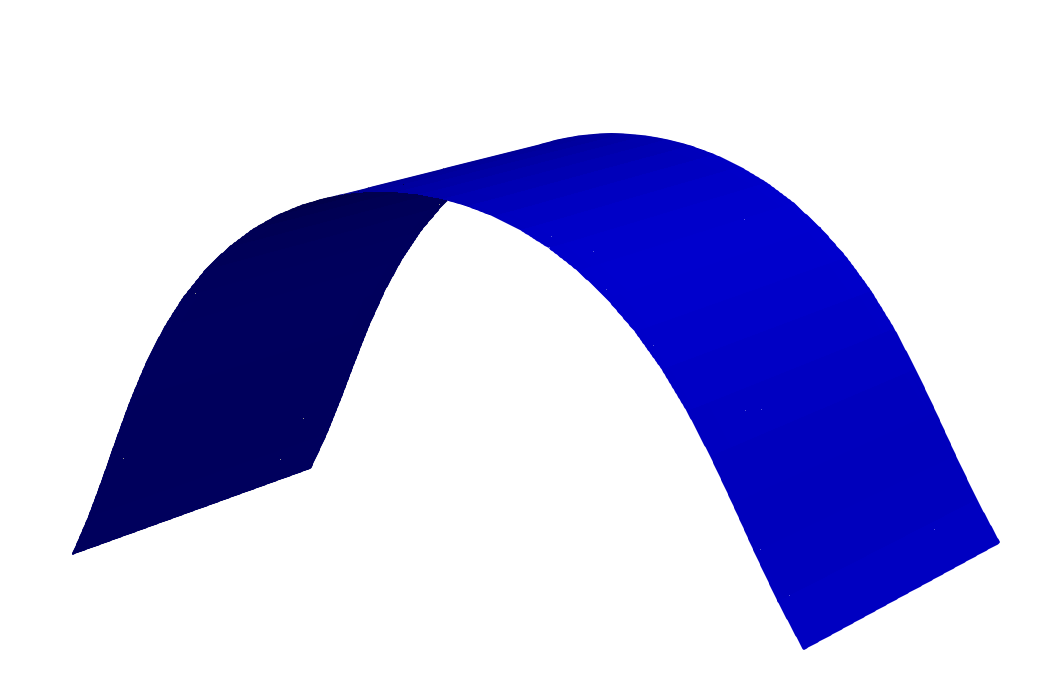}
		\caption{Deformed plate for the cylinder metric with one mode. Left: BC PP; middle: Metric PP; right: Final.} \label{fig:cylinder_one_mode}
	\end{center}
\end{figure}

\begin{table}[htbp]
	\begin{center}
		\begin{tabular}{|c|c|c|c|c|}
			\cline{2-5}
			\multicolumn{1}{c|}{ } & Initial & BC PP & Metric PP & Final \\
			\hline
			$E_h$ & 120.3590 & 1.1951 & 2.5464 & 1.7707 \\
			\hline
			$D_h$ & 9.8696 & 3.2899 & 9.8609E-2 & 9.5183E-2 \\
			\hline
		\end{tabular}
                \vspace{0.3cm}
		\caption{Energy and prestrain defect for the cylinder metric with one mode. All the algorithms behave as intended: the boundary conditions preprocessing (BC PP) reduces the energy by constructing a deformation with compatible boundary conditions, the metric preprocessing (Metric PP) reduces the metric defect and the gradient flow (Final) reduced the energy to its minimal value while keeping a control on the metric defect.} \label{tab:cylinder_one_mode}
	\end{center}	
\end{table}

Interestingly, when no Dirichlet boundary conditions are imposed, i.e., the free boundary case, then the flat deformation (pure stretching) 
\begin{equation*}
\vy(x_1,x_2) = \left(\int_{-2}^{x_1}\sqrt{1+\frac{\pi^2}{4}\cos\left(\frac{\pi}{4}(s+2)\right)^2} ds,x_2,0\right)^T
\end{equation*}
is also compatible with the metric \eqref{def:metric_one_mode} and has a smaller energy. We observe that $y_1(2,x_2)-y_1(-2,x_2)\approx 5.85478$ for $x_2\in(-2,2)$ corresponds to a stretching ratio of approximately $1.5$.  The outcome of Metric PP in Algorithm \ref{algo:preprocess} starting from the flat plate produces an initial deformation with $E_h=0.81755$ and $D_h=0.09574$ using 37 iterations. The stationary solution of the main gradient flow is reached in 68 iterations and  produces a flat plate with energy $E_h=0.376257$ and metric defect $D_h=0.0957329$. 
 
 \subsubsection{\bf \emph{Two modes}}
This example is similar to that of Section \ref{sec:one_mode} but with one additional \emph{mode} of higher frequency, namely we consider the immersible metric
\begin{equation*}
g(x_1,x_2) =
\begin{bmatrix}  
1+\left(\frac{\pi}{2}\cos\left(\frac{\pi}{4}(x_1+2)\right)+\frac{5\pi}{8}\cos\left(\frac{5\pi}{4}(x_1+2)\right)\right)^2 & 0 \\
0 & 1
\end{bmatrix}.
\end{equation*}
In this case, the deformation
$$
\vy(x_1,x_2)=\left(x_1,x_2, 2\sin\left(\frac{\pi}{4}(x_1+2)\right)+\frac{1}{2}\sin\left(\frac{5\pi}{4}(x_1+2)\right)\right)^T
$$
is compatible (isometric immersion) with the metric and we impose the corresponding Dirichlet boundary conditions on $\Gamma_D$ as in Section \ref{sec:one_mode}.

Using the same setup as in Section \ref{sec:one_mode}, Algorithm \ref{algo:preprocess} produced a suitable initial guess in 1271 iterations, while Algorithm \ref{algo:main_GF} terminated after 1833 steps. The deformations obtained after each of the three main procedures are given in Figure \ref{fig:cylinder_two_modes}. The corresponding energy and prestrain defect are reported in Table \ref{tab:cylinder_two_modes}. We see that the main gradient flow decreases the energy upon bending the shape but keeping the metric defect roughly constant.

\begin{figure}[htbp]
	\begin{center}
		\includegraphics[width=4.1cm]{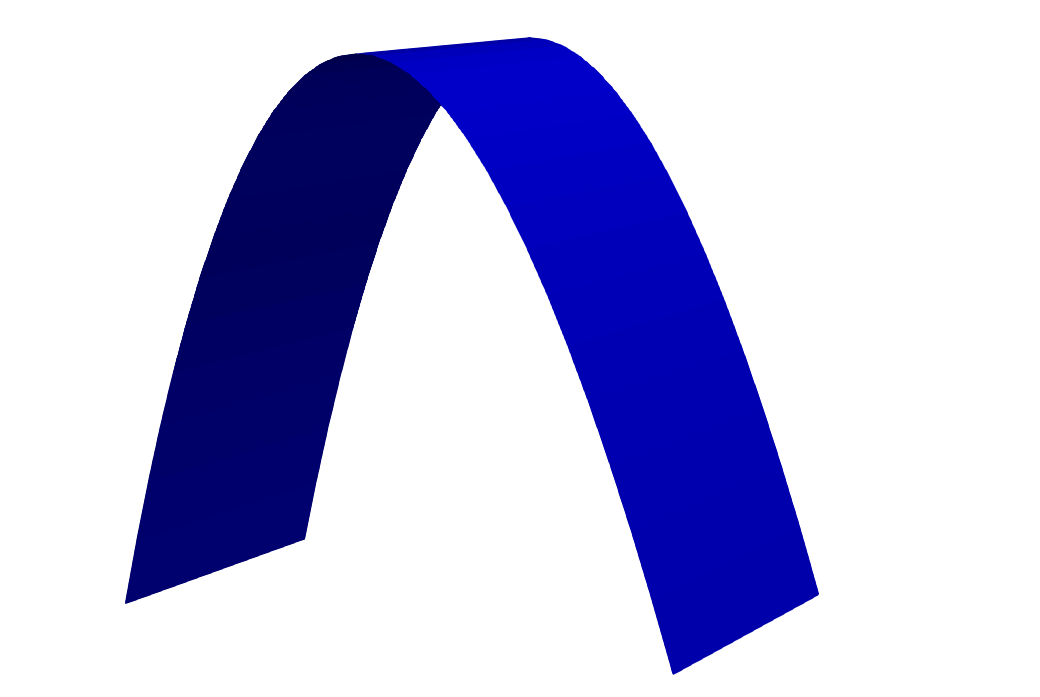}
		\includegraphics[width=4.1cm]{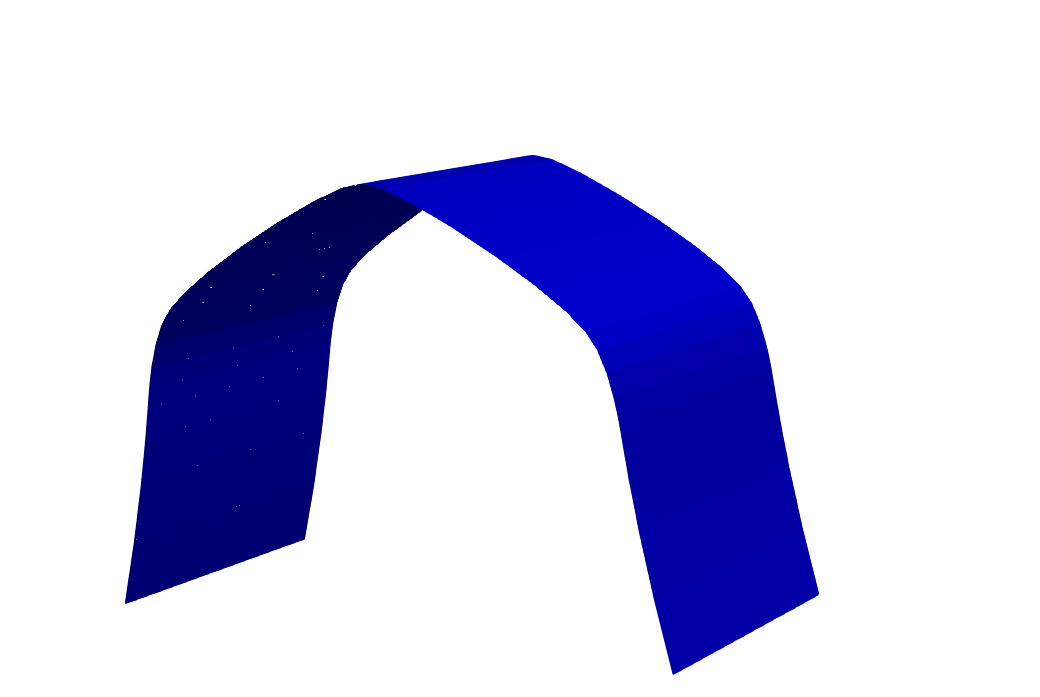}
		\includegraphics[width=4.1cm]{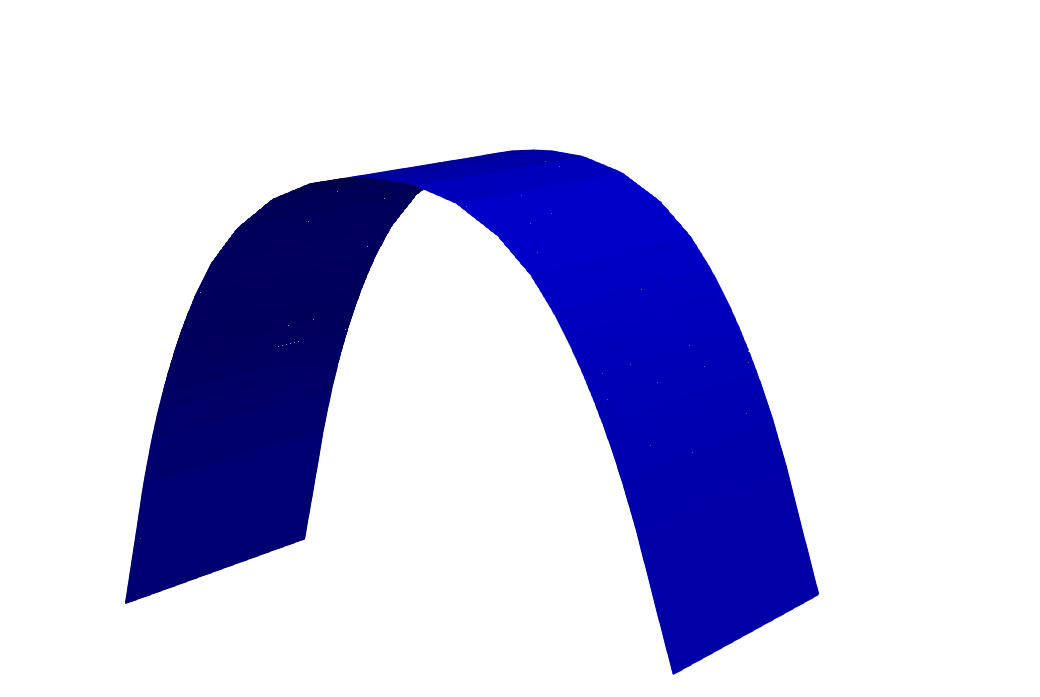}
		\caption{Deformed plate for the cylinder metric with two modes. Left: BC PP; middle: Metric PP; right: Final.
		Compare with Figure~\ref{fig:cylinder_one_mode} corresponding to the metric \eqref{def:metric_one_mode} (one mode).} \label{fig:cylinder_two_modes}
	\end{center}
\end{figure}

\begin{table}[htbp]
	\begin{center}
		\begin{tabular}{|c|c|c|c|c|}
			\cline{2-5}
			\multicolumn{1}{c|}{ } & Initial & BC PP & Metric PP & Final \\
			\hline
			$E_h$ & 413.7400 & 5.5344 & 28.9184 & 13.0706 \\
			\hline
			$D_h$ & 25.2909 & 26.1854 & 9.9997E-2 & 1.0178E-1 \\
			\hline
		\end{tabular}
                \vspace{0.3cm}
		\caption{Energy and metric defect for the cylinder metric with two modes. Compare with Table~\ref{tab:cylinder_one_mode} corresponding to one mode.} \label{tab:cylinder_two_modes}
	\end{center}	
\end{table}

\subsection{Rectangle with a \emph{catenoidal-helicoidal} metric}
Let $\Omega$ be a rectangle to be specified later and let the metric be
\begin{equation} \label{def:g_cate_heli}
g(x_1,x_2) =
\begin{bmatrix}
\cosh(x_2)^2 & 0 \\ 0 & \cosh(x_2)^2 
\end{bmatrix}.
\end{equation}
Notice that the family of deformations $\vy^\alpha:\Omega\rightarrow\mathbb{R}^3$, $0\leq \alpha \leq \frac{\pi}{2}$, defined by
\begin{equation} \label{def_y_alpha}
\vy^{\alpha}:=\cos(\alpha)\bar \vy + \sin(\alpha)\tilde{\vy}
\end{equation}
with
\begin{equation*}
\bar \vy(x_1,x_2)=
\begin{bmatrix}
\sinh(x_2)\sin(x_1) \\ -\sinh(x_2)\cos(x_1) \\ x_1 
\end{bmatrix},
\quad
\tilde \vy(x_1,x_2)=
\begin{bmatrix}
\cosh(x_2)\cos(x_1) \\ \cosh(x_2)\sin(x_1) \\ x_2 
\end{bmatrix},
\end{equation*}
are all compatible with the metric \eqref{def:g_cate_heli}. The parameter $\alpha=0$  corresponds to an {\it helicoid} while $\alpha=\pi/2$ represents a {\it catenoid}. Furthermore, the energy $E[\vy^\alpha]$ defined in \eqref{def:Eg_D2y} (or equivalently $E[\vy^\alpha]$ given in \eqref{E:final-bending}) has the same value for all $\alpha$. To see this, it suffices to note that the second fundamental form of $\vy^\alpha$ is given by
\begin{equation*}
\II[\vy^\alpha] = 
\begin{bmatrix}
-\cos(\alpha) & \sin(\alpha) \\ \sin(\alpha) & \cos(\alpha)
\end{bmatrix},
\quad  D^2y^\alpha_k=\cos(\alpha)D^2\bar y_k+\sin(\alpha)D^2\tilde y_k,
\end{equation*}
where $y^\alpha_k=(\vy^\alpha)_k$ is the $k$th component of $\vy^\alpha$ for $k=1,2,3$.

In the following sections, we show how the two extreme deformations can be obtained either by imposing the adequate boundary conditions or by starting with an initial configuration sufficiently close to the energy minima.

\subsubsection{\bf \emph{Catenoid case}}
We consider the domain $\Omega=(0,6.25)\times(-1,1)$. The mesh $\Th$ consists of 896 (almost square) rectangular cells of diameter $h_{\K}=h\approx 0.17$ (26880 DoFs).
We do not impose any boundary conditions on the deformations, which corresponds to $\Gamma_D=\emptyset$ (free boundary condition). We apply Algorithm \ref{algo:preprocess} (initialization) and start the metric preprocessing with $\widetilde{\vy}_h^0=\widehat \vy_h$, the solution to the bi-Laplacian problem \eqref{bi-Laplacian} with fictitious force $\widehat \vf = (0,0,4)^T$ and boundary condition $\vvarphi(\vx)=(\vx,0)$ on $\partial\Omega$ (but without $\Phi$). Moreover, we use three tolerances $\widetilde{tol} = 0.1, \, 0.025, \, 0.01$ for this preprocessing to investigate the effect on Algorithm \ref{algo:main_GF} (gradient flow). Figure \ref{fig:catenoid_2} depicts final configurations produced by Algorithm \ref{algo:main_GF} with the outputs of Algorithm \ref{algo:preprocess}. Corresponding energies and metric defects are given in Table \ref{tab:catenoid}. We see that the metric defect diminishes, as $\widetilde{tol}$ decreases, and the surface tends to a full (closed) catenoid as expected from the relation \eqref{def_y_alpha} with $\alpha =\pi/2$.

\begin{figure}[htbp]
	\begin{center}
	\begin{tabular}{ccc}
		\includegraphics[width=4.0cm]{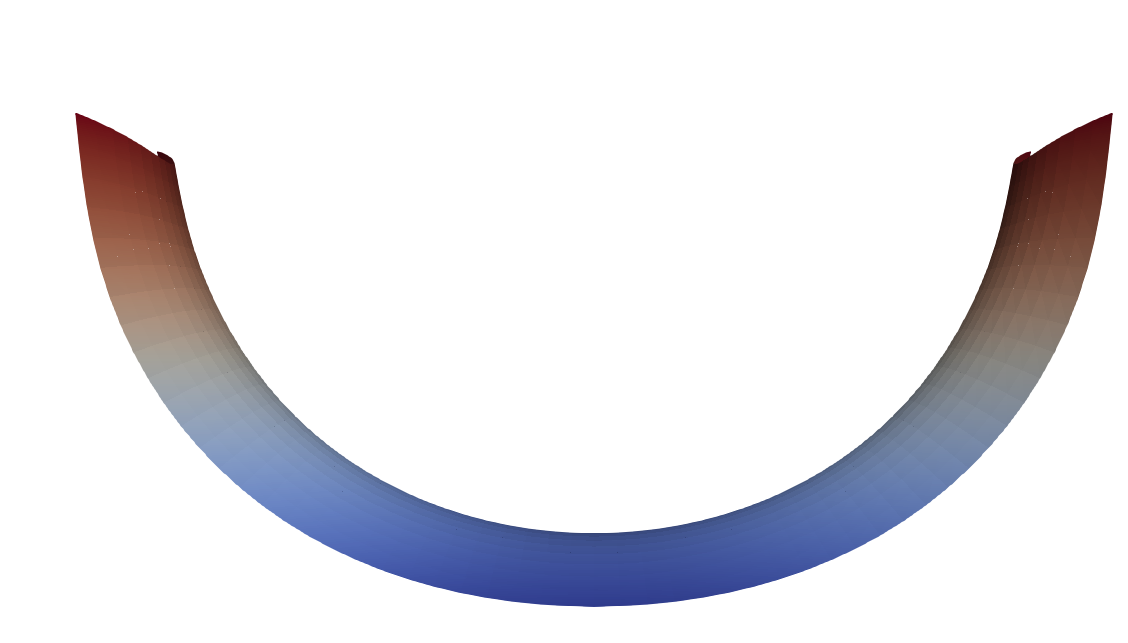}&
		\includegraphics[width=4.0cm]{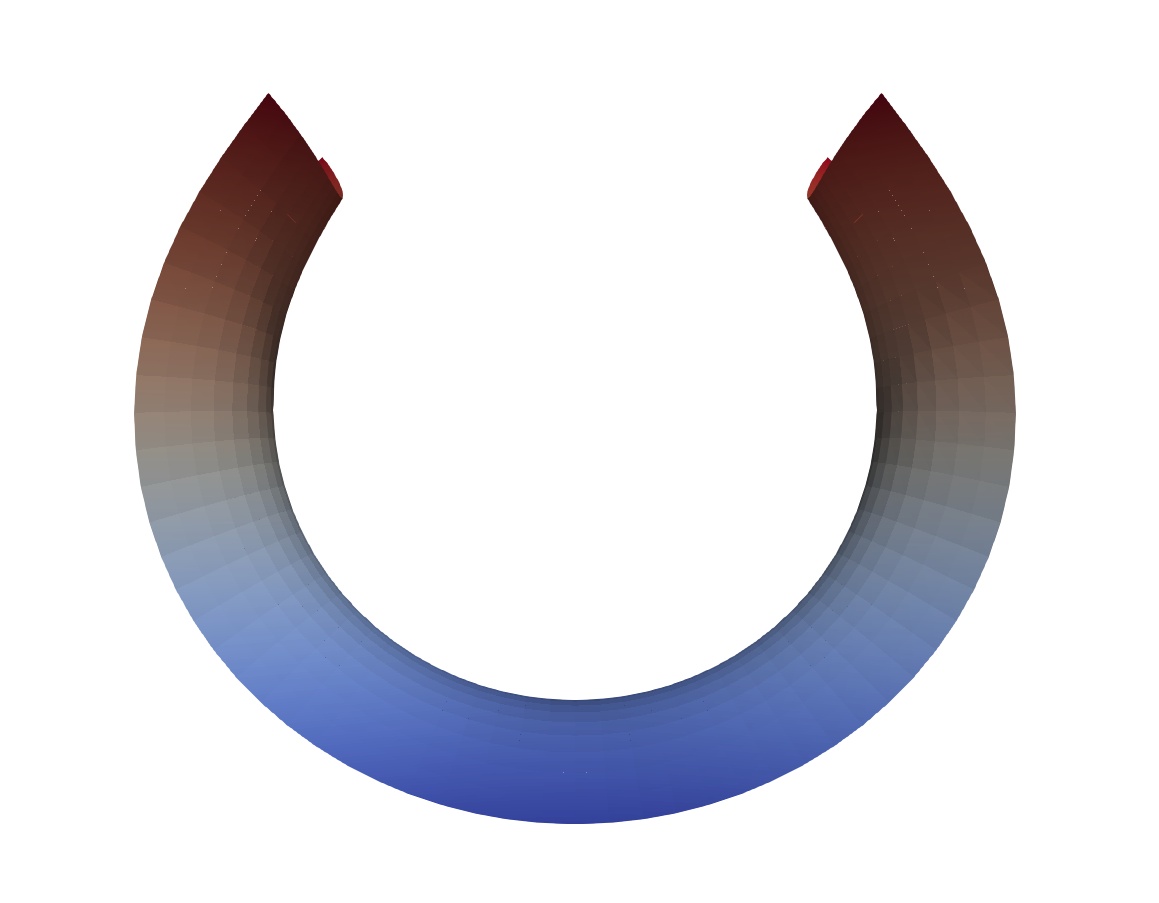}&
		\includegraphics[width=4.0cm]{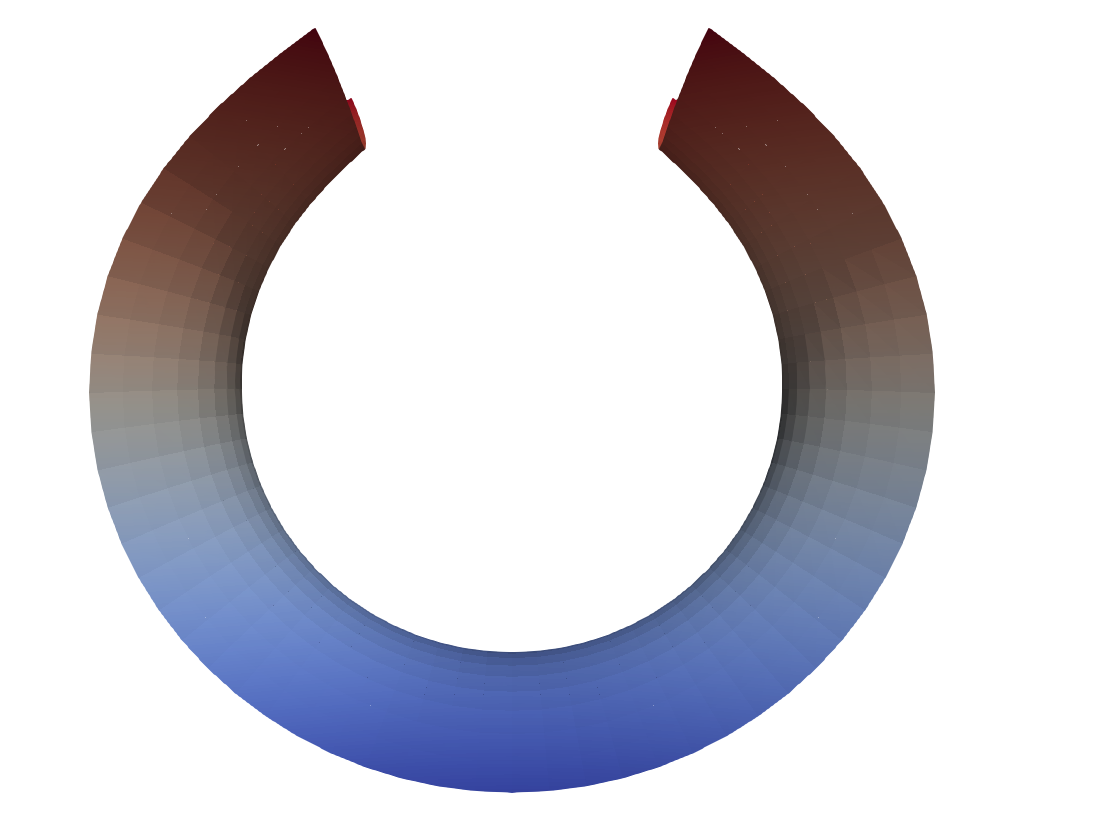}\\
		\includegraphics[width=4.0cm]{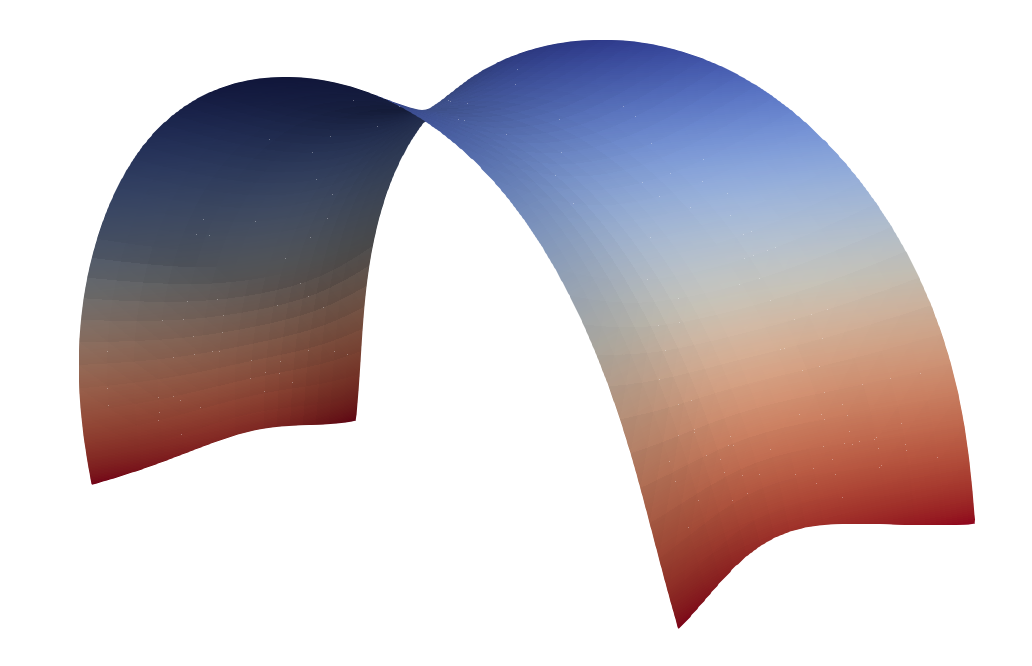}&
		\includegraphics[width=4.0cm]{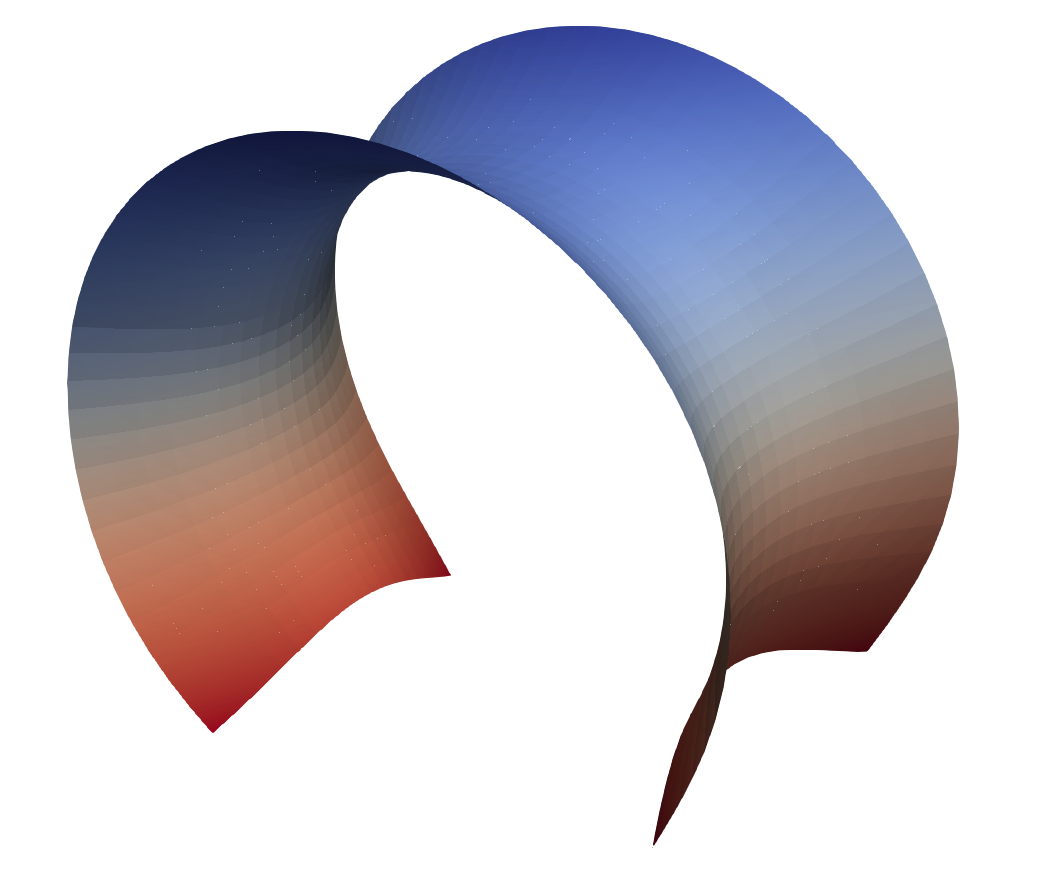}&
		\includegraphics[width=4.0cm]{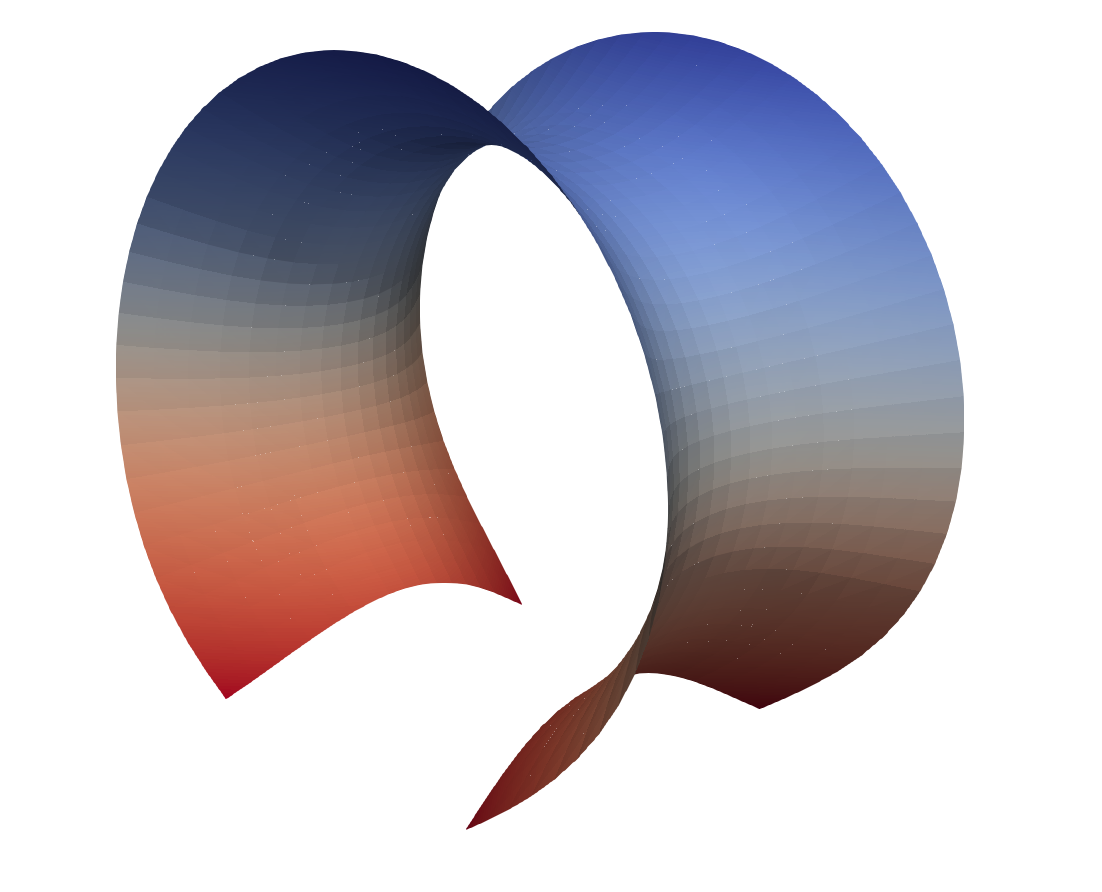}
		\end{tabular}
		\caption{Final configurations for the \emph{catenoidal-helicoidal} metric with free boundary condition using tolerances $\widetilde{tol}= 0.1$ (left), 0.025 (middle) and 0.01 (right) for the metric preprocessing of Algorithm \ref{algo:preprocess}. The second row offers a different view of the final deformations.}\label{fig:catenoid_2}
	\end{center}
\end{figure}

\begin{table}[htbp]
	\begin{center}
		\begin{tabular}{|c|c|c|c|c|c|c|}
			\cline{2-7}
			\multicolumn{1}{c}{ } & \multicolumn{2}{|c|}{$\widetilde{tol}=0.1$} & \multicolumn{2}{|c|}{$\widetilde{tol}=0.025$} & \multicolumn{2}{|c|}{$\widetilde{tol}=0.01$} \\
			\cline{2-7}
			\multicolumn{1}{c|}{ } & Algo 2 & Algo 1 & Algo 2 & Algo 1 & Algo 2 & Algo 1\\
			\hline
			$E_h$ & 36.9461 & 4.01094 & 103.838 & 7.42946 & 146.215 & 8.78622\\
			\hline
			$D_h$ & 2.62428 & 3.19839 & 1.36864 & 2.69258 & 0.853431 & 1.83427\\
			\hline
		\end{tabular}
                \vspace{0.3cm}
		\caption{Energies $E_h$ and metric defects $D_h$ produced by Algorithms \ref{algo:preprocess} and \ref{algo:main_GF} for the \emph{catenoidal-helicoidal} metric with free boundary condition. We see that the tolerance $\widetilde{tol}$ of Algorithms \ref{algo:preprocess} controls $D_h$ and that Algorithm \ref{algo:main_GF} does not increase $D_h$ much but reduces $E_h$ substantially. The smaller $\widetilde{tol}$ is the closer the computed surface gets to the catenoid, which is closed (see Figure \ref{fig:catenoid_2}).} \label{tab:catenoid}
	\end{center}	
\end{table}

\subsubsection{\bf \emph{Helicoid shape}}
All the deformations $\vy^\alpha$ in \eqref{def_y_alpha} are global minima of the energy but the final deformation is not always catenoid-like as in the previous section. 
In fact, starting with an initial deformation close to $\vy^\alpha$ with $\alpha=0$ leads to an helicoid-like shape. We postpone such an approach to Section~\ref{ss:osc}. An alternative to achieve an helicoid-like shape is to enforce the appropriate boundary conditions as described now.

We consider the domain $\Omega=(0,4.5)\times(-1,1)$ and enforce Dirichlet boundary conditions on $\Gamma_D=\{0\}\times (-1,1)$ compatible with $\vy^\alpha$ given by \eqref{def_y_alpha} with $\alpha=0$.  
The mesh $\Th$ consists of 640 (almost square) rectangular cells of diameter $h_{\K}=h\approx 0.17$ (19200 DoFs) and the pseudo time-step is $\tau = 0.01$.

We apply Algorithm \ref{algo:preprocess} (preprocessing) with $\widetilde \tau = 0.01$, $\widetilde\eps_0=0.1$ and $\widetilde{tol}=10^{-3}$ to obtain the initial deformation $\vy_h^0$. The preprocessing stopped after 2555 iterations, meeting the criteria $\widetilde \tau^{-1}|\widetilde E_h[\tilde \vy_h^{n+1}]-\widetilde E_h[\tilde \vy_h^n]|\leq \widetilde{tol}$, while 2989 iterations of Algorithm \ref{algo:main_GF} (gradient flow) were needed to reach the stationary deformation.
Figure \ref{fig:helicoid} displays the output of the boundary conditions preprocessing and the metric preprocessing, the two stages of Algorithm \ref{algo:preprocess}, as well as two views of the output of Algorithm \ref{algo:main_GF}. The corresponding energies and metric defects are reported in Table \ref{tab:helicoid}.

\begin{figure}[htbp]
	\begin{center}
		\includegraphics[width=3.0cm]{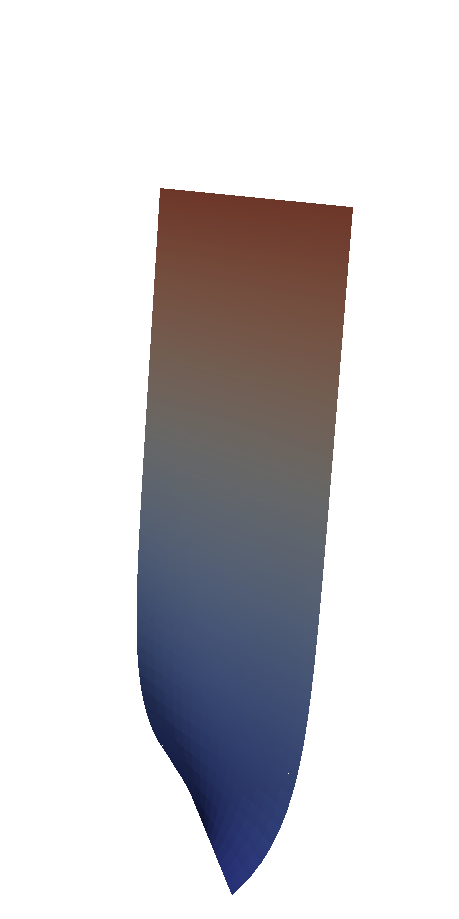}
		\includegraphics[width=3.0cm]{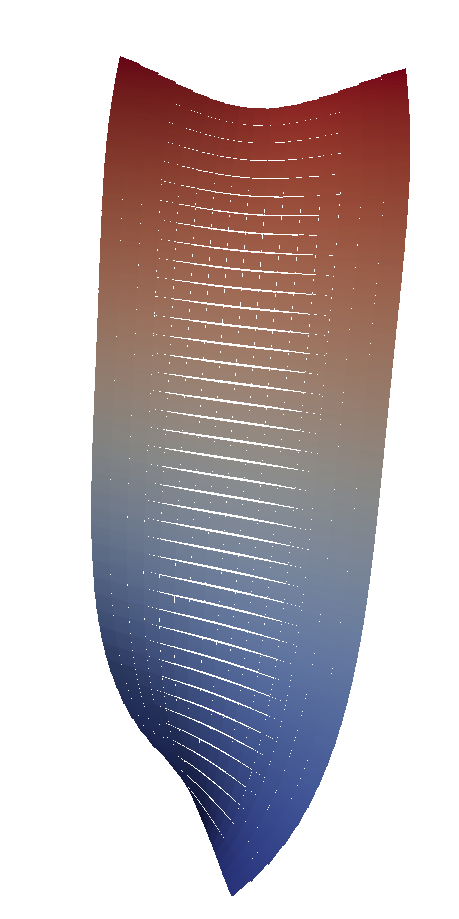}
		\includegraphics[width=3.0cm]{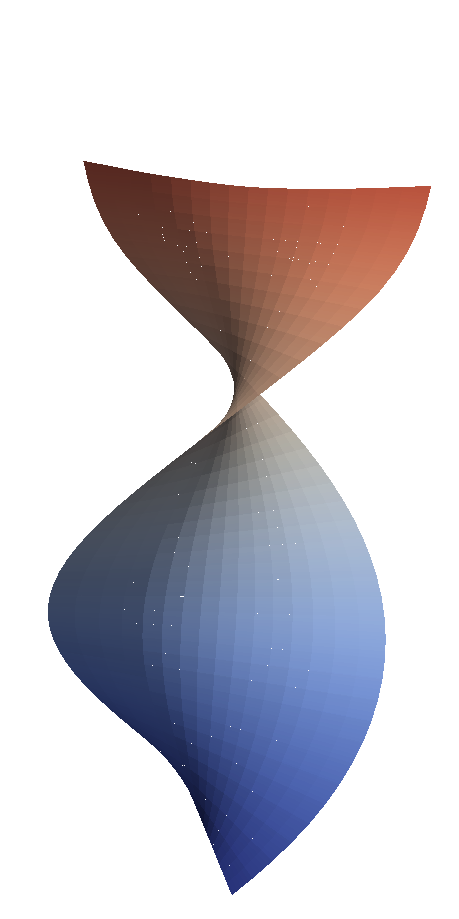}
		\includegraphics[width=3.0cm]{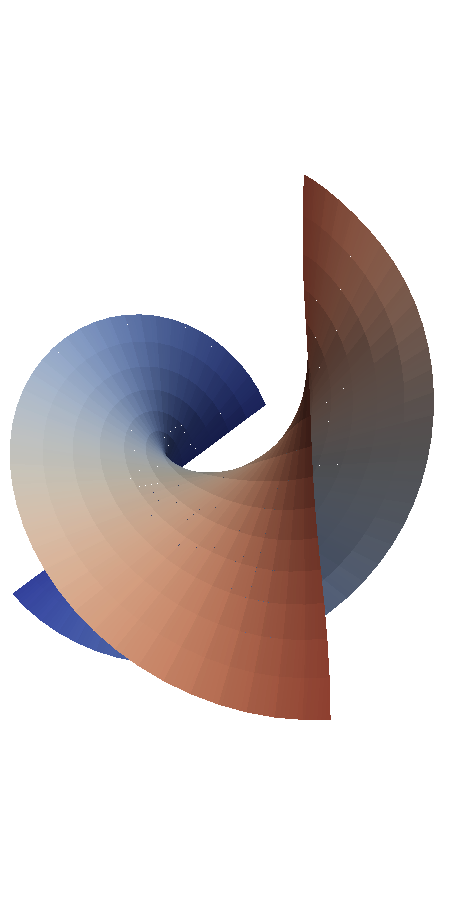} 
		\caption{Deformed plate for the \emph{catenoidal-helicoidal} with Dirichlet boundary conditions on the bottom side corresponding to $\{0\}\times (-1,1)$. From left to right: BC PP, Metric PP, and two views (the last from the top) of the output of Algoritm \ref{algo:main_GF}.} \label{fig:helicoid}
	\end{center}
\end{figure}

\begin{table}[htbp]
	\begin{center}
		\begin{tabular}{|c|c|c|c|c|}
			\cline{2-5}
			\multicolumn{1}{c|}{ } & Initial & BC PP & Metric PP & Final \\
			\hline
			$E_h$ & 138020 & 0.658342 & 202.144 & 7.7461 \\
			\hline
			$D_h$ & 5.17664 & 5.16565 & 0.248419 & 1.15764 \\
			\hline
		\end{tabular}
                \vspace{0.3cm}
		\caption{Energies and metric defects for the helicoid-like shape with
                Dirichlet boundary conditions on the bottom side.} \label{tab:helicoid}
	\end{center}	
\end{table}
   
\subsection{Disc with positive or negative Gaussian curvature} \label{sec:disc}
We now consider a plate consisting of a disc of radius $1$
\begin{equation*}
\Omega= \Big\{(x_1,x_2)\in\mathbb{R}^2: \quad x_1^2+x_2^2<1 \Big\}.
\end{equation*}
We prescribe several immersible metrics $g$ and impose no boundary conditions.

The mesh $\Th$ consists of $320$ quadrilateral cells of diameter $0.103553 \leq h_{\K}\leq 0.208375$ (9600 DoFs) and the pseudo time-step is $\tau = 0.01$. Moreover, we initialize the metric preprocessing of Algorithm \ref{algo:preprocess} with the identity function $\widetilde \vy_h^0(\vx)=(\vx,0)^T$ for $\vx\in\Omega$, and $\widetilde\tau=0.05$, $\widetilde\eps_0=0.1$, $\widetilde{tol}=10^{-6}$.
   
\subsubsection{\bf \emph{Bubble - positive Gaussian curvature}}
To obtain a bubble-like shape, we consider for any $\alpha>0$ the metric
\begin{equation}
g(x_1,x_2) =
\begin{bmatrix}
1+\alpha\frac{\pi^2}{4}\cos\left(\frac{\pi}{2}(1-r)\right)^2\frac{x_1^2}{r^2} & \alpha\frac{\pi^2}{4}\cos\left(\frac{\pi}{2}(1-r)\right)^2\frac{x_1x_2}{r^2} \\
\alpha\frac{\pi^2}{4}\cos\left(\frac{\pi}{2}(1-r)\right)^2\frac{x_1x_2}{r^2} & 1+\alpha\frac{\pi^2}{4}\cos\left(\frac{\pi}{2}(1-r)\right)^2\frac{x_2^2}{r^2}
\end{bmatrix}
\end{equation}
with $r:=\sqrt{x_1^2+x_2^2}$.
A compatible deformation is given by
\begin{equation*}
\vy(x_1,x_2)=\left(x_1,x_2,\sqrt{\alpha}\sin\left(\frac{\pi}{2}(1-r)\right)\right)^T,
\end{equation*}
i.e., $\vy$ is an isometric immersion $\I[\vy]=g$.
In the following, we choose $\alpha=0.2$.

In the absence of boundary conditions and forcing term, the flat configuration $\tilde \vy_h^0(\Omega)=\Omega$ has zero energy but has a metric defect of $D_h= 1.0857$. Algorithm \ref{algo:preprocess} (preprocessing) performs 877 iterations to deliver an energy $E_h=35.3261$ and a metric defect $D_h=0.0999797$. Algorithm \ref{algo:preprocess} only stretches the plate which remains flat; see Figure \ref{fig:bubble} (left and middle). Algorithm \ref{algo:main_GF} (gradient flow) then deforms the plate out of plane, and reaches a stationary state after 918 iterations with $E_h=2.08544$, while keeping the metric defect $D_h=0.087839$; see Figure \ref{fig:bubble}-right. 

We point out that the discussion after \eqref{E:coercive} also applies to Algorithm \ref{algo:main_GF}, i.e., a flat initial configuration ($y_3=0$) will theoretically lead to flat deformations throughout the gradient flow. However, in this example and the ones in Section \ref{S:gel-discs}, the the initial deformation produced by Algorithm \ref{algo:preprocess} has a non-vanishing third component $y_3$ (order of machine precision). Furthermore, Algorithm \ref{algo:preprocess} may also produce discontinuous configurations (as for the initial deformation in Figure~\ref{fig:bubble} left and middle) to accommodate for the constraint and will thus have a relatively large energy due to the jump penalty term. 
These two aspects combined may be responsible for the main gradient flow Algorithm~\ref{algo:main_GF} to produce out of plane deformations even when starting with a theoretical flat initial configuration. This is the case when starting with a disc with positive Gaussian curvature metric as in Figure~\ref{fig:bubble}.

\begin{figure}[htbp]
	\begin{center}
	  \begin{tabular}{cc||c}
		\includegraphics[width=4.1cm]{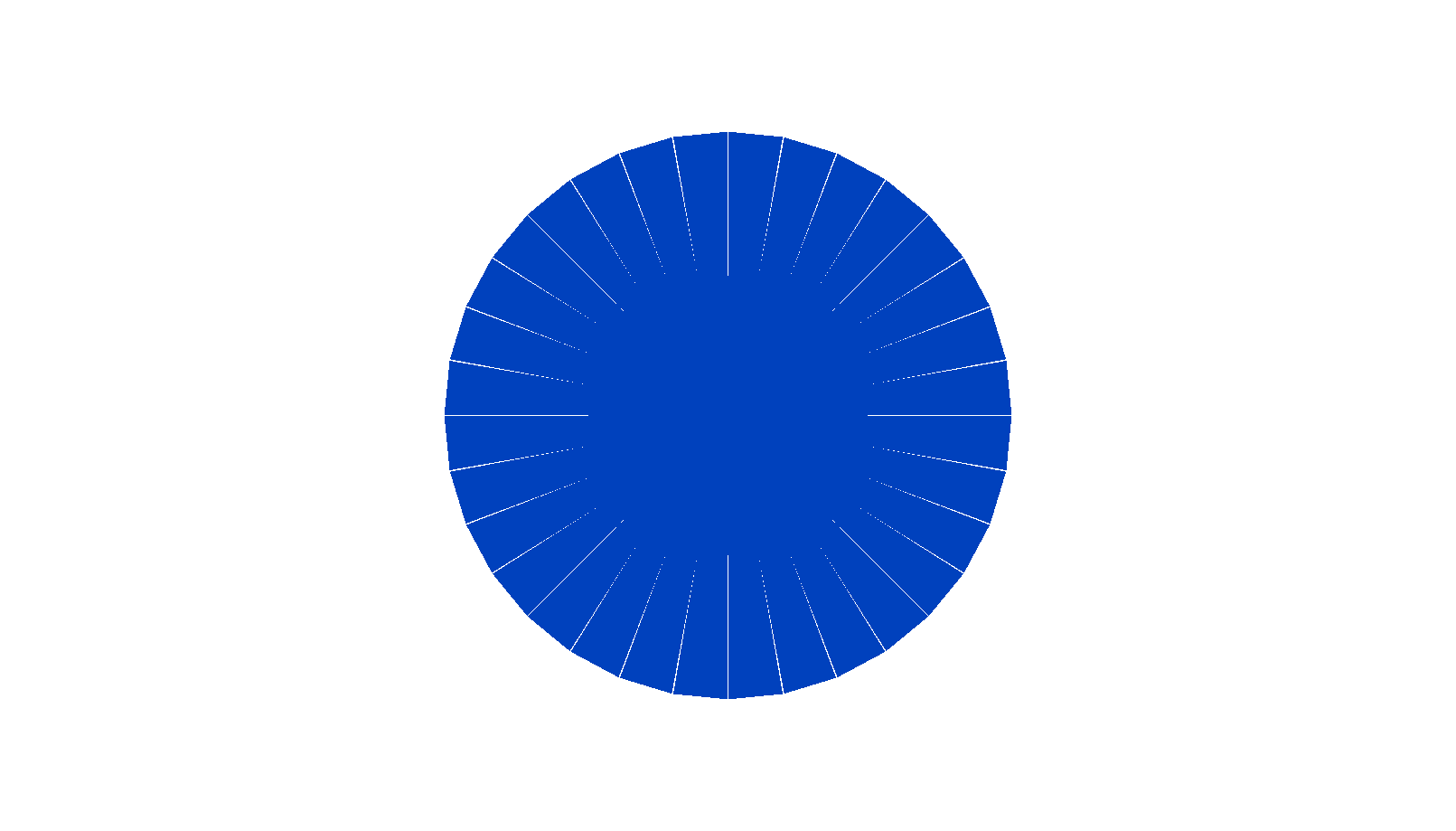}&
		\includegraphics[width=4.1cm]{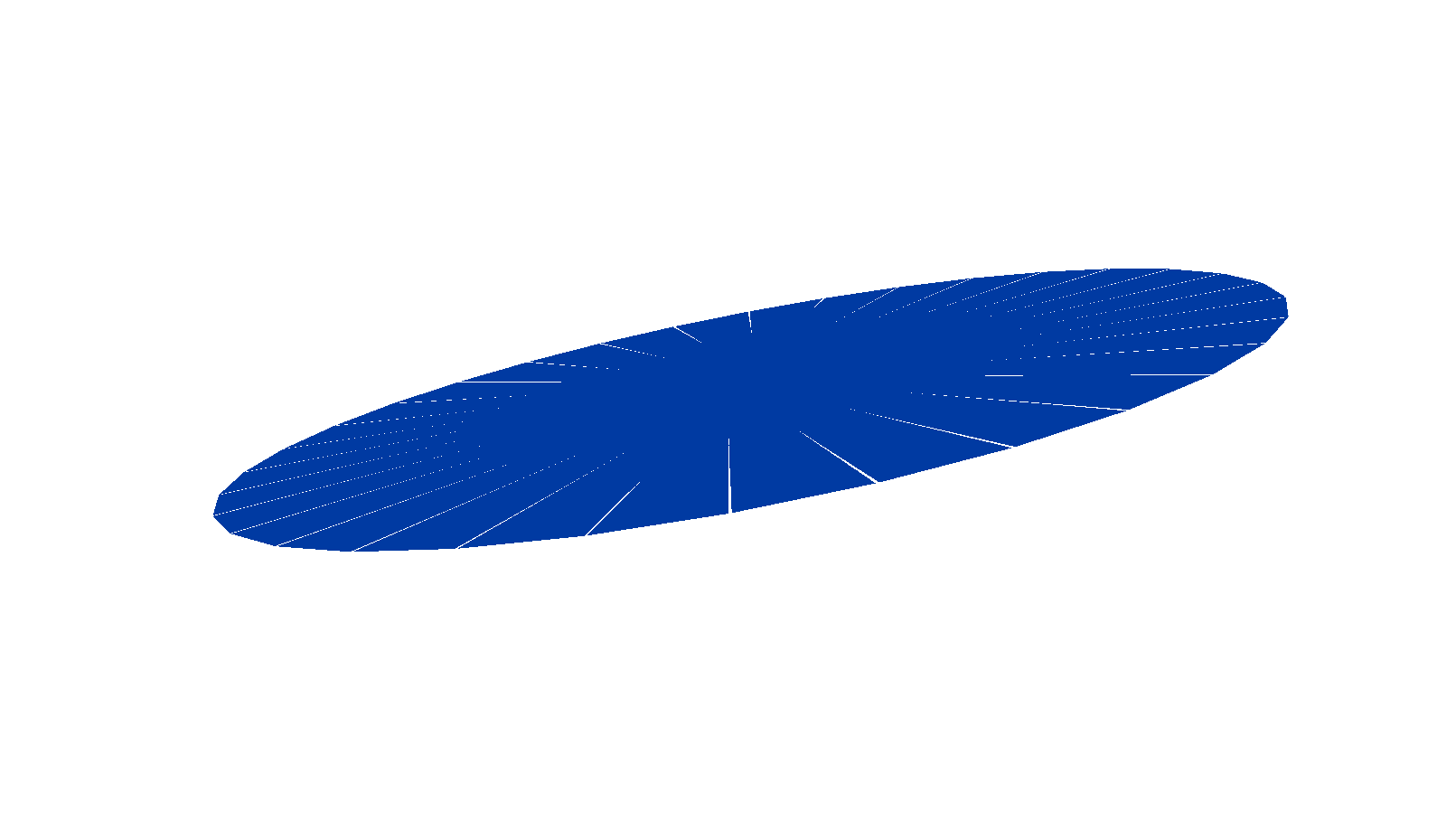}&
		\includegraphics[width=4.1cm]{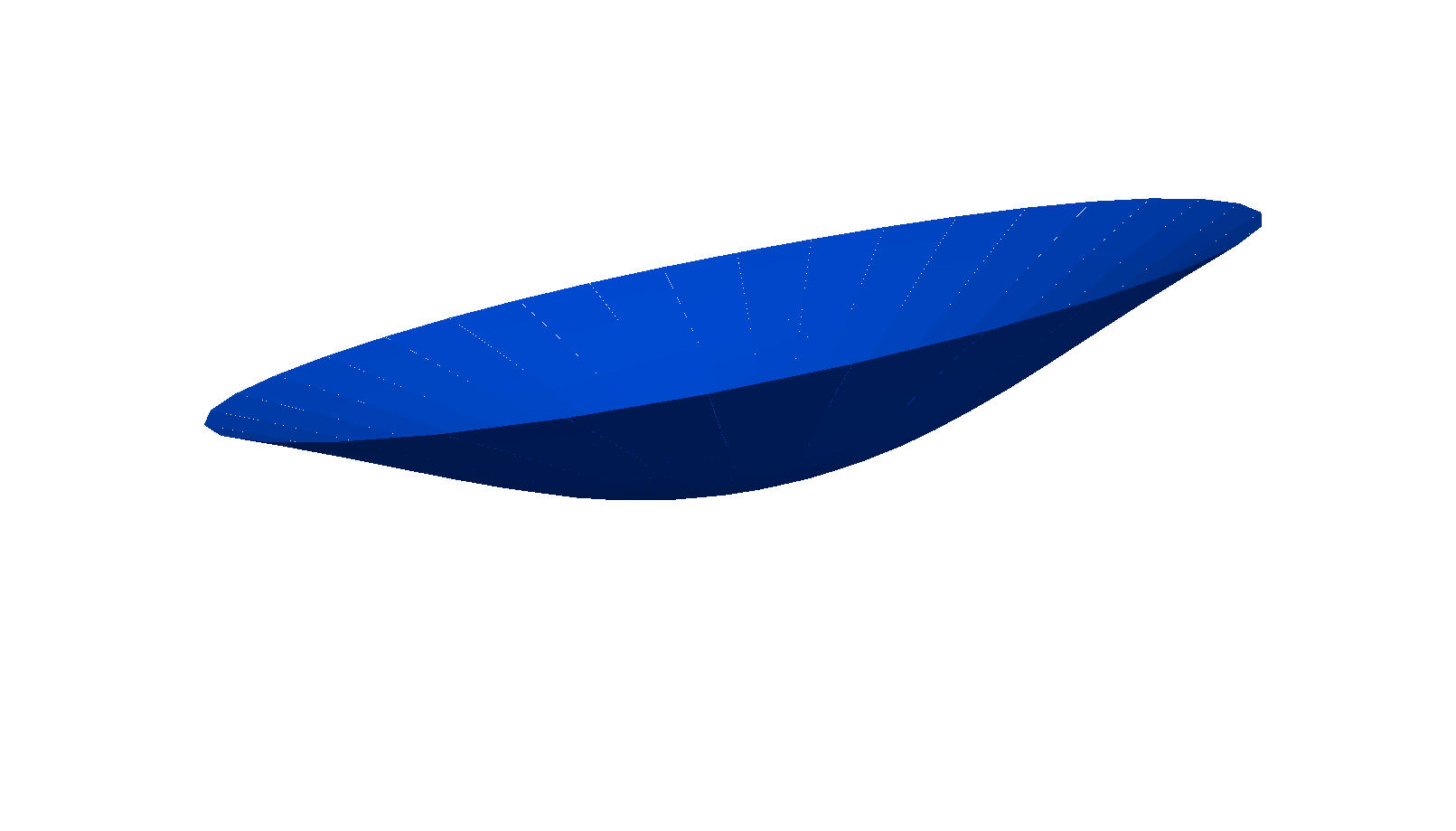}
	\end{tabular}
	\caption{Deformed plate for the disc with positive Gaussian curvature metric. Algorithm \ref{algo:preprocess} stretches the plate but keeps it flat (left and middle). Algorithm \ref{algo:main_GF} gives rise to an ellipsoidal shape (right).} \label{fig:bubble}
	\end{center}
\end{figure}

\subsubsection{\bf \emph{Hyperbolic paraboloid - negative Gaussian curvature}}
We consider the immersible metric $g$ with negative Gaussian curvature
\begin{equation}
g(x_1,x_2) =
\begin{bmatrix}
1+x_2^2 & x_1x_2 \\
x_1x_2 & 1+x_1^2
\end{bmatrix}.
\end{equation}
A compatible deformation is given by $\vy(x_1,x_2)=(x_1,x_2,x_1x_2)^T$, i.e.,
$\I[\vy]=g$.

In this setting, the flat configuration has a prestrain defect of $D_h= 1.56565$ (still vanishing energy). Algorithm \ref{algo:preprocess} (preprocessing) performs 856 iterations to reach the energy $E_h= 50.3934$ and metric defect $D_h=0.0999757$. Algorithm \ref{algo:main_GF} (gradient flow) executes 1133 iterations to deliver an energy $E_h=1.83112$ and metric defect $D_h=0.0980273$. Again, the metric defect remains basicallly constant throughout the main gradient flow, while the energy is significantly decreased. Figure \ref{fig:hyper_para} shows the initial (left) and final (middle) deformations of Algoritm \ref{algo:preprocess} and the output of Algorithm \ref{algo:main_GF} (right) which exhibit a sadlle point structure.

\begin{figure}[htbp]
	\begin{center}
	\begin{tabular}{cc||c}
		\includegraphics[width=4.1cm]{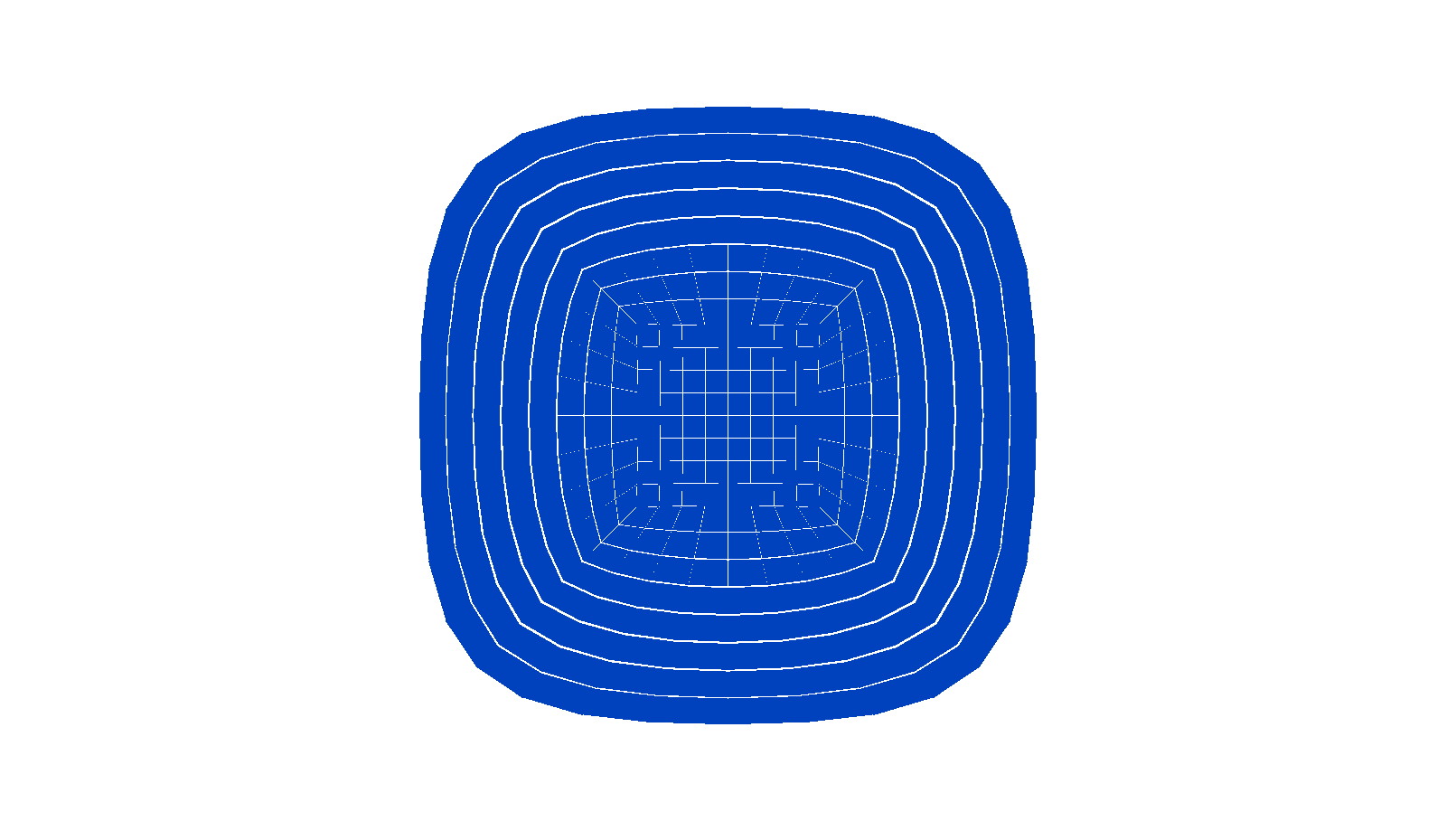}&
		\includegraphics[width=4.1cm]{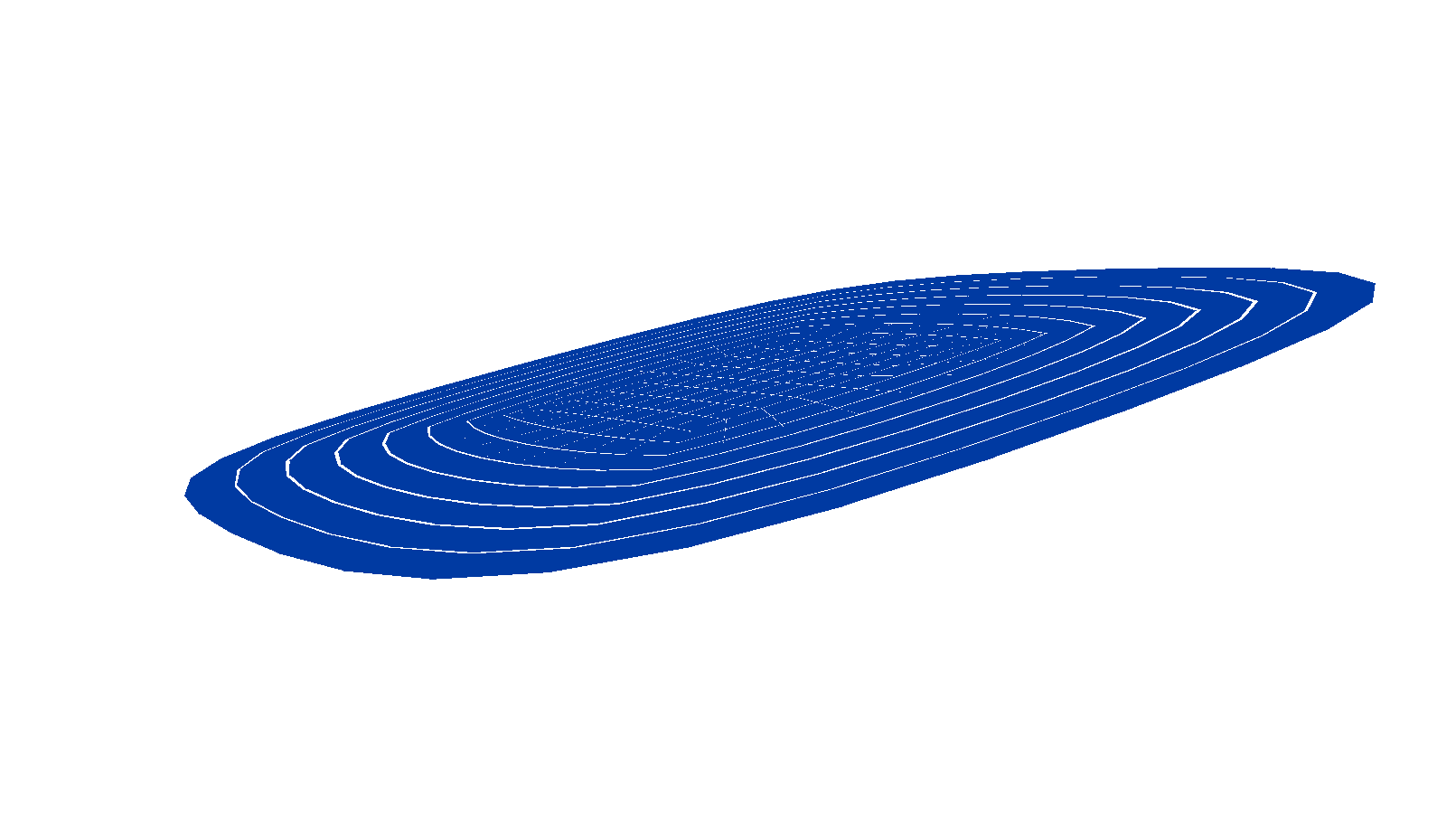}&
		\includegraphics[width=4.1cm]{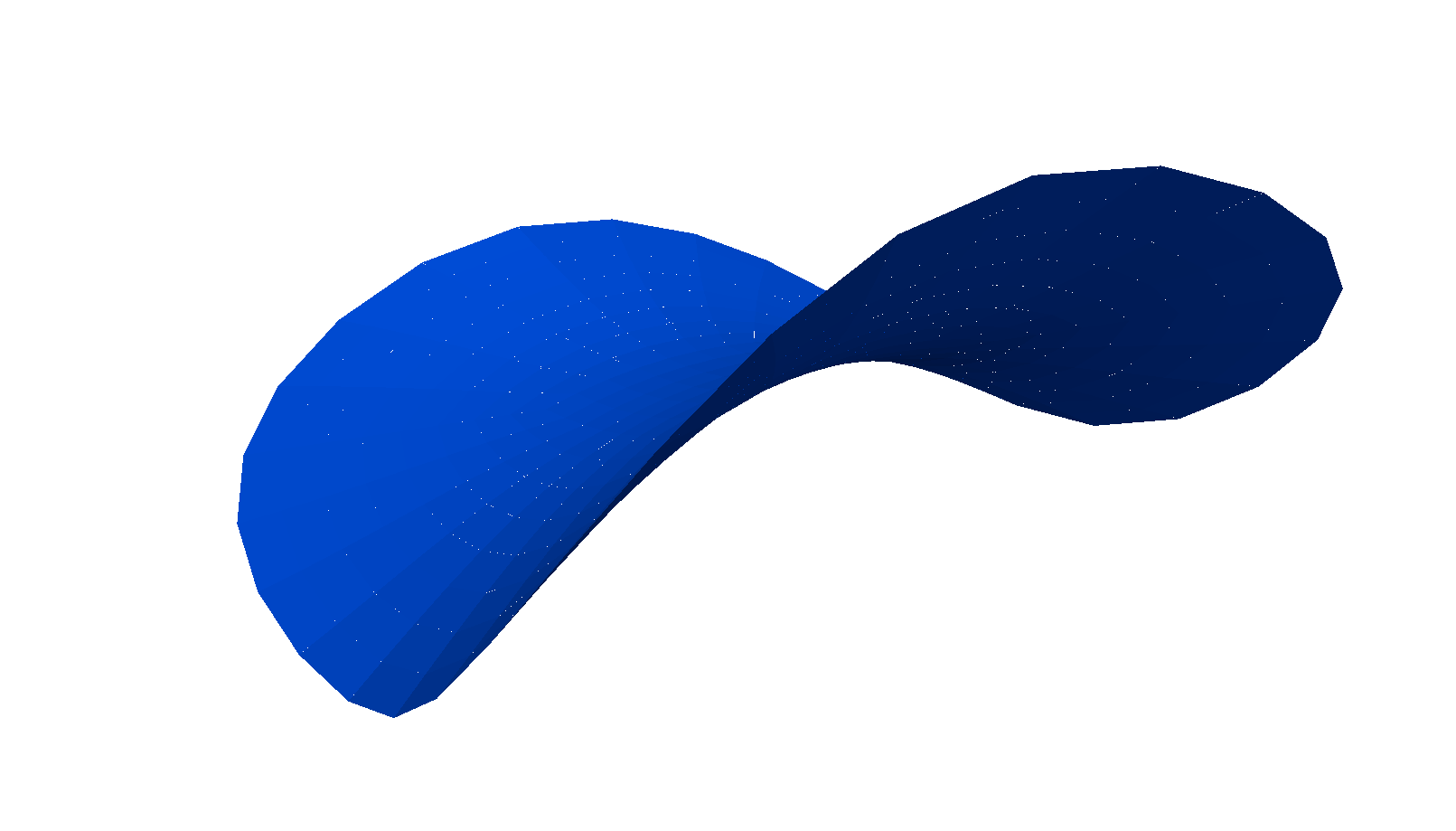}
		\end{tabular}
	\caption{Deformed plate for the disc with negative Gaussian curvature. Algorithm \ref{algo:preprocess} stretches the plate but keeps it flat (left and middle). Algoritm \ref{algo:main_GF} gives rise to a saddle shape (right). Compare with Figure~\ref{fig:bubble}. } \label{fig:hyper_para}
	\end{center}
\end{figure}

We point out that Algoritm \ref{algo:preprocess} gives rise to little gaps between elements of the deformed subdivisions as a consequence of not including jump stabilization terms in the bilinear form \eqref{pre-gf:bilinear}. These gaps are reduced by Algorithm \ref{algo:main_GF}.

\subsubsection{\bf \emph{Oscillating boundary}}\label{ss:osc}
We construct an immersible metric in polar coordinates $(r,\theta)$ with a six-fold oscillation near the boundary of the disc $\Omega$. Let $\widetilde{g}(r,\theta)=\I[\widetilde \vy(r,\theta)]$ be the first fundamental form of the deformation
\begin{equation} \label{def:tilde_y_polar}
\widetilde \vy(r,\theta)= \big(r\cos(\theta),r\sin(\theta),0.2r^4\sin(6\theta)\big).
\end{equation}
The expression of the prestrain metric $g=\I[\vy]$ in Cartesian coordinates
is then given by \eqref{E:Jacobian} and $\vy(x_1,x_2) = \widetilde \vy(r,\theta)$.

We set the parameters
\[
\tau=0.05,\quad
\widetilde\tau=0.05, \quad
\widetilde\eps_0=0.1, \quad
\widetilde{tol}=10^{-4}, \quad
tol=10^{-6},
\]
and note that Algorithm \ref{algo:main_GF} (gradient flow) does not necessarily stop at global minima of the energy. Local extrema are frequently achieved and they are, in fact, of particular interest in many applications. To illustrate this property, we consider a couple of initial deformations and run Algorithms \ref{algo:preprocess} and \ref{algo:main_GF}.

\smallskip\noindent
\textbf{Case 1: boundary oscillation.} We choose $\widetilde\vy_h^0$ to be the local nodal interpolation of $\vy=\widetilde \vy\circ\vpsi$ into $[\V^k_h]^3$, with $\widetilde\vy$ given by \eqref{def:tilde_y_polar}. The output deformations of Algorithms \ref{algo:preprocess} and \ref{algo:main_GF} are depicted in Figure \ref{fig:oscillation_boundary}. The former becomes the initial configuration $\vy_h^0$ of Algorithm \ref{algo:main_GF} and is almost the same as $\widetilde\vy_h^0$, which is approximately a disc with six-fold oscillations; see Figure \ref{fig:oscillation_boundary} (left). This is due to the fact that $\I[\widetilde\vy_h^0]$ is already close to the target metric $g$. Algorithm \ref{algo:main_GF} (gradient flow) breaks the symmetry: two peaks are amplified while the other four are reduced. After the preprocessing, the energy is $E_h=18.0461$ and metric defect is $D_h=0.00208473$. The final energy is $E_h=13.6475$ while the final metric defect is $D_h=0.00528294$. 

\begin{figure}[htbp]
	\begin{center}
		\includegraphics[width=3cm]{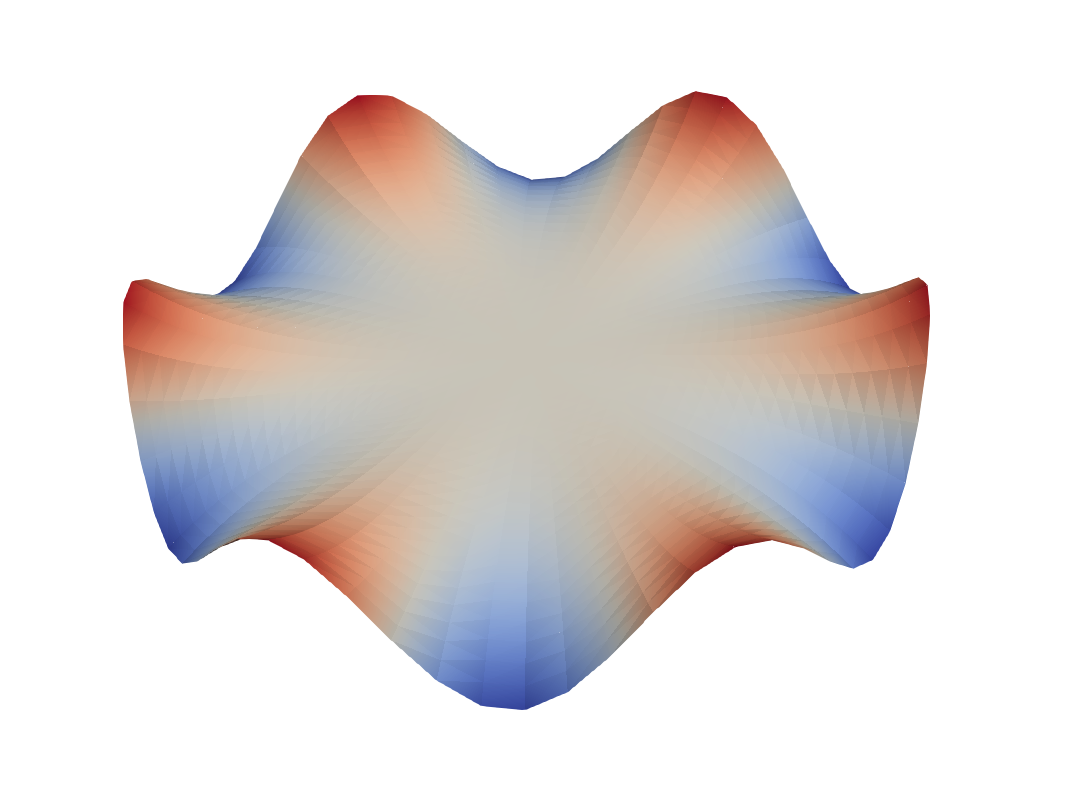}
		\includegraphics[width=3cm]{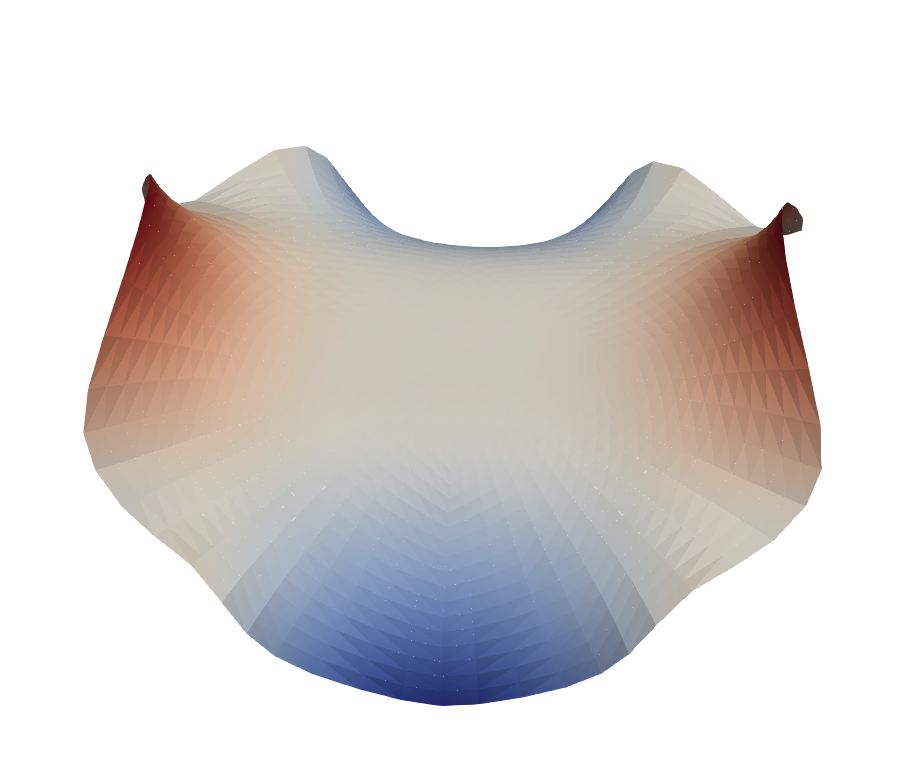}
		\includegraphics[width=4cm]{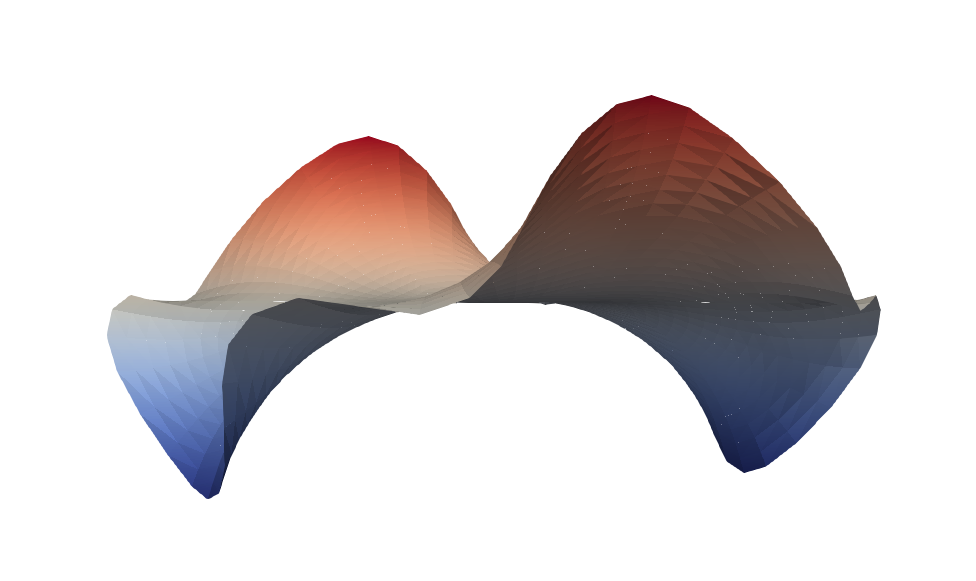}
		\caption{Deformed plate for the disc with oscillation boundary using the initial deformation described in \textbf{Case 1}. Left: output of Algorithm \ref{algo:preprocess} (preprocessing); Middle: output of Algorithm \ref{algo:main_GF} (gradient flow); Right: another view of output of Algorithm \ref{algo:main_GF}. }
		\label{fig:oscillation_boundary}
	\end{center}
\end{figure}

\smallskip\noindent
\textbf{Case 2: no boundary oscillation.} We run Algorithm \ref{algo:preprocess} with the bi-Laplacian problem \eqref{bi-Laplacian} with fictitious force $\widehat \vf=(0,0,1)^T$ and boundary condition $\vvarphi(\vx)=(\vx,0)$ on $\partial\Omega$ (but without $\Phi$). The output of Algorithm \ref{algo:preprocess} is an ellipsoid without oscillatory boundary as in Case 1.

This corresponds to an underlying metric rather different from the target $g$. Algorithm \ref{algo:main_GF} (gradient flow) is unable to improve on the metric defect because it is designed to decrease the bending energy. Therefore, the output of Algorithm \ref{algo:main_GF} is again an ellipsoidal surface totally different from that of Case 1 that is displayed in Figure \ref{fig:oscillation_boundary}. In this case, $D_h = 0.801464$ and $E_h = 0.0377544$ leading to a smaller bending energy but larger metric defect when compared with Case 1.

\begin{figure}[htbp]
	\begin{center}
\begin{tabular}{cc|cc}
		\includegraphics[width=0.2\textwidth]{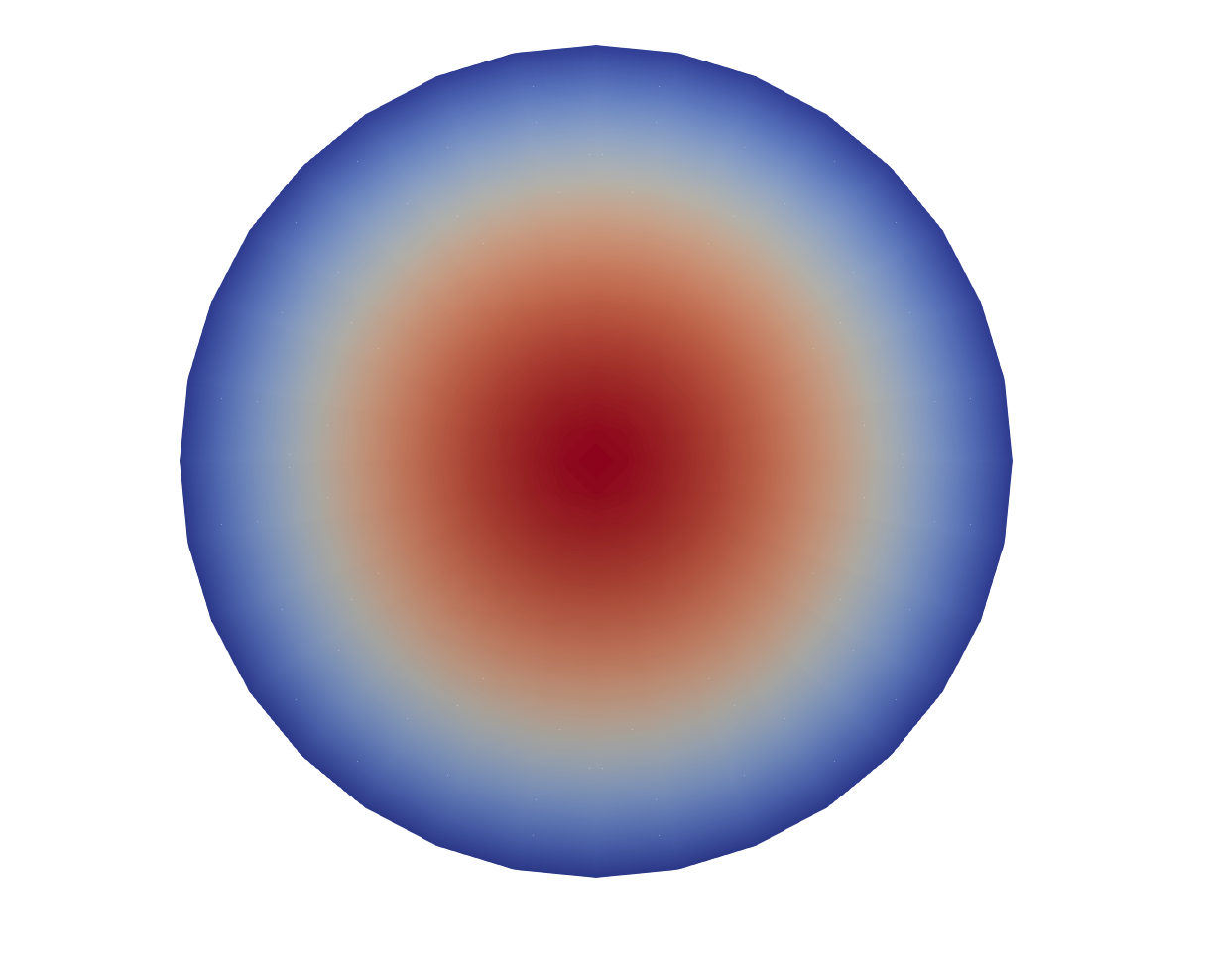}&
		\includegraphics[width=0.2\textwidth, trim=500 300 500 300, clip]{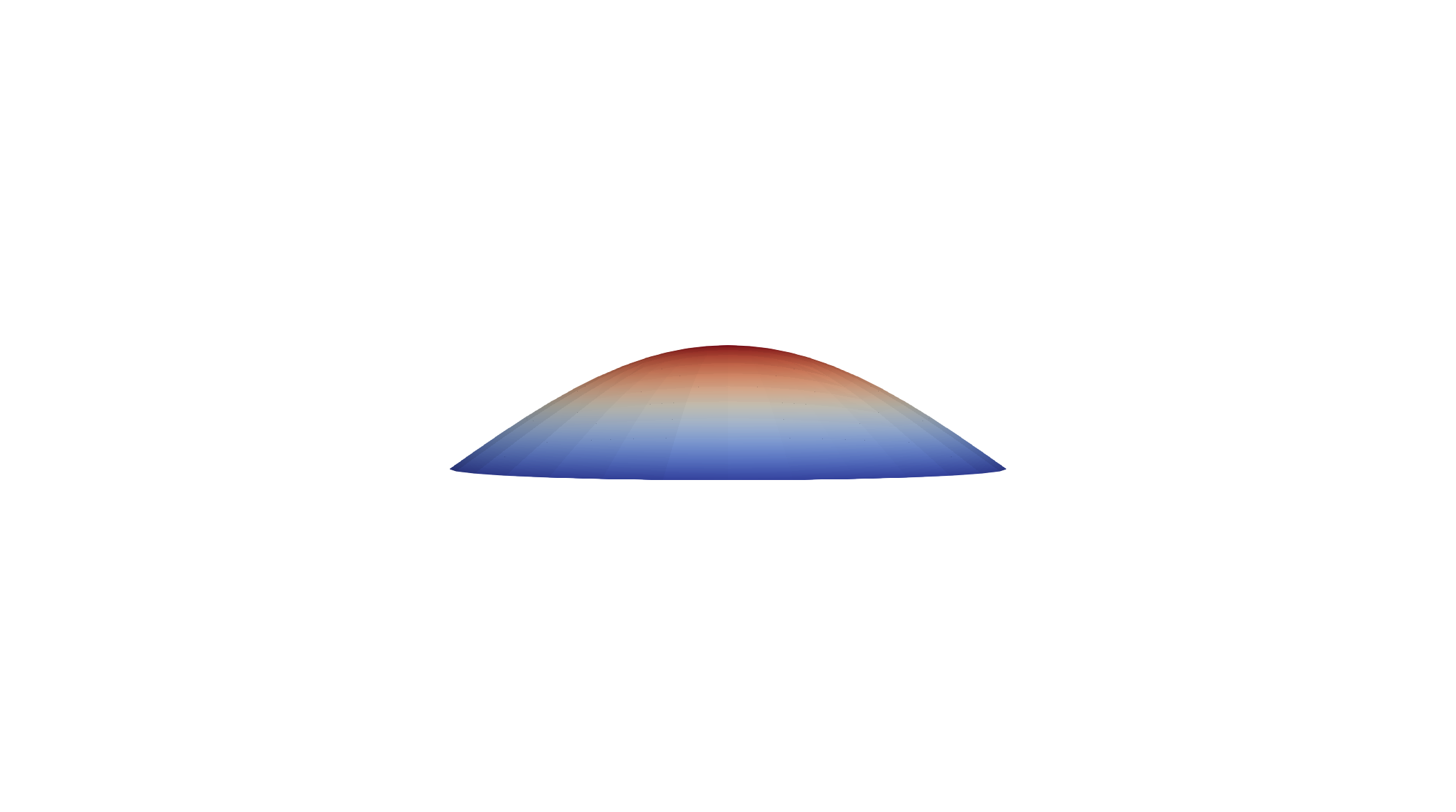}&
		\includegraphics[width=0.23\textwidth]{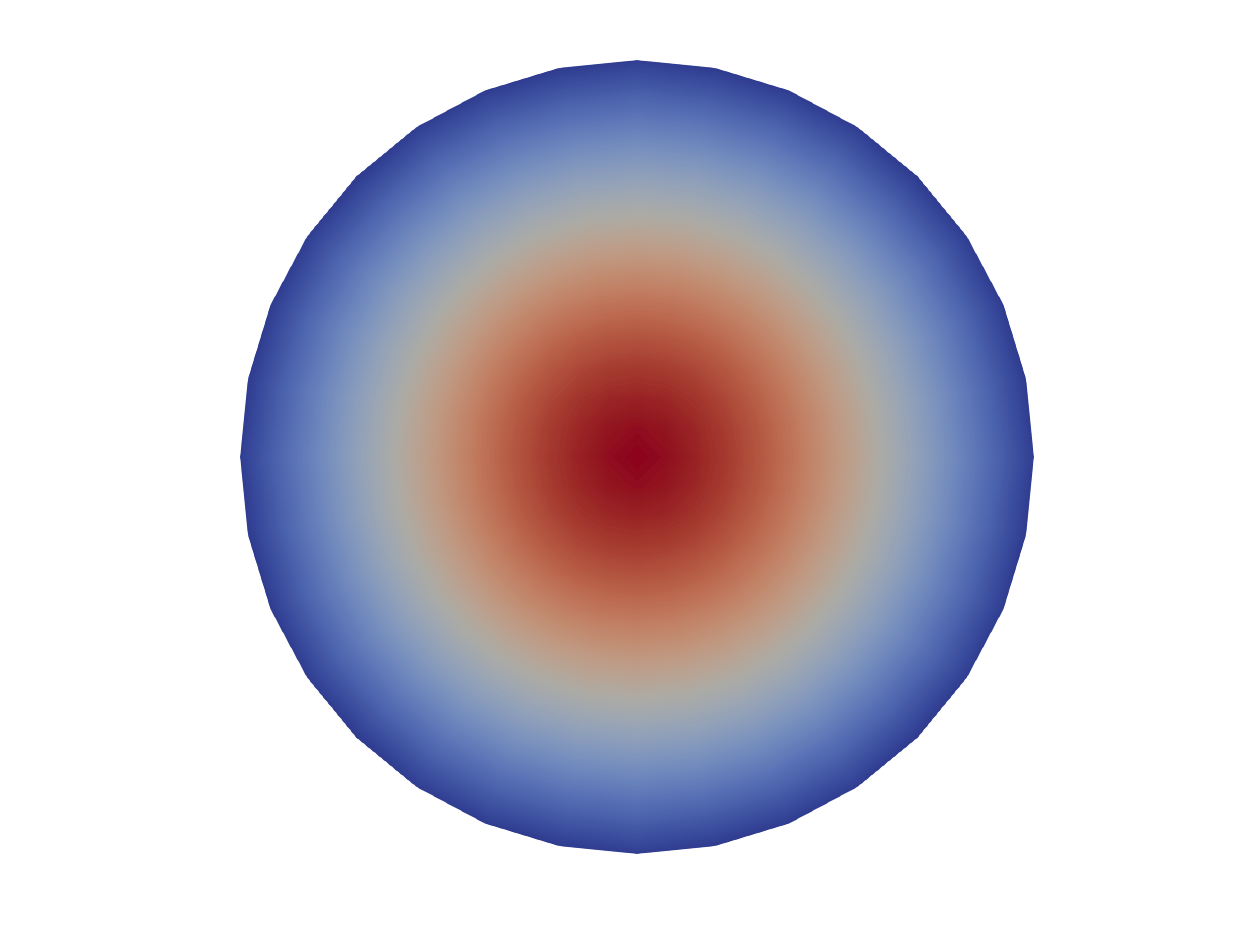}&
		\includegraphics[width=0.2\textwidth,trim=500 300 500 300, clip]{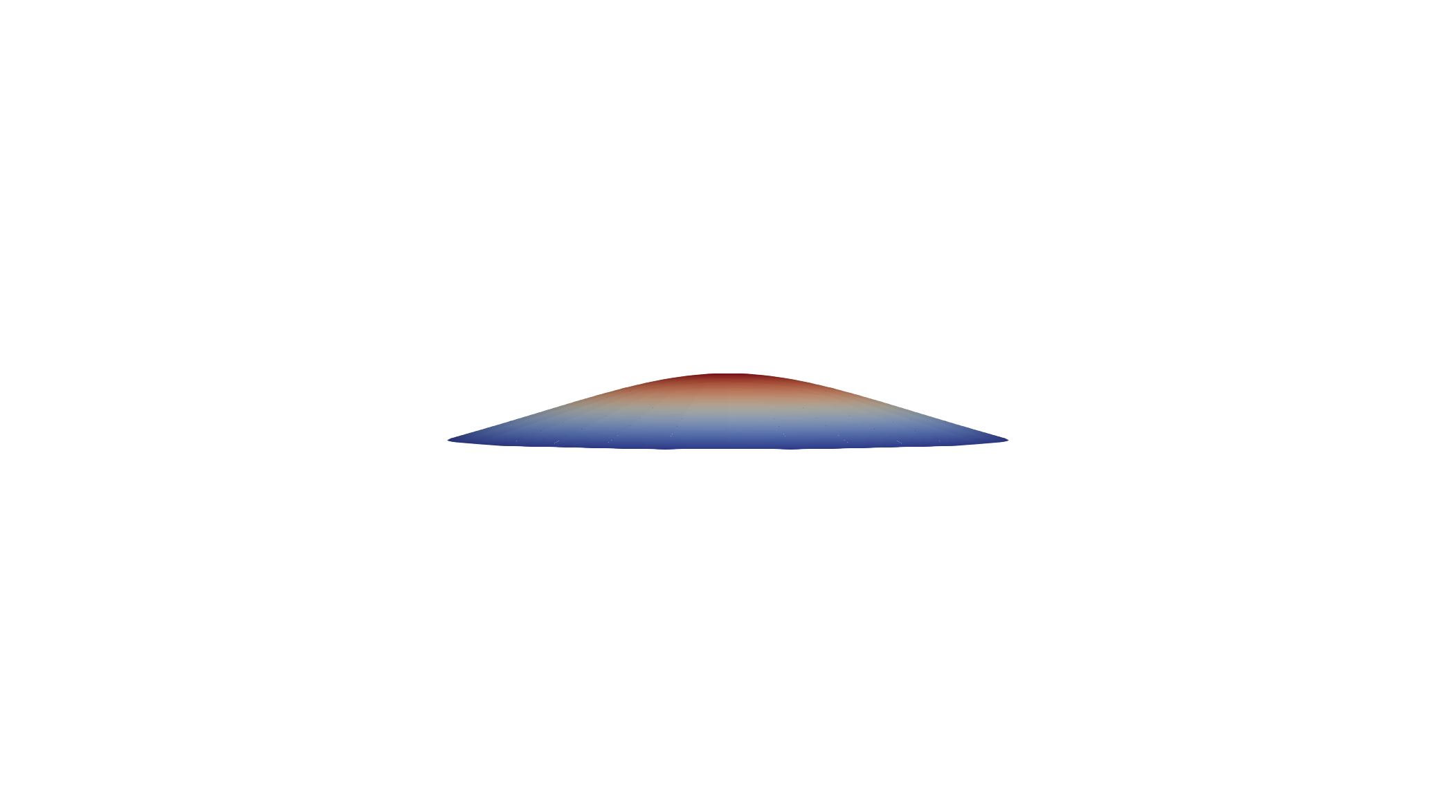}\\
		(a) & (b) & (c) & (d)
\end{tabular}
		\caption{Ellipsoidal-like deformation of a disc without boundary oscillation when using the initial deformation described in \textbf{Case 2}. (a)-(b): output of Algorithm \ref{algo:preprocess} (preprocessing) with maximal third component $y_3$ of the deformation about $7.8\times 10^{-2}$; (c)-(d): output of Algorithm \ref{algo:main_GF} (gradient flow) with maximal $y_3 \approx 4.4\times 10^{-2}$. (a) and (c) are views from the top while (b) and (d) are views from the side where the third component of the deformation is scaled by a factor 10.}
		\label{fig:oscillation_boundary_2}
	\end{center}
\end{figure}

\subsection{Gel discs}\label{S:gel-discs}
Discs made of a NIPA gel with various monomer concentrations can be manufactured in laboratories \cite{sharon2010mechanics,klein2007shaping}. NIPA gels undergo a differential shrinking in warm environments depending on the concentration. Monomer concentrations injected at the center of the disc generate prestrain metrics depending solely on the distance to the center. We thus propose, inspired by \cite[Section 4.2]{sharon2010mechanics}, prestrained metrics $\widetilde{g}(r,\theta)$ in polar coordinates of the form \eqref{E:metric-eta} with
\begin{equation}\label{positive-curvature}
  \eta(r) =
  \begin{cases}
    \frac{1}{\sqrt{K}}\sin(\sqrt{K}r) \quad & K>0,
    \\
    \frac{1}{\sqrt{-K}}\sinh(\sqrt{-K}r) \quad & K<0.
  \end{cases}
\end{equation}
In view of Section \ref{S:admissibility}, these metrics are immersible, namely there
exist compatible deformations $\vy$ such that $\I[\vy]=g$ (isometric immersions). We
now construct computationally isometric embeddings $\vy$
for both $K>0$ (elliptic) and $K<0$ (hyperbolic). It turns out that they
possess a constant Gaussian curvature $\kappa = K$ according to \eqref{E:Gauss-curvature}. 

We let the domain $\Omega$ be the unit disc centered at the origin, do not enforce any boundary conditions and let $\vf = \mathbf{0}$. The partition of $\Omega$ is as in Section \ref{sec:disc} and
\[
\widetilde\tau=0.05,\quad
\widetilde\eps_0=0.1,\quad
\widetilde{tol}=10^{-4}, \quad
tol=10^{-6}.
\]   
{\bf Case $K = 2$ (elliptic):} We use the fictitious force $\widehat \vf=(0,0,1)^T$ in Algorithm \ref{algo:preprocess} (preprocessing) and the pseudo-time step $\tau=0.05$ in Algorithm \ref{algo:main_GF} (gradient flow). We obtain a \emph{spherical-like} final deformation; see Figure \ref{fig:gel-disc-p} and Table \ref{tab:gel_disc_1} for the results.

\begin{figure}[htbp]
	\begin{center}
		\includegraphics[width=4.8cm]{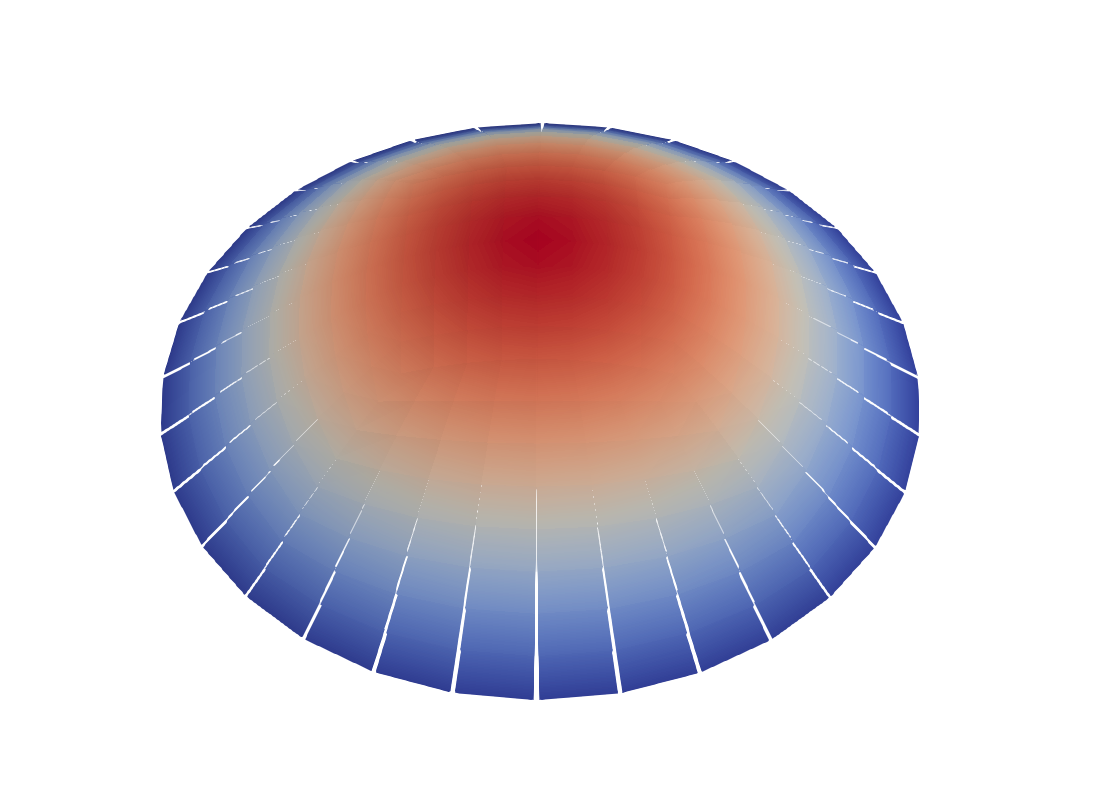}
		\includegraphics[width=4.8cm]{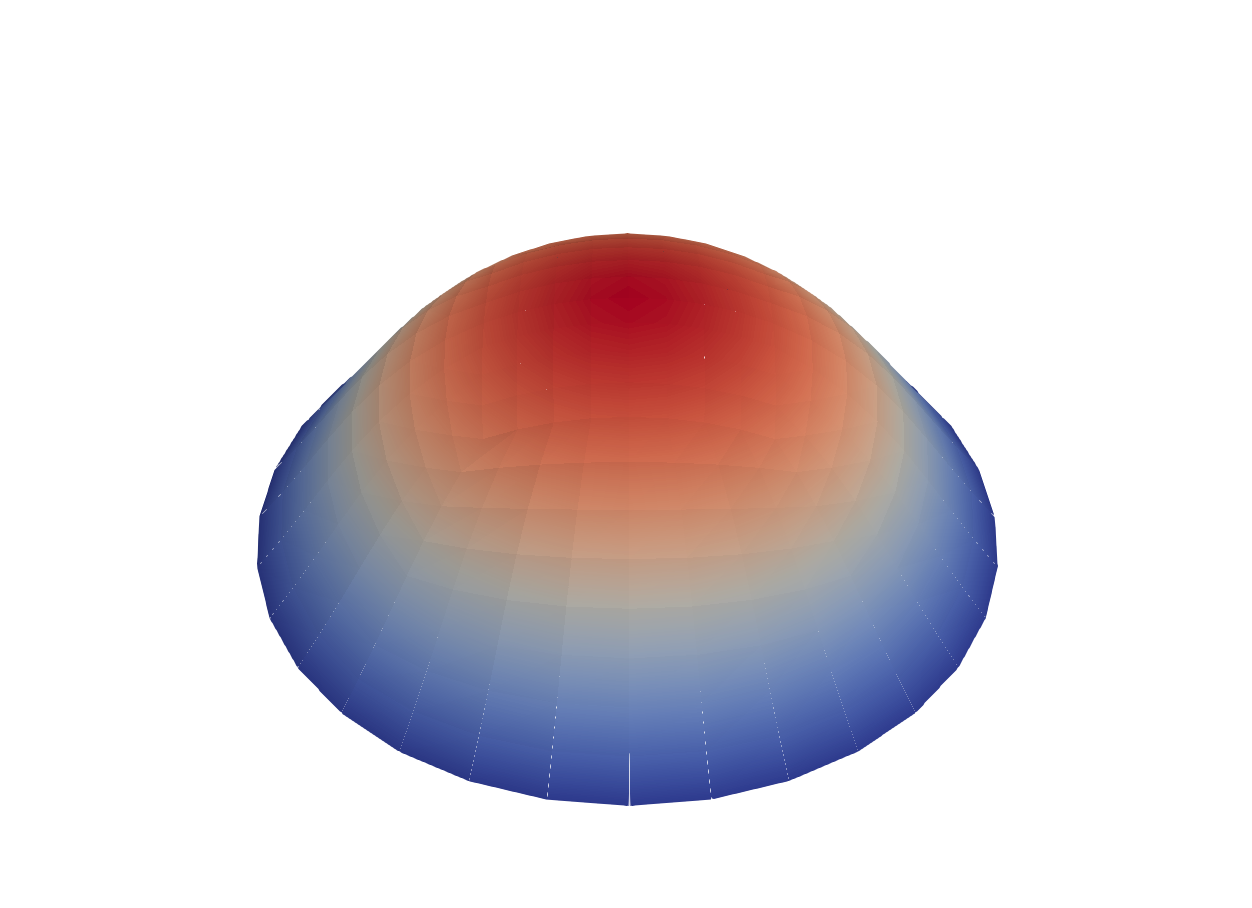}
		\caption{Deformed plate for the disc with constant Gaussian curvature $K=2$ (elliptic).  Outputs of Algorithm \ref{algo:preprocess} (left) and Algorithm \ref{algo:main_GF} (right).} 
		\label{fig:gel-disc-p}
	\end{center}
\end{figure}

\begin{table}[htbp]
	\begin{center}
		\begin{tabular}{|c|c|c|}
			\cline{2-3}
			\multicolumn{1}{c|}{ } & Algorithm \ref{algo:preprocess} & Algorithm \ref{algo:main_GF} \\
			\hline
			$E_h$ & 156.404 & 9.35368\\
			\hline
			$D_h$ & 0.0999494 & 0.188454\\
			\hline
		\end{tabular}
                \vspace{0.3cm}
		\caption{Energy and prestrain defect for disc with constant curvature $K=2$ (elliptic).} \label{tab:gel_disc_1}
	\end{center}	
\end{table}

\noindent
    {\bf Case $K=-2$ (hyperbolic):} We experiment with two different initial deformations for the metric preprocessing of Algorithm \ref{algo:preprocess}:  (i) we take the identity map or (ii) we solve the bi-Laplacian problem \eqref{bi-Laplacian} with a fictitious force $\widehat \vf=(0,0,1)^T$ and boundary condition $\vvarphi(\vx)=(\vx,0)$ on $\partial\Omega$ (but without $\Phi$). Algorithm \ref{algo:preprocess} produces \emph{saddle-like} surfaces in both cases  but with a different number of waves; see Figure \ref{fig:gel-disc-n}. Algorithm \ref{algo:main_GF} uses the pseudo-time steps $\tau=0.00625$ and $\tau=0.0125$ for (i) and (ii), respectively, while the other parameters remain unchanged. Table \ref{tab:gel_disc_2} documents the results.
    
\begin{figure}[htbp]
	\begin{center}
		\includegraphics[width=3cm]{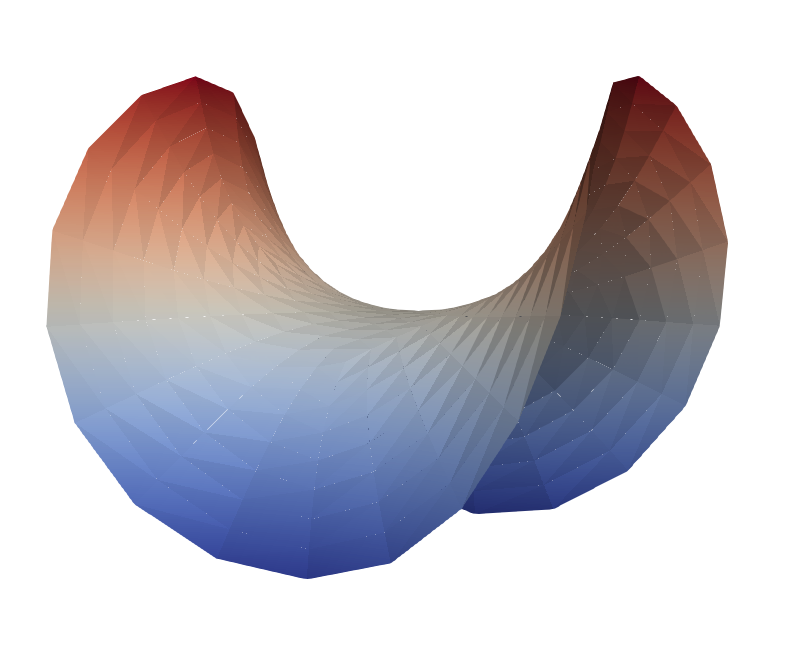}
		\includegraphics[width=3.4cm]{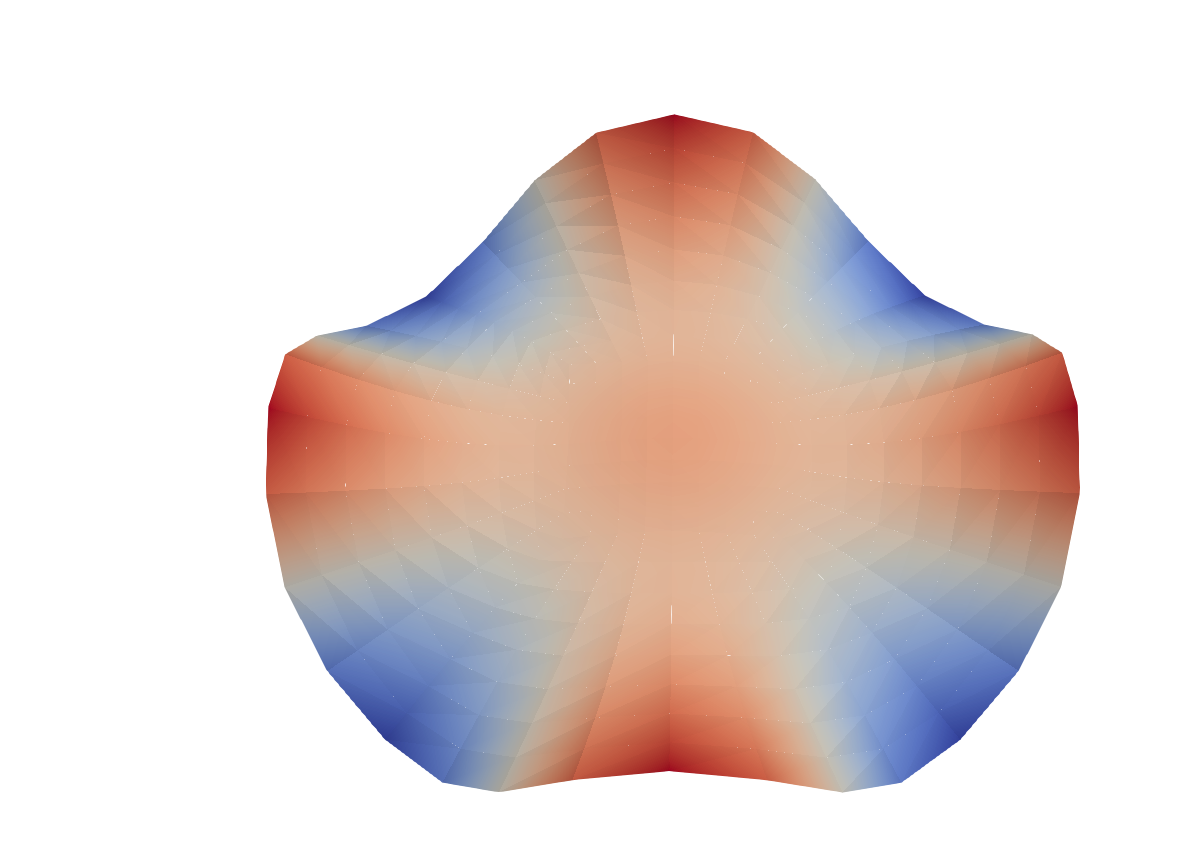}
		\includegraphics[width=3.5cm]{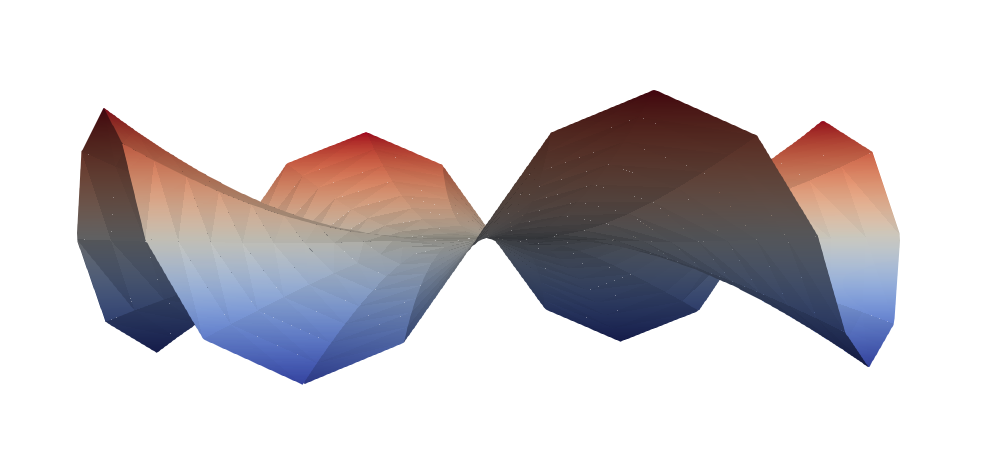}
		\caption{Deformed plate for the disc with constant Gaussian curvature $K=-2$ (hyperbolic). Outputs of Algorithm {\ref{algo:main_GF}} with initialization (i) (left) and initialization (ii) (middle) and another view with initialization (ii) (right). } 
		\label{fig:gel-disc-n}
	\end{center}
\end{figure}

\begin{table}[htbp]
	\begin{center}
		\begin{tabular}{|c|c|c||c|c|}
			\cline{2-5}
			\multicolumn{1}{c|}{ } & \multicolumn{2}{c||}{Initialization (i)} & \multicolumn{2}{c|}{Initialization (ii)}\\
			\cline{2-5}			
			\multicolumn{1}{c|}{ } & Algorithm \ref{algo:preprocess} &  Algorithm \ref{algo:main_GF} & Algorithm \ref{algo:preprocess} &  Algorithm \ref{algo:main_GF} \\
			\hline
			$E_h$  & 699.396 & 6.92318 & 699.399 & 12.0978\\
			\hline
			$D_h$ & 0.0998791 & 0.245552 & 0.0999183 & 0.232627 \\
			\hline
		\end{tabular}
                \vspace{0.3cm}
		\caption{Energy and metric defect for disc with constant Gaussian curvature $K=-2$ (hyperbolic) for two different initial deformations of Algorithm \ref{algo:preprocess}: (i) identity map  and (ii) solution to bi-Laplacian with fictitious force.} \label{tab:gel_disc_2}
	\end{center}	
\end{table}

It is worth mentioning that for the 3d slender model described in \cite{sharon2010mechanics}, it is shown that when $K<0$, the thickness $s$ of the disc influences the number of waves of the minimizing deformation for $K<0$. Our reduced model is asymptotic as $s\to0$ whence it cannot match this feature. However, it reproduces a variety of deformations upon starting Algorithm \ref{algo:preprocess} with suitable initial configurations.

\section{Conclusions}\label{S:conclusions}

In this article, we design and implement a numerical scheme for the simulation of large deformations of prestrained plates. Our contributions are:

\smallskip\noindent
1. {\it Model and asymptotics.} We present a formal asymptotic limit of a 3d hyperelastic energy in the bending regime. The reduced model, rigorously derived in \cite{lewicka2016},
consists of minimizing a nonlinear energy involving the second fundamental form of the deformed plate and the target metric under a nonconvex metric constraint. We show that this energy is equivalent to a simpler quadratic energy that replaces the second fundamental form by the Hessian of the deformation. This form is more amenable to computation and is further discretized.

\smallskip\noindent
2. {\it LDG: discrete Hessian.}
We introduce a local discontinuous Galerkin (LDG) approach for the discretization of the reduced energy, thereby replacing the Hessian by a reconstructed Hessian. The latter consists of three parts: the broken Hessian of the deformation, a lifting of the jumps of the broken gradient of the deformation, and a lifting of the jumps of the deformation. In contrast to interior penalty dG, the penalty parameters must be positive for stability but not necessarily large. The formulation of the discrete energy with LDG is conceptually simpler and it gives a method with reduced CPU time. This does not account for the computation of the discrete Hessian of each basis function which is done once at the beginning.
 
\smallskip\noindent
3. {\it Discrete gradient flow.}
We propose and implement a discrete $H^2$-gradient flow to decrease the discrete energy while keeping the metric defect under control. We emphasize the performance of Algorithm \ref{algo:main_GF} (gradient flow). The construction of suitable initial deformations by Algorithm \ref{algo:preprocess} (initialization) is somewhat ad-hoc leaving room for improvements in future studies.

\smallskip\noindent
4. {\it Simulations.}
We present several numerical experiments to investigate the performance of the proposed LDG approach and the model capabilities. A rich variety of configurations with and without boundary conditions, some of practical value, are accessible by this computational modeling. We also show a superior performance of LDG relative to the interior penalty dG method of \cite{bonito2018} for $g=I_2$.


\section*{Acknowledgment}
Ricardo H. Nochetto and Shuo Yang were partially supported by the NSF Grants DMS-1411808 and DMS-1908267. 

Andrea Bonito and Diane Guignard were partially supported by the NSF Grant DMS-1817691.

\bibliographystyle{amsplain} 
\bibliography{abrevjournal,bibliography}

\end{document}